\documentclass{article}
\usepackage{color}
\usepackage{amsfonts}
\usepackage{tikz}
\usepackage{amssymb}
\usepackage{amsmath}
\usepackage{circuitikz}
\usepackage{amsthm}
\usepackage[mathscr]{euscript}
\usepackage{hyperref}
\usepackage{geometry}

\usepackage{authblk}
\usepackage[T1]{fontenc}
\usepackage[utf8]{inputenc}

\newtheorem{theorem}{Theorem}[section]

\newtheorem{definition}{Definition}[section]
\newtheorem{lemma}{Lemma}[section]
\newtheorem{proposition}{Proposition}[section]
\newtheorem{remark}{Remark}[section]

\newtheorem{assump}{Assumption}[section]
\numberwithin{equation}{section}

\newcommand{\m}{\mathbb}
\newcommand{\ml}{\mathcal}

\newcommand{\p}{\partial}
\newcommand{\pprec}{\prec\!\!\!\prec}

\makeatletter
\newcommand\RSloop{\@ifnextchar\bgroup\RSloopa\RSloopb}
\makeatother
\newcommand\RSloopa[1]{\bgroup\RSloop#1\relax\egroup\RSloop}
\newcommand\RSloopb[1]%
{\ifx\relax#1%
	\else
	\ifcsname RS:#1\endcsname
	\csname RS:#1\endcsname
	\else
	\GenericError{(RS)}{RS Error: operator #1 undefined}{}{}%
	\fi
	\expandafter\RSloop
	\fi
}
\newcommand\X{0}
\newcommand\RS[1]%
{\begin{tikzpicture}
		[every node/.style=
		{circle,draw,fill,minimum size=1.5pt,inner sep=0pt,outer sep=0pt},
		line cap=round
		]
		\coordinate(\X) at (0,0);
		\RSloop{#1}\relax
	\end{tikzpicture}
}
\makeatother
\newcommand\RSdef[1]{\expandafter\def\csname RS:#1\endcsname}
\newlength\RSu
\RSdef{i}{\draw (\X) -- +(90:\RSu) node{};}
\RSdef{l}{\draw (\X) -- +(135:\RSu) node{};}
\RSdef{r}{\draw (\X) -- +(45:\RSu) node{};}
\RSdef{I}{\draw (\X) -- +(90:\RSu) coordinate(\X I);\edef\X{\X I}}
\RSdef{L}{\draw (\X) -- +(135:\RSu) coordinate(\X L);\edef\X{\X L}}
\RSdef{R}{\draw (\X) -- +(45:\RSu) coordinate(\X R);\edef\X{\X R}}
\RSdef{n}{\draw node{};}
\RSdef{c}{\draw ($\circ$);}
\begin{document}
		\author[1,\footnote{Email: zhangshlin9@mail2.sysu.edu.cn}]{\scshape  Shuolin Zhang}
	\author[1,\footnote{Email: mcsyzy@mail.sysu.edu.cn}]{\scshape  Zhaoyang Yin}
	\affil[1]{\slshape School of Science, Sun Yat-sen University, ShenZhen 518107, PR China}
	\title{Singular Stochastic Heat Equations with Fractional Memory}
	\date{}
	\maketitle
      \begin{abstract}
        We extend the paracontrolled distribution approach of Gubinelli, Imkeller, and Perkowski to singular stochastic heat equations with Caputo fractional time derivatives. By establishing Schauder estimates for the time-fractional heat flow and controlling history terms, we address the difficulties arising from temporal memory and the lack of the semigroup property. For the one-dimensional rough Burgers-type equation driven by space-time white noise with spatial dissipation \((-\Delta)^\sigma\), we prove global well-posedness for \(\sigma\in(5/6,1]\) and \(\nu\in(6\sigma/(12\sigma-5),1]\). For the heat equation with nonlinear multiplicative spatial white noise on the two-dimensional torus, we establish local well-posedness of renormalized solutions and an \(L^\infty\) blow-up criterion for every \(\nu\in(0,1]\).
    \end{abstract}
	\tableofcontents
	\section{Introduction}
\par  Heat conduction in classical homogeneous media is commonly described by Fourier's law, which leads to the standard heat equation. This model assumes that the heat flux responds instantaneously to the current temperature gradient and that the underlying transport mechanism is Markovian. However, experimental observations in various complex materials, including porous media, composite materials, biological tissues, and heterogeneous geological structures, have shown that heat transport may exhibit anomalous diffusion and long-term \textbf{``memory effects"}. In such systems, energy carriers can undergo trapping, retention, and multiscale interactions, leading to significant deviations from the classical diffusion law.
\par A natural way to incorporate memory into heat conduction is to replace the instantaneous temporal response by a nonlocal one, so that the present evolution depends on the past history of the system. Such memory effects can be described within the framework of Volterra-type integro-differential equations; see, for instance, \cite{pruss2013evolutionary,zacher2009weak}. More generally, one may introduce a \textbf{memory kernel $k$} and consider an evolution equation of the form
\begin{equation}\label{eq:general-memory}
    \partial_t\bigl(k*(u-u_0)\bigr)
    =
    \Delta u,
\end{equation}
where
\[
    (k*v)(t)
    :=
    \int_0^t k(t-s)v(s)\,ds.
\]
The kernel $k$ determines how the past states of the system contribute to its present evolution and thus characterizes the underlying memory mechanism. Different choices of $k$ lead to different types of temporal memory. In particular, fractional diffusion equations arise as a distinguished example corresponding to power-law memory kernels; see, e.g., \cite{metzler2000random,mainardi2007fundamental,mainardi2022fractional}.

\par In this work, we focus on a particular but fundamental example. For $\nu\in(0,1)$, let
\begin{equation}
    k_\nu(t)
    =
    \frac{t^{-\nu}}{\Gamma(1-\nu)}.
\end{equation}
Then
\begin{equation}
    \partial_t\bigl(k_\nu*(u-u_0)\bigr)
    =
    \partial_t^\nu u,
\end{equation}
where $\partial_t^\nu$ denotes the Caputo fractional derivative of order $\nu$, introduced by Caputo in his study of linear models of dissipation \cite{caputo1966linear,caputo1967linear}. Consequently, \eqref{eq:general-memory} reduces to the time-fractional heat equation
\begin{equation}
    \partial_t^\nu u
    =
    \Delta u.
\end{equation}
Such anomalous diffusion behavior can be interpreted through the continuous-time random walk (CTRW) theory; see, for instance, \cite{metzler2000random,meerschaert2002stochastic}. In particular, heavy-tailed waiting times lead to subdiffusive dynamics characterized by particle trapping and long resting periods. Probabilistically, the resulting time-fractional diffusion can be represented in terms of a Brownian motion time-changed by an inverse stable subordinator \cite{meerschaert2002stochastic,meerschaert2013inverse}. This representation provides a natural interpretation of the memory effect encoded by the Caputo derivative and, correspondingly, of the time-fractional heat kernel as the transition density of a non-Markovian diffusion process.
\subsection{Time-fractional model}
\par Beyond the linear diffusion setting, nonlinear effects can be incorporated 
into time-fractional models, leading to the time-fractional Burgers equation. 
A typical generalized form is
\begin{equation}
    \partial_t^\nu u-\Delta u=u^p\partial_xu,
    \qquad 0<\nu<1,\quad p>0.
\end{equation}
Such equations can be viewed as nonlinear counterparts of the time-fractional 
heat equation, in which the Caputo fractional derivative accounts for temporal 
memory and anomalous diffusion, while the Burgers-type nonlinearity describes 
nonlinear transport. The interplay between these memory effects, diffusion, 
and nonlinear convection provides a natural framework for modeling nonlinear 
anomalous transport processes. Time-fractional Burgers equations and their 
generalized forms have therefore been studied extensively; see, for instance, 
\cite{momani2006non,li2016linear,yokucs2017numerical} and the references therein.
\par In the presence of random perturbations, Hachemi and \O ksendal
\cite{oksendal2013stochastic} considered the time-fractional stochastic heat
equation on $\mathbb{R}^d$:
\begin{equation}
    \partial_t^\nu u
    =\lambda\Delta u+\sigma\xi(t,x),
\end{equation}
where $\xi(t,x)$ denotes space-time white noise satisfying
\begin{equation}
    \mathbb{E}[\xi(t,x)\xi(s,y)]
    =\delta(t-s)\delta(x-y).
\end{equation}
Here, the time-fractional derivative accounts for memory and anomalous
diffusion, while the additive space-time white noise represents random
fluctuations in the surrounding medium. Thus, the equation provides a
stochastic extension of the time-fractional heat equation for diffusion
processes evolving in a random or noisy environment. More generally,
state-dependent random perturbations lead to equations with multiplicative
noise. For example, Tuan \cite{tuan2020asymptotic} studied a time-fractional stochastic
elliptic equation of the form
\begin{equation}
    \partial_t^\nu u
    =\mathcal{L}u+\beta u+\gamma u\,\dot W_t,
\end{equation}
where $\mathcal{L}$ is a second-order elliptic operator and $\dot W_t$ denotes
temporal white noise. In contrast to additive noise, the amplitude of the
multiplicative noise depends on the state $u$, providing a model for random
perturbations whose intensity varies with the evolving system.

Motivated by these developments, in this paper we investigate several classes of stochastic time-fractional heat equations driven by different types of nonlinearities and random noises. Specifically, we consider the following models:
\begin{itemize}
    \item The time-fractional heat equation with Burgers-type nonlinearity:
    \begin{equation}\label{eq:fractional Rough Burgers}
        \partial_t^\nu u+(-\Delta)^\sigma u
        =\ml N(u)\partial_xu+\xi,
    \end{equation}
    where $\sigma\in(5/6,1]$, $\nu\in(6\sigma/(12\sigma-5),1]$, $\xi$ is a space-time white noise, and $\ml N\in C_b^3$.

    \item The time-fractional stochastic heat equation driven by multiplicative noise:
    \begin{equation}
        \partial_t^\nu u-\Delta u
        =F(u)\diamond\xi,
    \end{equation}
    where $\nu\in(0,1]$. Here $\xi$ denotes a spatial white noise, which will bring the singularity Hence $\diamond$ denotes the renormalized product $F\in C_b^3$.
\end{itemize}
\subsection{Singular SPDEs}
When $\nu=1$, \eqref{eq:fractional Rough Burgers} reduces to the rough Burgers-type equation
\begin{equation}\label{eq;Rough Burgers}
\partial_t u-\Delta u
=\ml N(u)\partial_xu+\xi,
\end{equation}
which is a prototypical example of a singular stochastic partial differential equation (SPDE). The main difficulty lies in the fact that the solution is typically too irregular for the nonlinear term $\ml N(u)\partial_xu$ to be interpreted in the classical sense.
\par To illustrate the source of this difficulty, consider first the linear stochastic equation
\begin{equation}
\partial_t \vartheta+(-\Delta)^\beta\vartheta=\xi.
\end{equation}
In the classical case $\nu=1$, the stochastic convolution $\vartheta$ belongs to $C_T\mathcal{C}^{\alpha}$ for $\alpha<\beta-\frac12$, as is well known; see, e.g., \cite{gubinelli2015paracontrolled}. This is also the natural spatial regularity one expects for the solution $u$ of \eqref{eq;Rough Burgers}. Consequently, when $\beta\leq 1$, the derivative $\partial_xu$ is generally too singular to be multiplied with $\ml N(u)$ in the classical sense. More precisely, the product $\ml N(u)\partial_xu$ falls outside the range of well-defined products provided by the standard paraproduct theory. This illustrates the intrinsic singularity of the equation.
\par Another classical example of a singular SPDE is the generalized parabolic Anderson model
\begin{equation}
\partial_tu-\Delta u=F(u)\diamond\xi,
\end{equation}
where $\xi$ is a spatial white noise on $\mathbb{T}^2$ and $\diamond$ denotes a suitable renormalized product. In this case, the singularity arises from the multiplication of the rough solution with the spatial white noise, which is itself a distribution rather than a classical function.
\par Over the past decade, substantial progress has been made in the analysis of such singular equations. One major breakthrough is Hairer's theory of regularity structures, which was inspired in part by the ideas of rough path theory and provides a systematic framework for solving singular SPDEs beyond the scope of classical function spaces; see \cite{hairer2013solving,hairerTheoryRegularityStructures2014}. Around the same period, Gubinelli and his collaborators developed the theory of paracontrolled distributions, combining harmonic analysis with controlled rough paths to construct and analyze singular products. The paracontrolled approach is particularly flexible and has been successfully applied to a broad class of equations with different analytic structures, including both parabolic and dispersive equations; see \cite{gubinelli2015paracontrolled,gubinelli2017kpz,gubinelli2023paracontrolled,catellier2018paracontrolled}.
\subsection{Main difficulties}
\par In the presence of memory, however, the analysis becomes substantially more delicate. Although the general philosophy for treating singularities remains similar to that in the classical parabolic setting, the nonlocal dependence on the past fundamentally changes the underlying evolution. In particular, the classical heat semigroup is replaced by a family of operators which no longer enjoys the semigroup property. Consequently, many arguments relying on the Markovian and local-in-time structure of the classical heat flow are no longer directly available.

For the power-law memory considered in this work, these difficulties are reflected in the following three parts: 
More specifically, the presence of temporal memory gives rise to several analytical difficulties that distinguish the time-fractional setting from the classical parabolic case.

\begin{itemize}
    \item[(1)] \textbf{Regularity estimates for the time-fractional heat flow.}
    Consider the generalized time-space fractional equation in its Duhamel form
    \begin{equation}\label{eq;duhamel}
        u(t)
        =
        U(t)u_0
        +
        \int_0^t V(t-s)f(s)\,\mathrm ds,
    \end{equation}
    where $U(t)$ and $V(t)$ denote the solution operators associated with the initial condition and the forcing term, respectively. Although various analytic properties of these operators have been studied in the literature; see, for instance, \cite{li2022finite,jin2021fractional,liu2018quasi}, the available estimates are not directly suited to the analysis of singular SPDEs. In particular, the paracontrolled approach requires precise Schauder estimates in spatial H\"older--Besov spaces, together with quantitative time-regularity estimates and corresponding bounds for the Duhamel operator. Establishing such estimates for $U(t)$ and $V(t)$ is therefore one of the basic analytic ingredients of our approach.

    \item[(2)] \textbf{Finite smoothing induced by temporal memory.}
    A further difference from the classical parabolic setting is that the time-fractional solution operators possess only finite spatial smoothing at each fixed positive time. When $\nu=1$, the operators reduce, up to the corresponding convention for the forcing operator, to the fractional heat semigroup
    \begin{equation}
        e^{-t(-\Delta)^\sigma},
    \end{equation}
    which is infinitely smoothing for every $t>0$. In contrast, for $0<\nu<1$, the operators $U(t)$ and $V(t)$ are described by Mittag--Leffler functions and exhibit only polynomial decay at high spatial frequencies. Consequently, they provide only a finite gain of spatial regularity, even for fixed $t>0$. This loss of infinite smoothing is a direct manifestation of the nonlocal-in-time dynamics and plays an essential role in determining the admissible regularity regime of the singular equations considered below.
    \item[(3)] \textbf{Loss of commutator smoothing for the Caputo derivative.}
    A more subtle difficulty appears when one attempts to adapt the paracontrolled approach of \cite{gubinelli2015paracontrolled}. In the classical setting, one introduces a modified paraproduct and writes, schematically,
    \begin{equation}
        u
        =
        F(u)\pprec\vartheta+u^\sharp,
    \end{equation}
    where $\pprec$ is designed to improve the commutation between the paraproduct and the parabolic operator. In particular, the construction exploits the regularizing behavior of commutators of the form
    \begin{equation}
        \partial_t(f\pprec g)
        -
        f\pprec\partial_t g.
    \end{equation}

    Such a mechanism cannot be transferred directly to the time-fractional setting. The essential obstruction comes from the memory encoded in the Caputo derivative. More precisely, the Caputo derivative can be decomposed in the form
    \begin{equation}
        \partial_t^\nu f
        =
        {}_{-\infty}\partial_t^\nu f
        -
        \langle K_t^-,f\rangle,
    \end{equation}
    where ${}_{-\infty}\partial_t^\nu$ is the translation-invariant fractional derivative on the whole time line and can be treated as a Fourier multiplier, whereas $\langle K_t^-,f\rangle$ is a boundary contribution arising from the past history. The Fourier-multiplier part retains a useful commutator structure, but the boundary term does not provide the corresponding smoothing effect. As a consequence, the commutator
    \begin{equation}
        \partial_t^\nu(f\pprec g)
        -
        f\pprec\partial_t^\nu g
    \end{equation}
    does not, in general, enjoy the same regularity improvement as in the classical case. This prevents a direct application of the standard modified-paraproduct argument and requires a different treatment of the memory contribution.
    \item[(4)]{ \textbf{Failure of the semigroup property and difficulties in restarting the evolution.}
A further challenge in establishing global well-posedness arises from the temporal memory inherent in the Caputo derivative. In the classical parabolic setting, the fractional heat semigroup
\begin{equation}
    P(t)=e^{-t(-\Delta)^\sigma}
\end{equation}
satisfies
\begin{equation}
    P(t+s)=P(t)P(s).
\end{equation}
Consequently, the evolution can be restarted at any intermediate time using the solution at that time as new initial data. This property allows local well-posedness estimates to be iterated over successive time intervals and plays a fundamental role in standard continuation arguments.

In contrast, for $0<\nu<1$, the time-fractional solution operators $U(t)$ and $V(t)$ do not, in general, satisfy the corresponding semigroup identities. The solution at a given time therefore does not contain all the information needed to determine its subsequent evolution, since the Caputo derivative retains contributions from the entire preceding history. In particular, a solution represented by \eqref{eq;duhamel} cannot generally be restarted at a time $t_0>0$ by simply replacing $u_0$ with $u(t_0)$ in the same Duhamel formula. Instead, the evolution on $[t_0,t_0+T_0]$ contains additional terms depending on the initial data and the forcing accumulated over $[0,t_0]$.

These history-dependent contributions obstruct a direct iteration of the local well-posedness argument and require estimates that remain compatible with the memory structure. To overcome this difficulty, we introduce auxiliary restarted evolutions and compare them with the original solution, while controlling the resulting history terms through suitable Schauder estimates. This provides a mechanism for propagating local bounds across successive time intervals and is an essential ingredient in our global well-posedness analysis. }
\end{itemize}
\subsection{Contributions and Organization of the Paper}
\par The main purpose of this paper is to address the above difficulties and to develop an analytical framework for singular SPDEs in the presence of temporal memory. We first establish a regularity theory for the solution operators associated with the generalized time-space fractional heat equation. This includes Schauder-type estimates and estimates for the corresponding Duhamel operator. Particular attention is paid to the finite smoothing property of the fractional heat flow and its effect on the admissible spatial regularities.

\par Building on this framework, we construct local solutions to the two classes of singular SPDEs introduced above, namely equations with Burgers-type nonlinearities and equations with multiplicative noise. We further establish the convergence of suitable smooth approximations, together with the corresponding renormalization. Our approach is also compatible with the classical parabolic setting. When $\nu=1$, the memory effect disappears, and the framework applies to the corresponding time-space fractional equations, thereby providing a direct connection with the paracontrolled approach developed in \cite{gubinelli2015paracontrolled}.

\par The remainder of the paper is organized as follows. In Section~\ref{sec;basic tools}, we introduce the basic analytic tools and notation used throughout the paper. In Section~\ref{sec;analysis}, we develop the regularity theory for the time-fractional heat flow and establish the Schauder-type and related estimates required in the subsequent analysis. In Section~\ref{sec;Rough Burgers}, we apply the developed framework to the time-fractional rough Burgers equation and establish the existence and uniqueness of a local weak solution. Finally, in Section~\ref{sec;time-fractional}, we study the time-fractional stochastic heat equation with nonlinear multiplicative noise and establish the corresponding local well-posedness result.
\section{Basic tools}\label{sec;basic tools}
\subsection{Time-fractional heat flow}
We first recall the definition of the Caputo fractional derivative.
\begin{definition}[Caputo fractional derivative]
    Let $\nu\in(0,1)$. For a sufficiently regular function $f$, the
    Caputo derivative of order $\nu$ is defined by
    \begin{equation}\label{eq:Caputo-derivative}
        \partial_t^\nu f(t)
        :=
        \frac{1}{\Gamma(1-\nu)}
        \int_0^t
        (t-s)^{-\nu}\partial_s f(s)\,\mathrm{d}s.
    \end{equation}
    where $\Gamma$ is the Gamma function.
\end{definition}

We next introduce the solution operators associated with the
time-fractional heat equation and establish some of their smoothing
properties. Consider
\begin{equation}\label{eq:time-fractional-heat}
    \partial_t^\nu u-\Delta u=f,
    \qquad
    u(0,x)=u_0(x).
\end{equation}
Motivated by Duhamel's formula for the classical heat equation, we
write the solution of \eqref{eq:time-fractional-heat} in the mild form
\begin{equation}
    u(t)
    =
    U(t)u_0
    +
    \int_0^t V(t-s)f(s)\,\mathrm{d}s.
\end{equation}
In the classical case $\nu=1$, both operators reduce to the heat
semigroup,
\[
    U(t)=V(t)=e^{t\Delta},
\]
whose smoothing properties are classical; see, for instance,
\cite{gubinelli2015paracontrolled,bahouri2011fourier}. The purpose of this section is to derive analogous
estimates for the time-fractional case.

To this end, for $\delta>0$ and $\theta\in(\pi/2,\pi)$, we introduce
the Hankel-type contour
\begin{equation}\label{eq:hankel-contour}
    \Gamma_{\theta,\delta}
    :=
    \left\{
        \eta\in\mathbb{C}:
        |\eta|=\delta,\ |\arg\eta|\leq\theta
    \right\}
    \cup
    \left\{
        \eta\in\mathbb{C}:
        \eta=\rho e^{\pm i\theta},\ \rho\geq\delta
    \right\},
\end{equation}
oriented counterclockwise from the lower ray to the upper ray.
Since $\theta<\pi$, for every
$\eta\in\Gamma_{\theta,\delta}$ we may write
\[
    \eta=|\eta|e^{i\varphi},
    \qquad
    \varphi=\arg\eta\in[-\theta,\theta]\subset(-\pi,\pi).
\]
Throughout this section, we therefore understand the fractional power
$\eta^\nu$ in the principal branch, namely,
\[
    \eta^\nu
    =
    |\eta|^\nu e^{i\nu\arg\eta}.
\]
In particular, on the two rays of $\Gamma_{\theta,\delta}$,
\[
    \eta=\rho e^{\pm i\theta}
    \quad\Longrightarrow\quad
    \eta^\nu=\rho^\nu e^{\pm i\nu\theta}.
\]

We impose the following assumption on the choice of the contour.

\begin{assump}\label{assump;contour}
    Let $\delta>0$. The angle $\theta$ is chosen such that
    \begin{itemize}
        \item[(i)] $\theta\in(\pi/2,\pi)$ if $\nu=1$;
        \item[(ii)] $\theta\in(\pi/2,\min(\pi,\pi/(2\nu)))$ if $\nu\in(0,1)$.
    \end{itemize}
\end{assump}

The restriction on $\theta$ ensures that the contour lies in a sector
on which the relevant resolvent is well defined and the contour
integrals below have the required decay. By the standard Laplace
transform and contour representation for fractional evolution
equations (see, e.g., \cite{jin2021fractional}), we obtain the following representation.

\begin{lemma}\label{le;duhamel}
    Let $u$ be a solution to the time-space fractional heat equation
    \begin{equation}
        \partial_t^\nu u+(-\Delta)^{\sigma}u=f,
        \qquad
        u(0,x)=u_0(x).
    \end{equation}
    Then $u$ admits the mild representation
    \begin{equation}\label{eq:duhamel}
        u(t)
        =
        U(t)u_0
        +
        \int_0^t V(t-s)f(s)\,\mathrm{d}s,
    \end{equation}
    where $U(0)=I$, and for any $t>0$,
    \begin{align}
        U(t)g
        &=
        \frac{1}{2\pi\mathrm{i}}
        \int_{\Gamma_{\theta,\delta}}
        e^{\eta t}\eta^{\nu-1}
        (\eta^\nu+(-\Delta)^\sigma)^{-1}g
        \,\mathrm{d}\eta,
        \label{eq:U-contour}\\
        V(t)g
        &=
        \frac{1}{2\pi\mathrm{i}}
        \int_{\Gamma_{\theta,\delta}}
        e^{\eta t}
        (\eta^\nu+(-\Delta)^\sigma)^{-1}g
        \,\mathrm{d}\eta.
        \label{eq:V-contour}
    \end{align}
    Here $\Gamma_{\theta,\delta}$ is the contour defined in
    \eqref{eq:hankel-contour}, with $\theta$ satisfying
    Assumption \ref{assump;contour}.
\end{lemma}

Let us briefly verify that the above representation is consistent with
the classical time derivative. When $\nu=1$, the two operators coincide,
and
\begin{align*}
    U(t)g=V(t)g
    =
    \frac{1}{2\pi\mathrm{i}}
    \int_{\Gamma_{\theta,\delta}}
    e^{\eta t}
    (\eta+(-\Delta)^\sigma)^{-1}g
    \,\mathrm{d}\eta.
\end{align*}
Indeed, for $\lambda\geq0$,
\begin{equation}\label{eq:classical-contour}
    \frac{1}{2\pi\mathrm{i}}
    \int_{\Gamma_{\theta,\delta}}
    \frac{e^\eta}{\eta+\lambda}\,\mathrm{d}\eta
    =
    e^{-\lambda}.
\end{equation}
Consequently, $U(t)$ and $V(t)$ are Fourier multipliers with symbols
\begin{align*}
    e^{-t|\xi|^{2\sigma}},
\end{align*}
and hence
\begin{equation}
    U(t)=V(t)=e^{-t(-\Delta)^\sigma}.
\end{equation}
Therefore, when $\nu=1$, \eqref{eq:duhamel} reduces to the usual
Duhamel formula for the fractional heat equation,
\begin{equation}
    u(t)
    =
    e^{-t(-\Delta)^\sigma}u_0
    +
    \int_0^t
    e^{-(t-s)(-\Delta)^\sigma}f(s)
    \,\mathrm{d}s.
\end{equation}
In particular, when $\nu=\sigma=1$, this further reduces to the
classical heat equation and
\begin{equation}
    u(t)
    =
    e^{t\Delta}u_0
    +
    \int_0^t e^{(t-s)\Delta}f(s)\,\mathrm{d}s.
\end{equation}

\subsection{Distribution theory}
We briefly recall some basic facts from Littlewood--Paley theory and
Bony's paraproduct decomposition. Let $\mathcal{S}(\mathbb{R})$ and
$\mathcal{S}'(\mathbb{R})$ denote the spaces of Schwartz functions and
tempered distributions, respectively. Let $(\Delta_j)_{j\geq -1}$ be
the nonhomogeneous dyadic blocks associated with a smooth dyadic
partition of unity in the Fourier space. More precisely, we choose
smooth nonnegative radial functions $\chi$ and $\rho$ such that
\[
    \chi(\xi)+\sum_{j\geq 0}\rho(2^{-j}\xi)=1,
    \qquad \xi\in\mathbb{R},
\]
where $\chi$ is supported in a neighborhood of the origin and $\rho$
is supported in an annulus. We then define
\[
    \Delta_{-1}u=\mathcal{F}^{-1}
    \bigl(\chi\,\mathcal{F}u\bigr),
    \qquad
    \Delta_j u
    =
    \mathcal{F}^{-1}
    \bigl(\rho(2^{-j}\cdot)\mathcal{F}u\bigr),
    \quad j\geq 0.
\]
With this notation, every $u\in\mathcal{S}'(\mathbb{R})$ admits the
Littlewood--Paley decomposition
\begin{equation}\label{eq:LP-decomposition}
    u=\sum_{j\geq -1}\Delta_j u
\end{equation}
in $\mathcal{S}'(\mathbb{R})$.

For $j\geq -1$, we also introduce the low-frequency cut-off operator
\[
    S_j u:=\sum_{k\leq j-1}\Delta_k u,
\]
with the convention that $\Delta_k=0$ for $k<-1$. Thus, $S_j u$
contains frequencies of size strictly smaller than $2^j$, whereas
$\Delta_j u$ localizes $u$ at frequencies of order $2^j$.

For $s\in\mathbb{R}$ and $p,r\in[1,\infty]$, the nonhomogeneous Besov
space $B_{p,r}^s(\mathbb{R})$ is defined by
\[
    B_{p,r}^s(\mathbb{R})
    :=
    \left\{
        f\in\mathcal{S}'(\mathbb{R})
        :
        \|f\|_{B_{p,r}^s}<\infty
    \right\},
\]
where
\begin{equation}\label{eq:Besov-norm}
    \|f\|_{B_{p,r}^s}
    :=
    \left(
        \sum_{j\geq -1}
        2^{jsr}\|\Delta_j f\|_{L^p}^r
    \right)^{1/r}
\end{equation}
when $r<\infty$, while
\[
    \|f\|_{B_{p,\infty}^s}
    :=
    \sup_{j\geq -1}
    2^{js}\|\Delta_j f\|_{L^p}.
\]
In particular, throughout the paper we use the notation
\[
    \mathcal{C}^s:=B_{\infty,\infty}^s.
\]
For noninteger $s>0$, this space agrees, up to equivalence of norms,
with the classical H\"older space of regularity $s$.

We next recall Bony's paraproduct decomposition. For sufficiently
regular functions $f$ and $g$, their product can be decomposed
according to the interactions between different frequency scales.
Define the paraproduct
\begin{equation}\label{eq:paraproduct}
    f\prec g
    :=
    \sum_{j\geq -1}S_{j-1}f\,\Delta_j g
    =
    \sum_{j\geq -1}
    \sum_{i<j-1}
    \Delta_i f\,\Delta_j g,
\end{equation}
and set
\[
    f\succ g:=g\prec f.
\]
The resonant product is defined by
\begin{equation}\label{eq:resonant-product}
    f\circ g
    :=
    \sum_{\substack{i,j\geq -1\\ |i-j|\leq 1}}
    \Delta_i f\,\Delta_j g.
\end{equation}
Accordingly, one has Bony's decomposition
\begin{equation}\label{eq:Bony-decomposition}
    fg
    =
    f\prec g+f\circ g+f\succ g.
\end{equation}
Here, $f\prec g$ describes the interaction between the low frequencies
of $f$ and the high frequencies of $g$, while $f\succ g$ describes the
opposite interaction. The resonant term $f\circ g$ contains the
interactions between comparable frequency scales.

This decomposition is particularly useful when $f$ and $g$ have low
regularity. Indeed, the paraproduct terms are well defined under rather
weak assumptions, whereas the resonant product requires a suitable
compatibility between the regularities of $f$ and $g$. More precisely,
\begin{lemma}[\cite{bahouri2011fourier}]
    \label{le;est;paraproduct}
    Let $\alpha,\beta\in\mathbb{R}$. Then the following estimates hold:
    \begin{itemize}
        \item[1.] For any $f\in L^\infty$ and
        $g\in\mathcal{C}^{\beta}$,
        \[
            \|f\prec g\|_{\mathcal{C}^{\beta}}
            \leq
            C\|f\|_{L^\infty}
            \|g\|_{\mathcal{C}^{\beta}}.
        \]

        \item[2.] If $\alpha<0$, then for any
        $f\in\mathcal{C}^{\alpha}$ and
        $g\in\mathcal{C}^{\beta}$,
        \[
            \|f\prec g\|_{\mathcal{C}^{\alpha+\beta}}
            \leq
            C\|f\|_{\mathcal{C}^{\alpha}}
            \|g\|_{\mathcal{C}^{\beta}}.
        \]

        \item[3.] If $\alpha+\beta>0$, then for any
        $f\in\mathcal{C}^{\alpha}$ and
        $g\in\mathcal{C}^{\beta}$, the resonant product
        $f\circ g$ is well defined and satisfies
        \[
            \|f\circ g\|_{\mathcal{C}^{\alpha+\beta}}
            \leq
            C\|f\|_{\mathcal{C}^{\alpha}}
            \|g\|_{\mathcal{C}^{\beta}}.
        \]
    \end{itemize}
    Here, the constant $C>0$ depends only on the regularity
    exponents and the choice of the Littlewood--Paley decomposition.
    By symmetry, the corresponding estimates for $f\succ g$ follow
    from the identity $f\succ g=g\prec f$.
\end{lemma}

\begin{lemma}[\cite{gubinelli2015paracontrolled}]\label{le;para-differential calculus}
     Let $\alpha \in (0,1)$, $\beta \in (0,\alpha)$, $\gamma < 0$ be such that $\alpha +\beta +\gamma > 0$ but $\alpha +\gamma < 0$. Let $F\in C_{b}^{1 + \beta /\alpha}$. Then there exists a locally bounded map $\Pi_{F}\colon \mathcal{C}^{\alpha}\times \mathcal{C}^{\gamma}\to \mathcal{C}^{\alpha +\beta +\gamma}$ such that
$$
F(f)\circ g = F^{\prime}(f)(f\circ g) + \Pi_{F}(f,g)
$$
for all $f\in \mathcal{C}^{\alpha}$ and all smooth $g$. More precisely, we have
$$
\| \Pi_{F}(f,g)\|_{\alpha +\beta +\gamma}\lesssim \| F\|_{C_{b}^{1 + \beta /\alpha}}(1 + \| f\|_{a}^{1 + \beta /\alpha})\| g\|_{\gamma}.
$$
If $F\in C_{b}^{2 + \beta /\alpha}$, then $\Pi_{F}$ is locally Lipschitz continuous:
$$
\| \Pi_{F}(f,g) - \Pi_{F}(u,v)\|_{\alpha +\beta +\gamma}\lesssim \| F\|_{C_{b}^{2 + \beta /\alpha}}(1 + (\| f\|_{\alpha} + \| u\|_{\alpha})^{1 + \beta /\alpha} + \| v\|_{\gamma})(\| f - u\|_{\alpha} + \| g - v\|_{\gamma}). 
$$

\end{lemma}
\begin{lemma}[\cite{gubinelli2015paracontrolled,bahouri2011fourier}]\label{le;paralinearization}
    Let $\alpha \in (0,1)$, $\beta \in (0,\alpha]$, and let $F \in C_{b}^{1 + \beta /\alpha}$. There exists a locally bounded map $R_{F}: \mathcal{C}^{\alpha} \to \mathcal{C}^{\alpha +\beta}$ such that
$$
F(f) = F^{\prime}(f)\prec f + R_{F}(f)
$$
for all $f \in \mathcal{C}^{\alpha}$. More precisely, we have
$$
\| R_{F}(f)\|_{\alpha +\beta}\lesssim \| F\|_{C_{b}^{1 + \beta /\alpha}}(1 + \| f\|_{a}^{1 + \beta /\alpha}).
$$
If $F \in C_{b}^{2 + \beta /\alpha}$, then $R_{F}$ is locally Lipschitz continuous:
$$
\| R_{F}(f) - R_{F}(g)\|_{\alpha +\beta}\lesssim \| F\|_{C_{b}^{2 + \beta /\alpha}}(1 + \| f\|_{\alpha} + \| g\|_{\alpha})^{1 + \beta /\alpha}\| f - g\|_{\alpha}.
$$

\end{lemma}
\begin{lemma}[\cite{gubinelli2015paracontrolled}]\label{le;para-controlled}
     Assume that $\alpha \in (0,1)$ and $\beta ,\gamma \in \mathbb{R}$ are such that $\alpha +\beta +\gamma >0$ and $\beta +\gamma < 0$. Then there exists a bounded trilinear operator $C\in \mathcal{L}^{3}(\mathcal{C}^{\alpha}\times \mathcal{C}^{\beta}\times \mathcal{C}^{\alpha},\mathcal{C}^{\alpha +\beta +\gamma})$ such that
$$
C(f,g,h) = ((f\prec g)\circ h) - f(g\circ h)
$$
whenever $f,g,h\in \mathcal{S}$.
\end{lemma}

\section{Analysis of time-fractional heat flow}\label{sec;analysis}
In this section, we establish several estimates for the time-fractional heat flows \(U(t)\) and \(V(t)\) introduced in Lemma \ref{le;duhamel}. As a simple motivating example, consider first the classical case \(\nu=1\). Since \(-D^2=\Delta\), we have
\begin{align*}
     U(t)=V(t)=e^{t\Delta}
    =\psi(t^{1/2}D),
\end{align*}
where
$$
    \psi(x)=e^{-|x|^{2\sigma}}.
$$
More precisely, by the contour integral representation, \(\psi\) can be written as
$$
    \psi(x)
    =
    \frac{1}{2\pi i}
    \int_{\Gamma_{\theta,\delta}}
    \frac{e^\eta}{\eta+|x|^{2\sigma}}\,\mathrm d\eta,
$$
where the contour \(C=\Gamma_{\theta,\delta}\) is given by \eqref{eq:hankel-contour} with $\delta>0$, $\theta\in(\pi/2,\pi)$.
This representation also makes the rapid decay of the multiplier transparent. Indeed, since \(e^\eta\eta^n\) is entire for $n\in \m N$, Cauchy's theorem, applied after closing the truncated contour in the left half-plane, together with the exponential decay of \(e^{\operatorname{Re}\eta}\), yields
$$
    \int_{\Gamma_{\theta,\delta}}e^\eta\eta^n\,\mathrm d\eta=0.
$$
Consequently, for every \(N\in\mathbb N\) and \(\lambda>0\), a direct algebraic computation gives
\begin{align*}
\lambda^N
\int_C e^\eta\frac{1}{\eta+\lambda}\mathrm d\eta
&=
\sum_{n=0}^{N-1}
(-1)^n\lambda^{N-1-n}
\underbrace{\int_C e^\eta\eta^n\mathrm d\eta}_{=0}
+
(-1)^N
\int_C e^\eta
\frac{\eta^N}{\eta+\lambda}\mathrm d\eta .
\end{align*}
The last integral is uniformly bounded in \(\lambda>0\), owing to the exponential decay of \(e^{\operatorname{Re}\eta}\) along the two rays of the contour, together with the fact that the denominator \(\eta+\lambda\) remains uniformly separated from zero on \(C\). This yields the expected rapid decay in \(\lambda\), and hence the corresponding smoothing estimates for the classical heat flow.
\par We now turn to the genuinely fractional case \(\nu\in(0,1)\).
Let $\eta=re^{i\zeta}$. Under Assumption \ref{assump;contour}, we have $\cos(\nu\zeta)>\cos(\nu\theta)>-1$. Therefore, for any $\epsilon\in[0,1]$ and $y>0$, we have
\begin{align*}
| \frac{1}{\eta^\nu +y}|=\frac{1}{\sqrt{r^{2\nu}+2r^\nu y\cos(\nu\zeta)+y^2}}\lesssim_\theta \frac{1}{\sqrt{r^{2\nu}+y^2}}\lesssim \frac{1}{r^{\nu\epsilon}y^{1-\epsilon}}.
\end{align*}

Similarly, for $k\geq1$, we also have 
\begin{lemma}\label{le;bound integral}
    Let $\nu\in(0,1]$, $\theta\in(\pi/2,\pi)$, and $\delta>0$.
    Then, for any $k\in\m N$, $\alpha\in[0,k]$, $\beta\in\m R$,
    and $\lambda>0$, we have
    \begin{align*}
        \left|
        \int_{\Gamma_{\theta,\delta}}
        e^\eta
        \frac{\lambda^\alpha\eta^\beta}
        {(\eta^\nu+\lambda)^k}
        \mathrm d\eta
        \right|
        \lesssim 1,
    \end{align*}
    where the implicit constant is independent of $\lambda$.

    \begin{proof}
        We decompose $\Gamma_{\theta,\delta}$ into the two rays
        \begin{align*}
            \Gamma_{\theta,\delta}^1
            =
            \{re^{\pm i\theta}:r\in[\delta,\infty)\}
        \end{align*}
        and the circular arc
        \begin{align*}
            \Gamma_{\theta,\delta}^2
            =
            \{\delta e^{i\zeta}:\zeta\in[-\theta,\theta]\}.
        \end{align*}
        Since $\nu\in(0,1]$ and $\theta<\pi$, we have
        $\nu\theta<\pi$. Hence, for $\eta=re^{i\zeta}$ with
        $|\zeta|\leq\theta$,
        \begin{align*}
            |\eta^\nu+\lambda|^2
            &=
            r^{2\nu}
            +
            2r^\nu\lambda\cos(\nu\zeta)
            +
            \lambda^2\\
            &\gtrsim
            r^{2\nu}+\lambda^2.
        \end{align*}
        Therefore, for every $\epsilon\in[0,1]$,
        \begin{align}\label{eq;denominator;bound}
            |\eta^\nu+\lambda|
            \gtrsim
            r^{\nu\epsilon}\lambda^{1-\epsilon}.
        \end{align}

        Suppose first that $k\geq1$. Choose
        \begin{align*}
            \epsilon
            =
            1-\frac{\alpha}{k}
            \in[0,1].
        \end{align*}
        By \eqref{eq;denominator;bound},
        \begin{align*}
            \frac{\lambda^\alpha|\eta|^\beta}
            {|\eta^\nu+\lambda|^k}
            &\lesssim
            \lambda^{\alpha-k(1-\epsilon)}
            r^{\beta-k\nu\epsilon}\\
            &=
            r^{\beta-k\nu\epsilon}.
        \end{align*}

        On $\Gamma_{\theta,\delta}^1$, we have
        $\Re(\eta)=r\cos\theta$ with $\cos\theta<0$. Hence
        \begin{align*}
            \left|
            \int_{\Gamma_{\theta,\delta}^1}
            e^\eta
            \frac{\lambda^\alpha\eta^\beta}
            {(\eta^\nu+\lambda)^k}
            \mathrm d\eta
            \right|
            &\lesssim
            \int_\delta^\infty
            e^{r\cos\theta}
            r^{\beta-k\nu\epsilon}
            \mathrm dr\\
            &\lesssim 1.
        \end{align*}

        On $\Gamma_{\theta,\delta}^2$, since $|\eta|=\delta$,
        we similarly obtain
        \begin{align*}
            \left|
            \int_{\Gamma_{\theta,\delta}^2}
            e^\eta
            \frac{\lambda^\alpha\eta^\beta}
            {(\eta^\nu+\lambda)^k}
            \mathrm d\eta
            \right|
            &\lesssim
            \delta^{\beta-k\nu\epsilon+1}
            \int_{-\theta}^{\theta}
            e^{\delta\cos\zeta}
            \mathrm d\zeta\\
            &\lesssim 1.
        \end{align*}
        Combining the estimates on
        $\Gamma_{\theta,\delta}^1$ and
        $\Gamma_{\theta,\delta}^2$ proves the assertion for
        $k\geq1$.

        If $k=0$, then necessarily $\alpha=0$, and
        \begin{align*}
            \left|
            \int_{\Gamma_{\theta,\delta}}
            e^\eta\eta^\beta\mathrm d\eta
            \right|
            \lesssim 1
        \end{align*}
        follows directly from the exponential decay on the rays
        and the fact that $|\eta|=\delta$ on the circular arc.
    \end{proof}
\end{lemma}
Using above properties of Integral, we can calculate the decay rate
\begin{lemma}\label{le;decay;varphi;phi}
    Let $\nu\in(0,1)$, $\sigma\in(0,1]$, $m\in\m N$, $\delta>0$, $\theta\in(\pi/2,\pi)$ and define, for $x\in\m R^d$,
    \begin{align*}
        \varphi_{\theta,\delta,m}(x)
        &=
        \frac{1}{2\pi i}
        \int_{\Gamma_{\theta,\delta}}
        e^{\eta}\eta^m
        \frac{\eta^{\nu-1}}{\eta^\nu+|x|^{2\sigma}}
        \mathrm d\eta,\\
        \phi_{\theta,\delta,m}(x)
        &=
        \frac{1}{2\pi i}
        \int_{\Gamma_{\theta,\delta}}
        e^{\eta}\eta^m
        \frac{1}{\eta^\nu+|x|^{2\sigma}}
        \mathrm d\eta,
    \end{align*}
    where $\Gamma_{\theta,\delta}$ is defined in
    \eqref{eq:hankel-contour}. Then
    $\varphi_{\theta,\delta,m}$ and $\phi_{\theta,\delta,m}$ are
    smooth functions on $\m R^d\backslash\{0\}$. Moreover, for every $k\in\m N$ and $|x|\geq1$,
    \begin{align}
        |\Delta^k\varphi_{\theta,\delta,m}(x)|
        &\lesssim
        (1+|x|^{2\sigma})^{-k/\sigma-1},
        \label{eq;decay;varphi}\\
        |\Delta^k\phi_{\theta,\delta,m}(x)|
        &\lesssim
        (1+|x|^{2\sigma})^{-k/\sigma-2}.
        \label{eq;decay;phi}
    \end{align}
    The implicit constants may depend on
    $\nu,\sigma,\theta,\delta,m,k$ and $d$, but not on $x$.

    If $\nu=1$, then for every $k,M\in\m N$ and $|x|\geq1$,
    \begin{align}
        |\Delta^k\varphi_{\theta,\delta,m}(x)|
        +
        |\Delta^k\phi_{\theta,\delta,m}(x)|
        \lesssim
        (1+|x|^{2\sigma})^{-M}.
        \label{eq;decay;nu;one}
    \end{align}
\end{lemma}

\begin{proof}
    Set $\lambda=|x|^{2\sigma}$ for $x\neq0$ and define
    \begin{align*}
        \widetilde\varphi_m(\lambda)
        &=
        \frac{1}{2\pi i}
        \int_{\Gamma_{\theta,\delta}}
        e^\eta
        \frac{\eta^{m+\nu-1}}{\eta^\nu+\lambda}
        \mathrm d\eta,\\
        \widetilde\phi_m(\lambda)
        &=
        \frac{1}{2\pi i}
        \int_{\Gamma_{\theta,\delta}}
        e^\eta
        \frac{\eta^m}{\eta^\nu+\lambda}
        \mathrm d\eta.
    \end{align*}
    Then
    \begin{align*}
        \varphi_{\theta,\delta,m}(x)
        =
        \widetilde\varphi_m(|x|^{2\sigma}),
        \qquad
        \phi_{\theta,\delta,m}(x)
        =
        \widetilde\phi_m(|x|^{2\sigma}).
    \end{align*}
    By differentiating under the integral sign,
    $\widetilde\varphi_m$ and $\widetilde\phi_m$ are smooth on
    $(0,\infty)$. Hence $\varphi_{\theta,\delta,m}$ and
    $\phi_{\theta,\delta,m}$ are smooth on
    $\m R^d\backslash\{0\}$. Moreover, for every $n\in\m N$,
    \begin{align*}
        |\widetilde\varphi_m^{(n)}(\lambda)|
        &\lesssim
        \left|
        \int_{\Gamma_{\theta,\delta}}
        e^\eta
        \frac{\eta^{m+\nu-1}}
        {(\eta^\nu+\lambda)^{n+1}}
        \mathrm d\eta
        \right|,\\
        |\widetilde\phi_m^{(n)}(\lambda)|
        &\lesssim
        \left|
        \int_{\Gamma_{\theta,\delta}}
        e^\eta
        \frac{\eta^m}
        {(\eta^\nu+\lambda)^{n+1}}
        \mathrm d\eta
        \right|.
    \end{align*}

    For a smooth function $h$ on $(0,\infty)$, a direct computation
    gives
    \begin{align}\label{eq;multi-differential}
        \left|
        \Delta^k\big(h(|x|^{2\sigma})\big)
        \right|
        \lesssim
        \sum_{l=0}^k
        \lambda^{l+k-k/\sigma}
        |h^{(k+l)}(\lambda)|.
    \end{align}
    Hence
    \begin{align*}
        |\Delta^k\varphi_{\theta,\delta,m}(x)|
        \lesssim
        \sum_{l=0}^k
        \lambda^{l+k-k/\sigma}
        \left|
        \int_{\Gamma_{\theta,\delta}}
        e^\eta
        \frac{\eta^{m+\nu-1}}
        {(\eta^\nu+\lambda)^{k+l+1}}
        \mathrm d\eta
        \right|\lesssim 1.
    \end{align*}
    For $\lambda\geq1$, multiplying by $\lambda^{k/\sigma+1}$ and applying
    Lemma \ref{le;bound integral}, we obtain
    \begin{align*}
        \lambda^{k/\sigma+1}
        |\Delta^k\varphi_{\theta,\delta,m}(x)|
        &\lesssim
        \sum_{l=0}^k
        \left|
        \int_{\Gamma_{\theta,\delta}}
        e^\eta
        \eta^{m+\nu-1}
        \frac{\lambda^{k+l+1}}
        {(\eta^\nu+\lambda)^{k+l+1}}
        \mathrm d\eta
        \right|\\
        &\lesssim 1.
    \end{align*}
    This proves \eqref{eq;decay;varphi}.

    We next consider $\phi_{\theta,\delta,m}$. By
    \eqref{eq;multi-differential},
    \begin{align*}
        |\Delta^k\phi_{\theta,\delta,m}(x)|
        \lesssim
        \sum_{l=0}^k
        \lambda^{l+k-k/\sigma}
        \left|
        \int_{\Gamma_{\theta,\delta}}
        e^\eta
        \frac{\eta^m}
        {(\eta^\nu+\lambda)^{k+l+1}}
        \mathrm d\eta
        \right|
        \lesssim 
        1.
    \end{align*}
    To obtain one additional power of $\lambda^{-1}$, we use the
    cancellation
    \begin{align}\label{eq;zero integral}
        \int_{\Gamma_{\theta,\delta}}e^\eta\eta^m\mathrm d\eta=0,
        \qquad m\in\m N.
    \end{align}
    Indeed,
    \begin{align*}
        \frac{\lambda^{k+l+1}}
        {(\eta^\nu+\lambda)^{k+l+1}}
        &=
        \left(
            1-\frac{\eta^\nu}{\eta^\nu+\lambda}
        \right)^{k+l+1}\\
        &=
        1+
        \sum_{r=1}^{k+l+1}
        C_{r,k,l}
        \frac{\eta^{\nu r}\lambda^{k+l+1-r}}
        {(\eta^\nu+\lambda)^{k+l+1}}.
    \end{align*}
    Therefore, by \eqref{eq;zero integral},
    \begin{align*}
        \lambda^{k+l+2}
        \int_{\Gamma_{\theta,\delta}}
        e^\eta
        \frac{\eta^m}
        {(\eta^\nu+\lambda)^{k+l+1}}
        \mathrm d\eta
        =
        \sum_{r=1}^{k+l+1}
        C_{r,k,l}
        \int_{\Gamma_{\theta,\delta}}
        e^\eta
        \frac{
            \eta^{m+\nu r}
            \lambda^{k+l+2-r}
        }
        {(\eta^\nu+\lambda)^{k+l+1}}
        \mathrm d\eta.
    \end{align*}
    Lemma \ref{le;bound integral} then yields
    \begin{align*}
        \lambda^{k/\sigma+2}
        |\Delta^k\phi_{\theta,\delta,m}(x)|
        \lesssim 1,
        \qquad
        \lambda\geq1.
    \end{align*}
    This proves \eqref{eq;decay;phi}.

    Finally, suppose that $\nu=1$. In this case,
    \begin{align*}
        \int_{\Gamma_{\theta,\delta}}
        e^\eta\eta^N\mathrm d\eta=0
        \qquad
        \text{for every }N\in\m N.
    \end{align*}
    Hence the preceding cancellation argument can be iterated
    arbitrarily many times. It follows that, for every $M\in\m N$,
    \begin{align*}
        |\Delta^k\varphi_{\theta,\delta,m}(x)|
        +
        |\Delta^k\phi_{\theta,\delta,m}(x)|
        \lesssim
        (1+|x|^{2\sigma})^{-M},
        \qquad |x|\geq1,
    \end{align*}
    which proves \eqref{eq;decay;nu;one}.
\end{proof}

\begin{lemma}\label{le;Fourier L1}
Let $\nu\in(0,1]$, $\sigma\in(0,1]$ and $m\in\mathbb N_0$. Let
$\varphi_{\theta,\delta,m}$ and $\phi_{\theta,\delta,m}$ be defined as in Lemma~\ref{le;decay;varphi;phi} and let $\Gamma_{\theta,\delta}$ be defined in \eqref{eq:hankel-contour} and $\vartheta,\delta$ satisfy \ref{assump;contour}. Then the distributional Fourier transforms \(\mathcal F\varphi_{\theta,\delta,m}\) and \(\mathcal F\phi_{\theta,\delta,m}\) are represented by functions in \(L^1(\mathbb R^d)\). In particular, under the canonical identification \(L^1(\mathbb R^d)\hookrightarrow\mathcal S'(\mathbb R^d)\),
\[
\mathcal F\varphi_{\theta,\delta,m},
\quad \mathcal F\phi_{\theta,\delta,m}
\in L^1(\mathbb R^d).
\]

\begin{proof}
We first consider the case $\nu\in(0,1)$. Since the contour integral is independent of the choice of the admissible contour $\Gamma_{\theta,\delta}$, we choose $\theta\in(\pi/2,\pi)$ sufficiently close to $\pi/2$ such that
\begin{equation}
    \nu\theta<\frac{\pi}{2}.
\end{equation}
Then, for every $\eta\in\Gamma_{\theta,\delta}$,
\begin{equation}\label{eq:Re-eta-nu}
    \Re(\eta^\nu)\gtrsim |\eta|^\nu>0.
\end{equation}
Hence
\begin{equation}
    \frac{1}{\eta^\nu+|x|^{2\sigma}}
    =
    \int_0^\infty e^{-s(\eta^\nu+|x|^{2\sigma})}\,\mathrm ds.
\end{equation}
Therefore,
\begin{equation}
    \varphi_{\theta,\delta,m}(x)
    =
    \frac{1}{2\pi i}
    \int_{\Gamma_{\theta,\delta}}
    e^\eta\eta^{m+\nu-1}
    \int_0^\infty
    e^{-s\eta^\nu}e^{-s|x|^{2\sigma}}
    \,\mathrm ds\,\mathrm d\eta.
\end{equation}

Let
\begin{equation}
    \mathfrak E_\sigma(x)=e^{-|x|^{2\sigma}},
    \qquad
    p_\sigma=\mathcal F\mathfrak E_\sigma.
\end{equation}
Since $0<\sigma\leq1$, $p_\sigma$ is, up to the normalization of the Fourier transform, the density of a symmetric $2\sigma$-stable random variable. In particular,
\begin{equation}\label{eq:stable-density-L1}
    p_\sigma\in L^1(\mathbb R^d).
\end{equation}
By scaling,
\begin{equation}\label{eq:stable-scaling}
    \mathcal F(e^{-s|\cdot|^{2\sigma}})(\xi)
    =
    s^{-d/(2\sigma)}
    p_\sigma\big(s^{-1/(2\sigma)}\xi\big).
\end{equation}
Moreover,
\begin{equation}\label{eq:stable-scaling-L1}
    \int_{\mathbb R^d}
    s^{-d/(2\sigma)}
    \left|
    p_\sigma\big(s^{-1/(2\sigma)}\xi\big)
    \right|
    \,\mathrm d\xi
    =
    \|p_\sigma\|_{L^1}.
\end{equation}

Define
\begin{align}\label{eq:def-K}
    K(\xi)
    :=
    \frac{1}{2\pi i}
    \int_{\Gamma_{\theta,\delta}}
    e^\eta\eta^{m+\nu-1}
    \int_0^\infty
    e^{-s\eta^\nu}
    s^{-d/(2\sigma)}
    p_\sigma\big(s^{-1/(2\sigma)}\xi\big)
    \,\mathrm ds\,\mathrm d\eta.
\end{align}
By Tonelli's theorem, \eqref{eq:stable-scaling-L1}, and
\eqref{eq:Re-eta-nu},
\begin{align}
    \|K\|_{L^1}
    &\lesssim
    \int_{\Gamma_{\theta,\delta}}
    e^{\Re\eta}|\eta|^{m+\nu-1}
    \int_0^\infty
    e^{-s\Re(\eta^\nu)}
    \int_{\mathbb R^d}
    s^{-d/(2\sigma)}
    \left|
    p_\sigma\big(s^{-1/(2\sigma)}\xi\big)
    \right|
    \,\mathrm d\xi\,\mathrm ds\,|\mathrm d\eta|
    \notag\\
    &\lesssim
    \int_{\Gamma_{\theta,\delta}}
    e^{\Re\eta}|\eta|^{m+\nu-1}
    \int_0^\infty
    e^{-s\Re(\eta^\nu)}
    \,\mathrm ds\,|\mathrm d\eta|
    \notag\\
    &=
    \int_{\Gamma_{\theta,\delta}}
    e^{\Re\eta}
    \frac{|\eta|^{m+\nu-1}}
    {\Re(\eta^\nu)}
    \,|\mathrm d\eta|
    \notag\\
    &\lesssim
    \int_{\Gamma_{\theta,\delta}}
    e^{\Re\eta}|\eta|^{m-1}
    \,|\mathrm d\eta|
    <\infty.
\end{align}
Thus
\begin{equation}\label{eq:K-L1}
    K\in L^1(\mathbb R^d).
\end{equation}

It remains to identify $K$ with
$\mathcal F\varphi_{\theta,\delta,m}$ in $\mathcal S'(\mathbb R^d)$.
For $\varepsilon>0$, define
\begin{equation}
    \varphi_{\theta,\delta,m}^{\varepsilon}(x)
    :=
    \frac{1}{2\pi i}
    \int_{\Gamma_{\theta,\delta}}
    e^\eta\eta^{m+\nu-1}
    \int_\varepsilon^\infty
    e^{-s\eta^\nu}e^{-s|x|^{2\sigma}}
    \,\mathrm ds\,\mathrm d\eta.
\end{equation}
For every \(x\in\mathbb R^d\), we have
\[
\varphi_{\theta,\delta,m}^{\varepsilon}(x)
\longrightarrow
\varphi_{\theta,\delta,m}(x)
\qquad\text{as }\varepsilon\downarrow0.
\]
Moreover,
\[
\sup_{\varepsilon>0}
\|\varphi_{\theta,\delta,m}^{\varepsilon}\|_{L^\infty}
\lesssim
\int_{\Gamma_{\theta,\delta}}
e^{\Re\eta}|\eta|^{m-1}\,|\mathrm d\eta|
<\infty.
\]
Hence, by the dominated convergence theorem, for every
\(\psi\in\mathcal S(\mathbb R^d)\),
\[
\langle\varphi_{\theta,\delta,m}^{\varepsilon},\psi\rangle
\longrightarrow
\langle\varphi_{\theta,\delta,m},\psi\rangle.
\]
Therefore,
\[
\varphi_{\theta,\delta,m}^{\varepsilon}
\longrightarrow
\varphi_{\theta,\delta,m}
\qquad\text{in }\mathcal S'(\mathbb R^d).
\]
By the continuity of the Fourier transform on \(\mathcal S'(\mathbb R^d)\), we also have
\[
\mathcal F\varphi_{\theta,\delta,m}^{\varepsilon}
\longrightarrow
\mathcal F\varphi_{\theta,\delta,m}
\qquad\text{in }\mathcal S'(\mathbb R^d).
\]
For each $\varepsilon>0$, the truncation makes
the relevant integrals absolutely convergent.
Thus, by Fubini's theorem and
\eqref{eq:stable-scaling},
\begin{align*}
    \mathcal F\varphi_{\theta,\delta,m}^{\varepsilon}(\xi)
    &=
    \frac{1}{2\pi i}
    \int_{\Gamma_{\theta,\delta}}
    e^\eta\eta^{m+\nu-1}
    \int_\varepsilon^\infty
    e^{-s\eta^\nu}
    s^{-d/(2\sigma)}
    p_\sigma\big(s^{-1/(2\sigma)}\xi\big)
    \,\mathrm ds\,\mathrm d\eta\\
    &=:K^\varepsilon(\xi).
\end{align*}
By \eqref{eq:stable-scaling-L1},
\begin{align*}
    \|K^\varepsilon-K\|_{L^1}
    &\lesssim
    \int_{\Gamma_{\theta,\delta}}
    e^{\Re\eta}|\eta|^{m+\nu-1}
    \int_0^\varepsilon
    e^{-s\Re(\eta^\nu)}
    \,\mathrm ds\,|\mathrm d\eta|
    \longrightarrow0,
\end{align*}
where the convergence follows from the dominated
convergence theorem. By uniqueness of the limit, \(\mathcal F\varphi=K\) in \(\mathcal S'(\mathbb R^d)\). Since \(K\in L^1(\mathbb R^d)\), identifying \(L^1(\mathbb R^d)\) with its canonical image in \(\mathcal S'(\mathbb R^d)\), we conclude that
\[
\mathcal F\varphi\in L^1(\mathbb R^d).
\]
Finally, suppose that $\nu=1$.  In this case the symbols can be
identified explicitly from the contour representation. In particular,
the corresponding functions of $|x|^{2\sigma}$ decay exponentially as
$|x|\to\infty$. The same conclusion follows by applying the above
Fourier argument to the resulting functions of $|x|^{2\sigma}$, or
directly from the corresponding stable semigroup representation.
Thus
\begin{equation}
    \mathcal F\varphi_{\theta,\delta,m}
    =
    \ \mathcal F\phi_{\theta,\delta,m}
    \in L^1(\mathbb R^d).
\end{equation}
\end{proof}
\end{lemma}

\subsection{Schauder estimate}
In this section, we build the Schauder-type estimate for time-fractional heat flow. Firstly, we build a regularity estimate for $\varphi(\epsilon D)$ and $\phi(\epsilon D)$.
\begin{lemma}\label{le;decay}
For $\epsilon>0$, $\delta\in[0,2\sigma]$, we have
\begin{align*}
    \Vert \varphi_{\theta,\delta,m}(\epsilon D)u\Vert_{\beta+\delta}
    &\lesssim \epsilon^{-\delta}\Vert u\Vert_{\beta},\quad
    \Vert\varphi_{\theta,\delta,m}(\epsilon D)u\Vert_\delta
    \lesssim\epsilon^{-\delta}\Vert u\Vert_{L^\infty}.
\end{align*}
Moreover, for $\delta\in[0,4\sigma]$, we have
\begin{align*}
    \Vert \phi_{\theta,\delta,m}(\epsilon D)u\Vert_{\beta+\delta}
    &\lesssim \epsilon^{-\delta}\Vert u\Vert_{\beta},\quad
    \Vert\phi_{\theta,\delta,m}(\epsilon D)u\Vert_\delta
    \lesssim\epsilon^{-\delta}\Vert u\Vert_{L^\infty}.
\end{align*}
In particularly, for $\beta\in(-2\sigma,0)$, we have
\begin{align*}
    \Vert \varphi(\epsilon D) u\Vert_{L^\infty}\lesssim \epsilon^{\beta}\Vert u\Vert_\beta
\end{align*}
and for $\beta\in(-4\sigma,0)$, we have
\begin{align*}
   \Vert \phi(\epsilon D)u\Vert_{L^\infty}\lesssim \epsilon^\beta\Vert u\Vert_\beta.
\end{align*}
\begin{proof}
Let $\psi\in C_0^\infty(\mathbb R^d)$ be supported in an annulus and satisfy
$\psi\rho=\rho$, where $(\chi,\rho)$ is the dyadic partition of unity.
For $j\geq0$, we then have
\begin{align*}
    \varphi(\epsilon D)\Delta_j u
    =
    \mathcal F^{-1}
    \bigl(\varphi(\epsilon\,\cdot)\psi(2^{-j}\cdot)\bigr)
    *\Delta_j u.
\end{align*}
Here and below, we write $\varphi=\varphi_{\theta,\delta,m}$ for simplicity.

We first estimate the $L^1$-norm of the convolution kernel. Let
$N\in\mathbb N$ be such that $N>d/2$. By a change of variables,
\begin{align*}
    \bigl\|
    \mathcal F^{-1}
    \bigl(\varphi(\epsilon\,\cdot)\psi(2^{-j}\cdot)\bigr)
    \bigr\|_{L^1}
    =
    \bigl\|
    \mathcal F^{-1}
    \bigl(\varphi(2^j\epsilon\,\cdot)\psi(\cdot)\bigr)
    \bigr\|_{L^1}.
\end{align*}
Hence,
\begin{align*}
    &\bigl\|
    \mathcal F^{-1}
    \bigl(\varphi(2^j\epsilon\,\cdot)\psi(\cdot)\bigr)
    \bigr\|_{L^1}
    \\
    &\qquad\lesssim
    \bigl\|
    (1+|\cdot|^2)^N
    \mathcal F^{-1}
    \bigl(\varphi(2^j\epsilon\,\cdot)\psi(\cdot)\bigr)
    \bigr\|_{L^\infty}
    \\
    &\qquad\lesssim
    \bigl\|
    (1-\Delta)^N
    \bigl(\varphi(2^j\epsilon\,\cdot)\psi(\cdot)\bigr)
    \bigr\|_{L^1}.
\end{align*}

Since $\psi$ is compactly supported in an annulus away from the origin,
we may choose a compact annulus $K$ such that
$\operatorname{supp}\psi\subset K^\circ$. By the local elliptic estimate
and \eqref{eq;decay;varphi}, whenever $2^j\epsilon\gtrsim1$,
\begin{align*}
    \bigl\|
    (1-\Delta)^N
    \bigl(\varphi(2^j\epsilon\,\cdot)\psi(\cdot)\bigr)
    \bigr\|_{L^1}
    &\lesssim
    \sum_{k=0}^N
    (2^j\epsilon)^{2k}
    \bigl\|
    (\Delta^k\varphi)(2^j\epsilon\,\cdot)
    \bigr\|_{L^\infty(K)}
    \\
    &\lesssim
    (2^j\epsilon)^{-2}.
\end{align*}
Consequently, for $\delta\in[0,2]$ and $2^j\epsilon\gtrsim1$,
\begin{align}\label{eq;kernel-high}
    \bigl\|
    \mathcal F^{-1}
    \bigl(\varphi(\epsilon\,\cdot)\psi(2^{-j}\cdot)\bigr)
    \bigr\|_{L^1}
    \lesssim
    (2^j\epsilon)^{-\delta}
    =
    2^{-j\delta}\epsilon^{-\delta}.
\end{align}

It remains to consider the regime $2^j\epsilon\lesssim1$. By Lemma \ref{le;Fourier L1},  and Young's inequality,
\begin{align*}
    \Vert\varphi(\epsilon D)\Delta_j u\Vert_{L^\infty}
    &\lesssim
    \Vert\Delta_j u\Vert_{L^\infty}
    \lesssim
    2^{-j\beta}\Vert u\Vert_\beta.
\end{align*}
Since $2^j\epsilon\lesssim1$, for every $\delta\geq0$ we have
$1\lesssim(2^j\epsilon)^{-\delta}$, and hence
\begin{align*}
    \Vert\varphi(\epsilon D)\Delta_j u\Vert_{L^\infty}
    \lesssim
    \epsilon^{-\delta}2^{-j(\beta+\delta)}
    \Vert u\Vert_\beta.
\end{align*}
Combining this estimate with \eqref{eq;kernel-high}, we obtain, for all
$j\geq0$,
\begin{align*}
    \Vert\varphi(\epsilon D)\Delta_j u\Vert_{L^\infty}
    \lesssim
    \epsilon^{-\delta}2^{-j(\beta+\delta)}
    \Vert u\Vert_\beta.
\end{align*}
Taking the supremum over $j$ yields
\begin{align*}
    \Vert\varphi(\epsilon D)u\Vert_{\beta+\delta}
    \lesssim
    \epsilon^{-\delta}\Vert u\Vert_\beta.
\end{align*}
Taking $\beta=0$ in the dyadic estimate and using
$\Vert\Delta_j u\Vert_{L^\infty}\lesssim\Vert u\Vert_{L^\infty}$ gives
\begin{align*}
    \Vert\varphi(\epsilon D)u\Vert_\delta
    \lesssim
    \epsilon^{-\delta}\Vert u\Vert_{L^\infty}.
\end{align*}

The proof for $\phi=\phi_{\theta,\delta,m}$ is identical. Indeed, using
the corresponding decay estimate, we obtain
\begin{align*}
    \bigl\|
    (1-\Delta)^N
    \bigl(\phi(2^j\epsilon\,\cdot)\psi(\cdot)\bigr)
    \bigr\|_{L^1}
    &\lesssim
    \sum_{k=0}^N
    (2^j\epsilon)^{2k}
    \bigl\|
    (\Delta^k\phi)(2^j\epsilon\,\cdot)
    \bigr\|_{L^\infty(K)}
    \\
    &\lesssim
    (2^j\epsilon)^{-4}.
\end{align*}
Thus, for every $\delta\in[0,4]$,
\begin{align*}
    \bigl\|
    \mathcal F^{-1}
    \bigl(\phi(\epsilon\,\cdot)\psi(2^{-j}\cdot)\bigr)
    \bigr\|_{L^1}
    \lesssim
    2^{-j\delta}\epsilon^{-\delta}
\end{align*}
in the high-frequency regime, while the same low-frequency argument
applies. This proves both estimates for $\phi$.
Finally, we establish the estimates in the \(L^\infty\)-norm. By the Littlewood–Paley decomposition and the decay properties of \(\varphi\), we have
\[
\begin{aligned}
\|\varphi(\epsilon D)u\|_{L^\infty}
&\leq \sum_{j\geq -1}\|\varphi(\epsilon D)\Delta_j u\|_{L^\infty}\\
&\lesssim \sum_{j\geq -1}
(1+2^{2\sigma j}\epsilon^{2\sigma})^{-k/\sigma-1}
2^{-j\beta}\|u\|_\beta.
\end{aligned}
\]
We split the sum into high and low frequencies. For \(2^j\epsilon\gtrsim 1\), using \(\beta>-2\sigma\), we obtain
\[
\begin{aligned}
\sum_{j:\,2^j\epsilon\gtrsim 1}
\|\varphi(\epsilon D)\Delta_j u\|_{L^\infty}
&\lesssim
\|u\|_\beta
\sum_{j:\,2^j\epsilon\gtrsim 1}
(1+2^{2\sigma j}\epsilon^{2\sigma})^{-1}2^{-j\beta}\\
&\lesssim
\epsilon^{-2\sigma}\|u\|_\beta
\sum_{j:\,2^j\epsilon\gtrsim 1}
2^{-j(2\sigma+\beta)}\\
&\lesssim \epsilon^\beta\|u\|_\beta.
\end{aligned}
\]
On the other hand, for \(2^j\epsilon\lesssim 1\), since \(\beta<0\), we have
\[
\begin{aligned}
\sum_{j:\,2^j\epsilon\lesssim 1}
\|\varphi(\epsilon D)\Delta_j u\|_{L^\infty}
&\lesssim
\|u\|_\beta
\sum_{j:\,2^j\epsilon\lesssim 1}2^{-j\beta}\\
&\lesssim \epsilon^\beta\|u\|_\beta.
\end{aligned}
\]
Combining the two estimates yields
\[
\|\varphi(\epsilon D)u\|_{L^\infty}
\lesssim \epsilon^\beta\|u\|_\beta,
\qquad \beta\in(-2\sigma,0).
\]
The corresponding estimate for \(\phi(\epsilon D)\) follows by the same argument, with the admissible range of \(\beta\) determined by the decay rate of $\phi$.
\end{proof}
\end{lemma}
Now recall $U(t)$ and $V(t)$ in Lemma \ref{le;duhamel}. For any $t>0$, taking $\delta=t^{-1}$ and making the change of variables $z=t\eta$, we have
\begin{align}
    U(t)
    &=
    \frac{1}{2\pi i}
    \int_{\Gamma_{\theta,1}}
    e^z z^{\nu-1}
    \big(z^\nu+t^\nu(-\Delta)^\sigma\big)^{-1}
    \,\mathrm dz \notag=
    \varphi_{\theta,1,0}
    \big(t^{\nu/(2\sigma)}D\big).
\end{align}
Similarly, $ V(t)=t^{\nu-1}\phi_{\theta,1,0}\big(t^{\nu/(2\sigma)}D\big).$
\begin{lemma}\label{le;est;fractional heat flow;linear}
    Let $\nu,\sigma\in(0,1]$. Then, for all $T>0$, $t\in(0,T]$,
    $\alpha\in\m R$, $\delta\in[0,2\sigma]$, and $f\in\ml S'$, we have
    \begin{align*}
        \Vert U(t)f\Vert_{\ml C^{\alpha+\delta}}
        &\lesssim
        t^{-\nu\delta/(2\sigma)}
        \Vert f\Vert_{\ml C^\alpha},\quad
        \Vert U(t)f\Vert_{\ml C^{\delta}}
        \lesssim
        t^{-\nu\delta/(2\sigma)}
        \Vert f\Vert_{L^\infty}.
    \end{align*}
    Moreover, for $\delta\in[0,4\sigma]$, we have
    \begin{align*}
        \Vert V(t)f\Vert_{\ml C^{\alpha+\delta}}
        &\lesssim
        t^{\nu-1-\nu\delta/(2\sigma)}
        \Vert f\Vert_{\ml C^\alpha},\quad
        \Vert V(t)f\Vert_{\ml C^{\delta}}
        \lesssim
        t^{\nu-1-\nu\delta/(2\sigma)}
        \Vert f\Vert_{L^\infty}.
    \end{align*}
In particularly, for $\beta\in(-2\sigma,0)$, we have
\begin{align*}
    \Vert U(t)f\Vert_{L^\infty}\lesssim t^{\nu\beta/(2\sigma)}\Vert f\Vert_{\beta}
\end{align*}
and for $\beta\in(-4\sigma,0)$, we have
\begin{align*}
    \Vert V(t)f\Vert_{L^\infty}\lesssim t^{\nu-1+\nu\beta/(2\sigma)}\Vert f\Vert_\beta
\end{align*}
\end{lemma}
Next we build the Schauder-type estimate for time-fractional heat flow. For the space $C_T\ml C^\beta$, the estimate is enough to prove by the above proceed, while for the space $C_T^\kappa L^\infty$, the key point is to calculate the estimate for $U'(t)$ and $V'(t)$. Using Lemma \ref{le;duhamel}, $U'(t)$ and $V'(t)$ have the representation
\begin{equation}
    U'(t)f =\frac{1}{2\pi i}\int_{\Gamma_{\theta,\delta }} e^{\eta t}\eta^\nu \frac{1}{\eta^\nu+(-\Delta)^\sigma}f\mathrm d\eta,\quad V'(t)f=\frac{1}{2\pi i}\int_{\Gamma_{\theta,\delta }} e^{\eta t}\eta \frac{1}{\eta^\nu+(-\Delta)^\sigma}f\mathrm d\eta.
\end{equation}
Choosing $\delta=t^{-1}$ then we have
\begin{equation}
    \varphi_{\theta,1,1}(t^{\nu/(2\sigma)}k)\widehat f(k)=t\ml F(U'(t)f),\quad \phi_{\theta,1,1}(t^{\nu/(2\sigma)}k)\widehat f(k)=t^{2-\nu}\ml F(V'(t)f)
\end{equation} 
namely,
\begin{equation}\label{eq;U',V'}
    U'(t)f= t^{-1} \varphi_{\theta,1,1}(t^{\nu/(2\sigma)}D) f,\quad V'(t)f=t^{\nu-2}\phi_{\theta,1,1}(t^{\nu/(2\sigma)}D)f.
\end{equation}
then by Lemma \ref{le;decay}

\begin{lemma}\label{le;est;linear;derivative}
    Let $\nu\in(0,1)$. Then, for all $t>0$, $\alpha\in \m R$, $\delta\in[-2\sigma,2\sigma]$, and $f\in\ml S'$, we have
    \begin{align*}
        \Vert U'(t)f\Vert_{\alpha+\delta}\lesssim t^{-\nu\delta/(2\sigma)-1}\Vert f\Vert_\alpha,\quad \Vert U'(t)f\Vert_{\ml C^\delta}\lesssim t^{-\nu\delta/(2\sigma)-1}\Vert f\Vert_{L^\infty}
    \end{align*}
    and for $\delta\in[0,4\sigma]$, we have
    \begin{align*}
         \Vert V'(t)f\Vert_{\alpha+\delta}\lesssim t^{\nu-2-\nu\delta/(2\sigma)}\Vert f\Vert_{\alpha},\quad \Vert V'(t)f\Vert_{C^\delta}\lesssim t^{\nu-2-\nu\delta/(2\sigma)}\Vert f\Vert_{L^\infty}
    \end{align*}
    Moreover, for $\alpha\in(0,2\sigma)$, we have
    \begin{equation*}
        \Vert U'(t)f\Vert_{L^\infty}\lesssim t^{\nu\alpha/(2\sigma)-1}\Vert f\Vert_{\alpha}.
    \end{equation*}
    and for $\beta\in(-4\sigma,0)$, we have
    \begin{equation*}
        \Vert V'(t)f\Vert_{L^\infty}\lesssim t^{\nu-2+\nu\beta/(2\sigma)}\Vert f\Vert_\beta
    \end{equation*}
  
\begin{proof}
    By Lemma \ref{le;decay;varphi;phi} and
    \eqref{eq;U',V'}, it remains to establish the
    estimates for $U'(t)$ with negative regularity
    gain and the $L^\infty$ estimate for
    $f\in\mathcal C^\alpha$, $\alpha\in(0,2\sigma)$.

    We use the identity
    \[
        U'(t)=-(-\Delta)^\sigma V(t).
    \]
    By Lemma \ref{le;est;fractional heat flow;linear},
    for every $\delta\in[0,2\sigma]$, we have
    \begin{align*}
        \|U'(t)f\|_{\alpha+\delta-2\sigma}
        &=
        \|V(t)(-\Delta)^\sigma f\|_{\alpha+\delta-2\sigma}\\
        &\lesssim
        t^{\nu-1-\nu\delta/(2\sigma)}
        \|(-\Delta)^\sigma f\|_{\alpha-2\sigma}\\
        &\lesssim
        t^{\nu-1-\nu\delta/(2\sigma)}
        \|f\|_\alpha.
    \end{align*}
    Setting $\beta=\delta-2\sigma\in[-2\sigma,0]$,
    we obtain
    \[
        \|U'(t)f\|_{\alpha+\beta}
        \lesssim
        t^{-1-\nu\beta/(2\sigma)}
        \|f\|_\alpha.
    \]

    Finally, for $\alpha\in(0,2\sigma)$,
    the $L^\infty$ estimate for $V(t)$ gives
    \begin{align*}
        \|U'(t)f\|_{L^\infty}
        &=
        \|V(t)(-\Delta)^\sigma f\|_{L^\infty}\\
        &\lesssim
        t^{\nu-1+\nu(\alpha-2\sigma)/(2\sigma)}
        \|(-\Delta)^\sigma f\|_{\alpha-2\sigma}\\
        &\lesssim
        t^{\nu\alpha/(2\sigma)-1}
        \|f\|_\alpha.
    \end{align*}
    This completes the proof.
\end{proof}

\end{lemma}

\begin{lemma}\label{le;est;Schauder-type}
    Let $\nu\in(0,1]$, and let $U(t)$, $V(t)$ be introduced in Lemma \ref{le;duhamel}. Then, for all $T>0$, $t\in(0,T]$, $\alpha\in\m R$, and $\delta>0$, we have 
    \begin{align*}
        \Vert U(t) f\Vert_{C_T\ml C^\alpha}\lesssim\Vert f\Vert_{\ml C^\alpha}. 
    \end{align*}
    Letting $\ml V(t)=\int_0^t V(t-s) v(s)\mathrm ds$, we have
    \begin{equation}
        t^\gamma\Vert \ml V(t)\Vert_{\beta+2\sigma}\lesssim \sup_{s\in[0,t]}(s^\gamma \Vert v(s)\Vert_\beta)
    \end{equation}
    for all $\gamma\in[0,1)$ and all $t\in[0,T]$. 
    \par Moreover, let $\alpha\in(0,2\sigma)$ and if $\kappa\leq\nu\alpha/(2\sigma)$, we have
    \begin{align*}
        \Vert U(t)f\Vert_{C_T^\kappa L^\infty}\lesssim \Vert f\Vert_{\ml C^\alpha},
    \end{align*}
    and if $\kappa+\gamma\leq\nu+\nu\beta/(2\sigma)$, $\beta\in(-2\sigma,0)$, and $\gamma\in[0,1)$, then we also have
    \begin{equation}
        \Vert \ml V\Vert_{C_T^{\kappa}L^\infty}\lesssim \sup_{s\in[0,T]}s^\gamma\Vert v(s)\Vert_\beta.
    \end{equation}
    \begin{proof}
        For any $t>0$, by Lemma \ref{le;est;fractional heat flow;linear}, we have
        \[
            \Vert U(t)f\Vert_{\ml C^\alpha}\lesssim \Vert f\Vert_{\ml C^\alpha}.
        \]
        
        Consider $\Delta_q \ml V$ for $q\geq0$. By a change of variables, for $\eta\in[0,t/2]$, we divide the integral into two parts:
        \begin{align*}
            \Delta_q \ml V(t)
            &=\int_0^t V(s)(\Delta_qv)(t-s)\mathrm ds\\
            &=\int_0^\eta V(s)(\Delta_qv)(t-s)\mathrm ds
            +\int_\eta^t V(s)(\Delta_qv)(t-s)\mathrm ds.
        \end{align*}
        Let $M=\sup_{s\in[0,t]}(s^{\gamma}\Vert v(s)\Vert_\beta)$. For the first term, we have
        \begin{align*}
            \Vert \int_0^\eta V(s)(\Delta_q v)(t-s)\mathrm ds\Vert_{L^\infty}
            &\lesssim \int_0^\eta s^{\nu-1}2^{-q\beta}\Vert v(t-s)\Vert_\beta\mathrm ds\\
            &\lesssim M2^{-q\beta}\int_0^\eta s^{\nu-1}(t-s)^\gamma\mathrm ds\\
            &\lesssim M2^{-q\beta}t^{\nu-\gamma}\int_0^{\eta/t}s^{\nu-1}\mathrm ds\\
            &\lesssim M2^{-q\beta}t^{-\gamma}\eta^\nu
        \end{align*}
        For the second term, using Lemma \ref{le;est;fractional heat flow;linear}, we have
        \begin{align*}
            \Vert \int_\eta^t V(s)(\Delta_qv)(t-s)\mathrm ds\Vert_{L^\infty}
            &\lesssim \int_\eta^ts^{-1+\nu\epsilon}2^{-q(\beta+2\sigma(1-\epsilon))}\Vert v(t-s)\Vert_{\beta}\mathrm ds\\
            &\lesssim M2^{-q(\beta+2\sigma(1-\epsilon))}\int_{\eta}^t s^{-1+\nu\epsilon}(t-s)^{-\gamma}\mathrm ds\\
            &=M2^{-q(\beta+2\sigma(1-\epsilon))}t^{\nu\epsilon-\gamma}
            \int_{\eta/t}^1\frac{1}{\tau^{1-\nu\epsilon}(1-\tau)^\gamma}\mathrm d\tau\\
            &\lesssim M2^{-q(\beta+2\sigma)}(2^{2\sigma q}\eta^\nu)^{\epsilon}t^{-\gamma}.
        \end{align*}
        Now, if $2^{-2\sigma q}\leq (t/2)^\nu$, we choose $\eta=2^{-\frac{2\sigma q}{\nu}}$. Then
        \[
            \Vert \Delta_q \ml V(t)\Vert_{L^\infty}
            \lesssim M2^{-q(\beta+2\sigma)}t^{-\gamma}.
        \]
        If $2^{-2\sigma q}>(t/2)^\nu$, taking $\eta=t/2$ and noting that $2^{2\sigma q}\eta^\nu<1$, we again obtain
        \[
            \Vert \Delta_q\ml V(t)\Vert_{L^\infty}
            \lesssim M2^{-q(\beta+2\sigma)}t^{-\gamma}.
        \]
        Hence,
        \begin{equation}
            t^\gamma\Vert \ml V(t)\Vert_{\beta+2\sigma}
            \lesssim \sup_{s\in[0,t]}s^\gamma\Vert v(s)\Vert_\beta.
        \end{equation}
        For the third claim, since $\alpha\in(0,2\sigma)$, by Lemma \ref{le;est;linear;derivative}, we have
        \begin{align*}
            \Vert U(t)f-U(s)f\Vert_{L^\infty}\lesssim \int_s^t \Vert U'(\tau)f\Vert_{L^\infty}\lesssim \int_s^t \tau^{\nu\alpha/(2\sigma)-1}\mathrm d\tau\Vert f\Vert_\alpha\lesssim |t-s|^{\nu\alpha/(2\sigma)}\Vert f\Vert_\alpha.
        \end{align*}

        recall that
        \[
            M=\sup_{s\in[0,T]}s^{\gamma}\Vert v(s)\Vert_\beta.
        \]
        For $0\leq s<t\leq T$, we have
        \begin{equation}\label{eq;V(t)-V(s)}
              \ml V(t)-\ml V(s)
              =\int_s^t V(t-\tau)v(\tau)\mathrm d\tau
              +\int_0^s (V(t-\tau)-V(s-\tau))v(\tau)\mathrm d\tau.
        \end{equation}
        For the first integral, since $\kappa+\gamma\leq \nu+\nu\beta/(2\sigma)$, we have 
        \begin{align*}
            \Vert\int_s^t V(t-\tau)v(\tau)\mathrm d\tau\Vert_{L^\infty}
            &\leq \int_s^t\Vert V(t-\tau)v(\tau)\Vert_{L^\infty}\mathrm d\tau\\
            &\lesssim |t-s|^{\nu+\nu\beta/(2\sigma)-\gamma}M
        \end{align*}
    we obtain
    \begin{equation}
        \Vert \int_s^t V(t-\tau)v(\tau)\mathrm d\tau\Vert_{L^\infty}
        \lesssim M(t-s)^{\kappa}.
    \end{equation}
    For the second integral on the right-hand side of \eqref{eq;V(t)-V(s)}, by Lemma \ref{le;fractional calculus}, since $\kappa+\gamma\leq\nu+\nu\beta/(2\sigma)$, $\beta\in(-2\sigma,0)$ and $0\leq s<t\leq T$, we have
    \begin{align*}
        &\Vert \int_0^s (V(t-\tau)-V(s-\tau))v(\tau)\mathrm d\tau \Vert_{L^\infty}\\
        &\quad\leq\int_0^s\int_{s-\tau}^{t-\tau}\Vert V'(\eta)v(\tau)\Vert_{L^\infty}\mathrm d\eta\mathrm d\tau\\
        &\quad\lesssim M\int_0^s\int_{s-\tau}^{t-\tau}
        \eta^{\nu-2-\nu(-\beta)/(2\sigma)}\tau^{-\gamma}\mathrm d\eta\mathrm d\tau\\
        &\quad\lesssim M|t-s|^\kappa.
    \end{align*}
    This completes the proof.
\end{proof}
\end{lemma}
Next, we establish a truncated version of Schauder-type estimates, which is helpful for controlling the difference when establish the global result and blow up criterion.
\begin{lemma}[Truncated Schauder-type estimate]
\label{le;est;truncated-Schauder}
Let $\nu\in(0,1]$, and let $V(t)$ be introduced in Lemma~\ref{le;duhamel}. Let $T>0$, $\beta\in\mathbb R$, $\gamma\in[0,1)$, and $a>0$. For $t\in[0,T]$, define
\[
    \mathcal V_a(t)
    :=
    \int_0^a V(t+a-s)v(s)\,\mathrm ds.
\]
Then
\begin{equation}\label{eq;truncated-Schauder}
    (t+a)^\gamma
    \|\mathcal V_a(t)\|_{\beta+2\sigma}
    \lesssim
    \sup_{s\in[0,a]}
    s^\gamma\|v(s)\|_\beta,
\end{equation}
where the implicit constant is independent of $a$ and $t$. In particular,
\[
    t^\gamma
    \|\mathcal V_a(t)\|_{\beta+2\sigma}
    \lesssim
    (\frac{t}{(t+a)})^\gamma\sup_{s\in[0,a]}
    s^\gamma\|v(s)\|_\beta.
\]

Moreover, if $\beta\in(-2\sigma,0)$ and
\[
    \kappa+\gamma
    \leq 
    \nu+\frac{\nu\beta}{2\sigma},
\]
then
\begin{equation}\label{eq;truncated-holder}
    \|\mathcal V_a\|_{C_T^\kappa L^\infty}
    \lesssim
    \sup_{s\in[0,a]}
    s^\gamma\|v(s)\|_\beta.
\end{equation}

In particular, for $0<T_0\leq\tau\leq T$ and $t\geq0$ such that
$\tau+t\leq T$, we have
\begin{equation}\label{eq;truncated-Schauder-block}
    (\tau+t)^\gamma
    \left\|
        \int_0^{T_0}
        V(\tau+t-s)v(s)\,\mathrm ds
    \right\|_{\beta+2\sigma}
    \lesssim
    \sup_{s\in[0,T_0]}
    s^\gamma\|v(s)\|_\beta.
\end{equation}
\begin{proof}
By the change of variables $r=t-s$,
\[
    \mathcal V_a(t)
    =
    \int_{t-a}^{t}V(r)v(t-r)\,\mathrm dr.
\]
The proof is the same as that of Lemma~\ref{le;est;Schauder-type}. Indeed, for each dyadic block $\Delta_q$, the integration in the proof of Lemma~\ref{le;est;Schauder-type} is now restricted to the interval $[t-a,t]$. Thus, after splitting at $\eta\in[0,t/2]$, the two integration regions are replaced by
\[
    [0,\eta]\cap[t-a,t]
    \qquad\text{and}\qquad
    [\eta,t]\cap[t-a,t],
\]
respectively. Since all the estimates in the proof of Lemma~\ref{le;est;Schauder-type} are obtained by taking absolute values of the corresponding integrals, restricting the integration domains does not change the estimates. Hence, with
\[
    M_a
    :=
    \sup_{s\in[0,a]}
    s^\gamma\|v(s)\|_\beta,
\]
the same choice of $\eta$ yields
\[
    \|\Delta_q\mathcal V_a(t)\|_{L^\infty}
    \lesssim
    M_a\,2^{-q(\beta+2\sigma)}t^{-\gamma}.
\]
Taking the supremum over $q$ gives
\[
    t^\gamma
    \|\mathcal V_a(t)\|_{\beta+2\sigma}
    \lesssim M_a,
\]
which proves \eqref{eq;truncated-Schauder}. Similarly, since 
\begin{align*}
    \ml V_a(t)-\ml V(s)=\int_0^a[ V(t+a-r)-V(s+a-r)]v(r)\mathrm dr
\end{align*}
 then the estimate is obviously hold.
The estimate \eqref{eq;truncated-Schauder-block} follows by applying
\eqref{eq;truncated-Schauder} with $a=T_0$ and replacing $t$ by
$\tau+t$.
\end{proof}
\end{lemma}

\section{Rough Burgers type equation}\label{sec;Rough Burgers}
In this section, we consider the Burgers type equation
\begin{equation}\label{eq;Burgers type}
    \p_t^\nu u+(-\Delta)^{\sigma} u=\ml N(u)\p_xu+\xi
\end{equation}

\subsection{Construction of
\texorpdfstring{$\vartheta$ and $\vartheta\circ\partial_x\vartheta$}
{theta and its resonant product with its spatial derivative}}
Let $\vartheta$ satisfy the linear equation
\begin{equation}
\p_t^\nu \vartheta+(-\Delta)^\sigma\vartheta=\xi,\quad \vartheta(0)=0.
\end{equation}
Recall the fact $V(t)=t^{\nu-1}\phi_{\theta,1,0}(t^{\nu/2\sigma}D)$, we have the Fourier mode $k$ of $\vartheta$ satisfying that
\begin{equation}
\widehat\vartheta(t,k)=\int_0^t (t-\tau)^{\nu-1} \phi((t-\tau)^{\nu/2}k) \mathrm d\tau
\end{equation}
where we write $\phi=\phi_{\theta,1,0}$ simply. Unlike the case $\nu=1$, the covariance does not admit an elementary closed-form expression. We therefore introduce a decomposition of $\m E[\widehat{\vartheta}(t,k)\widehat{\vartheta}(s,k')]$. Define $I_{t,t'}(f)$ and $II_{t,t'}(f)$ as follows:
\begin{align}
   \label{eq;I} I_{t,t'}(f)=&\int_0^{t\wedge t'}(f(t-s)-f(t'-s))^2\mathrm{d}s\\
   \label{eq;II} II_{t,t'}(f)=&\int_{t\wedge t'}^{t\vee t'}f^2(t\vee t'-s)\mathrm{d}s
\end{align}
The quantities $I_{t,t'}(f)$ and $II_{t,t'}(f)$ satisfy several useful properties.
\begin{itemize}
\item {(1) $I_{t,t}=II_{t,t}=0$.}
\item {(2) $I_{t,t'}=I_{t',t}$, $II_{t,t'}=II_{t',t}$.}
\end{itemize}
Using the above definitions, we can represent the covariance of $\widehat{\vartheta}(t,k)$ as follows.
\begin{lemma}\label{le;covariance}
Let $\widehat{\vartheta}(t,k)=\ml F_x{\vartheta(t,x)}(k)$, $\ml E_{\nu,k}=t^{\nu-1}\phi(t^{\nu/2\sigma}k)$, then for any $0<s,t\leq T$ we have following decomposition
\begin{align*}
    \m E[\widehat{\vartheta}(t,k)\widehat{\vartheta}(s,k')]=\textbf{1}_{k+k'=0}\frac{1}{4\pi}(II_{t\vee s,0}(\ml E_{\nu,k})+II_{t\wedge s,0}(\ml E_{\nu,k})-I_{t,s}(\ml E_{\nu,k})-II_{t,s}(\ml E_{\nu,k}))
\end{align*}
Particularly, if $t=s$, then we have 
\begin{align*}
    \m E[\widehat{\vartheta}(t,k)\widehat{\vartheta}(t,k')]=\frac{\textbf{1}_{k+k'=0}}{2\pi}II_{t,0}(\ml E_{\nu,k})
\end{align*}
Let $\widehat \vartheta_{s,t}(k)=\widehat \vartheta(t,k)-\widehat \vartheta(s,k)$, then we have
    \begin{align*}
   \m E[\widehat \vartheta_{s,t}(k)\widehat \vartheta_{s,t}(k')]=\frac{\textbf{1}_{k+k'=0}}{2\pi}(I_{t,s}(\ml E_{\nu,k})+II_{t,s}(\ml E_{\nu,k})).
    \end{align*}
\end{lemma}
Now our aim is to estimate $\ml E_{\nu,k}$, recall that $\ml E_{\nu,k}=t^{\nu-1}\phi(t^{\nu/2\sigma}k)$, where $\phi$ satisfy Lemma \ref{le;decay;varphi;phi}, by the decay properties of $\phi$ we have the fact for any $\theta\in[0,1]$
\begin{equation}
   \sup_{t>0,k\in\m Z} t^{1+\nu\theta+\nu\theta/\sigma-\nu}|k|^{2\theta+2\sigma\theta}|\ml E_{\nu,k}|\lesssim(1+t^\nu|k|^{2\sigma})^{1/\sigma+1}\phi(t^{\nu/(2\sigma)}k)\lesssim C
\end{equation}
Similarly, noting that $\ml E_{\nu,k}'=t^{\nu-2}\phi_{\theta,1,1}(t^{\nu/(2\sigma)}k)$, we have 
\begin{equation}
   \sup_{t>0,k\in\m Z} t^{2+\nu\theta+\nu\theta/\sigma-\nu}|k|^{2\theta+2\sigma\theta}|\ml E_{\nu,k}'|\lesssim(1+t^\nu|k|^{2\sigma})^{1/\sigma+1}\phi_{\theta,1,1}(t^{\nu/(2\sigma)}k)\lesssim C.
\end{equation}Then we obtain the two condition which Lemma \ref{le;est;I,II} need, assume that $2\nu-2\nu\theta-2\nu\theta/\sigma-1>0$, for any $\kappa\in(0,2\nu-2\nu\theta-2\nu\theta/\sigma-1]$ 
\begin{equation}\label{eq;est;covariance}
    I_{t,s}(\ml E_{\nu,k})+II_{t,s}(\ml E_{\nu,k})+|II_{t,0}(\ml E_{\nu,k})-II_{s,0}(\ml E_{\nu,k})|\lesssim |k|^{-4\theta-4\sigma\theta}T^{2\nu-2\nu\theta-2\nu\theta/\sigma-1-\kappa}|t-s|^\kappa.
\end{equation}
Combining the fractional calculus estimates with the argument in \cite{gubinelli2015paracontrolled}, we obtain the regularity of $\vartheta(t,x)$ in the following proposition   
\begin{proposition}\label{prop;regularity}
Let $\alpha<2\sigma-\sigma/\nu-1/2$ and any $p\geq1$, the process $\vartheta$ satisfies
\begin{equation}
\m E[\Vert \vartheta\Vert_{C_T\ml C^\alpha}^p]<\infty.
\end{equation}
\begin{proof}
By the argument in \cite{gubinelli2015paracontrolled}, it suffices to estimate $\m E[|\Delta_l\vartheta_{s,t}(x)|^2]$, where $\vartheta_{s,t}=\vartheta(t)-\vartheta(s)$. If $l\geq0$, then Fourier inversion, Lemma \ref{le;covariance}, and Lemma \ref{le;decay} imply that
\begin{align*}
\m E[|\Delta_l\vartheta_{s,t}(x)|^2=&(2\pi)^{-2}\sum_{k,k'\in \m Z}\rho_l(k)\rho_l(k')e^{i(k-k')x}\m E[\ml F\vartheta_{s,t}(k)\overline{\ml F\vartheta_{s,t}(k')}]\\
\lesssim &\sum_{k\in\m Z}\rho^2_l(k)(I_{s,t}(\ml E_{v,k})+II_{s,t}(\ml E_{\nu,k})\lesssim |t-s|^\kappa 2^{l(1-4\theta-4\sigma\theta)}
\end{align*}
for $0<\theta<\frac{\sigma(2\nu-1-\kappa)}{2\nu(1+\sigma)}$. Hence, for any $\alpha<2\sigma-\sigma/\nu-1/2$, by choosing $\kappa$ sufficiently small and $\theta$ sufficiently close to $\frac{\sigma(2\nu-1-\kappa)}{2\nu(1+\sigma)}$, and the fact that
\begin{equation}
    \m E[\Vert \Delta_l\vartheta_{s,t}\Vert_{L^{2p}(\m T)}^{2p}]\lesssim \Vert \m E[|\Delta_l\vartheta_{s,t}(x)|^2\Vert_{L^p}^p
\end{equation}
we obtain
\begin{align*}
\m E[\Vert \vartheta(t,\cdot)-\vartheta(s,\cdot)\Vert_{B_{2p,2p}^\alpha}^{2p}]\lesssim \sum_{l\geq-1}2^{l\alpha 2p}2^{l(1-4\theta-4\sigma\theta)p}|t-s|^{\kappa p}\lesssim |t-s|^{\kappa p}
\end{align*}
Taking $p$ sufficiently large, Kolmogorov's continuity criterion implies that $\vartheta$ admits a continuous version satisfying $ \m E[\Vert \vartheta\Vert_{C_TB_{2p,2p}^\alpha}^p]<\infty$. Since $p$ can be chosen arbitrarily large, the desired conclusion follows from the Besov embedding.

\end{proof}
\end{proposition}
Next, we construct the area $\vartheta \circ \partial_{x}\vartheta$ .
\begin{lemma}\label{le;area;1}
   Let $\sigma>3/4$ and $\nu>4\sigma/(8\sigma-3)$, define

$$
\vartheta \circ \partial_{x}\vartheta = \left(\vartheta^{k}\circ \partial_{x}\vartheta^{\ell}\right)_{1\leqslant k,\ell \leqslant n} = \left(\sum_{|i - j|\leqslant 1}\Delta \vartheta^{i}\Delta_{j}\partial_{x}\vartheta^{j}\right)_{1\leqslant k,\ell \leqslant n}.
$$

Then almost surely $\vartheta \circ \partial_{x}\vartheta \in C_{T}\mathcal{C}^{2\alpha - 1}(\mathbb{T}; \mathbb{R}^{n \times n})$ for all $\alpha<2\sigma-\sigma/\nu-1/2$ . Moreover, if $\psi \in \mathcal{S}$ is such that $\int \psi (x) \mathrm{d}x = 1$ and $\vartheta^{\epsilon} = \psi^{\epsilon} * \vartheta$ , where $\psi^{\epsilon} = \epsilon^{- 1}\psi (\epsilon^{- 1})$ , then we have for all $p \geqslant 1$ that

$$
\lim_{\epsilon \to 0}\mathbb{E}\big[\| \vartheta^{\epsilon}\circ \partial_{x}\vartheta^{\epsilon} - \vartheta \circ \partial_{x}\vartheta \|_{C_{T}\mathcal{C}^{2\alpha - 1}}^{p}\big] = 0. \quad (30)
$$
\begin{proof}
  Using a similar argument as in \cite{gubinelli2015paracontrolled}, it suffices to estimate
\[
\mathbb{E}\left[\left|\sum_{|i - j|\leqslant 1}\mathbf{1}_{\ell \leqslant i}\Delta_{\ell}(\Delta_i\vartheta(t,\cdot)\Delta_{j}\partial_x \vartheta_{s,t}^{2})(x)\right|^{2}\right].
\]
By Wick's theorem,
  \begin{align*} \m E[\widehat{\vartheta}(t,k)\widehat{\vartheta}_{s,t}(k')\overline{\widehat{\vartheta}(t,m)}\overline{\widehat{\vartheta}_{s,t}(m')}]=&\m E[\widehat{\vartheta}(t,k)\widehat{\vartheta}_{s,t}(k')]\m E[\overline{\widehat{\vartheta}(t,m)}\overline{\widehat{\vartheta}_{s,t}(m')}]\\ &+\m E[\widehat{\vartheta}(t,k)\overline{\widehat{\vartheta}_{s,t}(m')}]\m E[\widehat{\vartheta}_{s,t}(k')\overline{\widehat{\vartheta}(t,m)}]\\ &+\m E[\widehat{\vartheta}(t,k)\overline{\widehat{\vartheta}(t,m)}]\m E[\widehat{\vartheta}_{s,t}(k')\overline{\widehat{\vartheta}_{s,t}(m')}]\\ =&I_1+I_2+I_3 \end{align*}
We then expand
\begin{align*}
    &\m E\left[\left|\sum_{|i-j|\leq 1}\textbf{1}_{i\gtrsim l}\Delta_l\left(\Delta_i\vartheta(t)\Delta_j(\p_x\vartheta_{s,t})\right)(x)\right|^2\right]\\
    =&\sum_{|i-j|\leq 1}\sum_{|i'-j'|\leq 1}\textbf{1}_{i\gtrsim l}\textbf{1}_{i'\gtrsim l}
    \sum_{k,k'\in\m Z}\sum_{m,m'\in\m Z}
    \rho_l(k+k')\rho_l(m+m')\rho_i(k)\rho_j(k')\rho_{i'}(m)\rho_{j'}(m')\\
    &\times k'm'\m E\left[
    \widehat{\vartheta}(t,k)
    \widehat{\vartheta}_{s,t}(k')
    \overline{\widehat{\vartheta}(t,m)}
    \overline{\widehat{\vartheta}_{s,t}(m')}
    \right]=A_1+A_2+A_3
\end{align*}
where, for $r=1,2,3$, $A_r$ denotes the contribution corresponding to the Wick contraction $I_r$. We first observe that the contribution of $I_1$ to the above sum vanishes. Indeed, by Lemma \ref{le;covariance}, we have
\begin{equation}
    I_1=\textbf{1}_{k+k'=0}\textbf{1}_{m+m'=0}C(k)C(m)
\end{equation}
for some even function $C$. Hence, the contribution of $I_1$ is given by
\begin{align*}
    A_1
    =&\sum_{|i-j|\leq 1}\sum_{|i'-j'|\leq 1}
    \textbf{1}_{i\gtrsim l}\textbf{1}_{i'\gtrsim l}
    \sum_{k,m\in\m Z}
    \rho_l(0)^2\rho_i(k)\rho_j(-k)\rho_{i'}(m)\rho_{j'}(-m)\\
    &\times km C(k)C(m).
\end{align*}
Since $C$ and the Littlewood--Paley multipliers are even, we have
\begin{align*}
    \sum_{k\in\m Z}k\rho_i(k)\rho_j(-k)C(k)=0.
\end{align*}
Therefore,
\begin{equation}
    A_1=0.
\end{equation}
 For our convenience, define $I_{s,t}(\ml E_{\nu,k})=I_{s,t}^k$ and $II_{s,t}((\ml E_{\nu,k}))=II^k_{s,t}$ and assume that $0\leq s\leq t\leq T$. By Lemma \ref{le;covariance} with a directly computation, we obtain that
     \begin{align*}
            I_2=& \frac{1}{4}\textbf{1}_{k=m'}\textbf{1}_{k'=m}(II_{t,0}^k-II_{s,0}^k+I_{t,s}^k+II_{t,s}^k)(II_{t,0}^m-II_{s,0}^m+I_{t,s}^m+II_{t,s}^m)
        \end{align*}
        then by \eqref{eq;est;covariance}, for $\kappa\in(0,2\nu-2\nu\theta-2\nu\theta/\sigma-1)$, we have 
        \begin{align*}
        I_2\lesssim T^{4\nu-4\nu\theta-4\nu\theta/\sigma-2-2\kappa}|k|^{-4\theta-4\sigma\theta}|m|^{-4\theta-4\sigma\theta}|t-t'|^{2\kappa}
        \end{align*}
        For $I_3$, noting that $I_3=\textbf{1}_{k=m}\textbf{1}_{k'=m'}II_{t,0}^k(I_{t,s}^{k'}+II_{t,s}^{k'})$, by \eqref{eq;est;covariance}, for $\theta<\frac{\sigma(2\nu-1)}{2\nu(1+\sigma)}$ and $\epsilon\in(0,2\nu-2\nu\theta-2\nu\theta/\sigma-1]$we have
        \begin{align*}
            I_3\lesssim T^{2\nu-2\nu\theta-2\nu\theta/\sigma-1} |k|^{-4\theta-4\sigma\theta}T^{2\nu-2\nu\theta-2\nu\theta/\sigma-1-\epsilon} |k'|^{-4\theta-4\sigma\theta}|t-s|^\epsilon
        \end{align*}
        assume that $\kappa<\nu-\nu\theta-\nu\theta/\sigma-1/2$ and let $\epsilon=2\kappa$, combining above estimates, we have
      \begin{align*}
    I_2+I_3
    \lesssim
    T^{4\nu-4\nu\theta-4\nu\theta/\sigma-2-2\kappa}
    |k|^{-4\theta-4\sigma\theta}|k'|^{-4\theta-4\sigma\theta}|t-s|^{2\kappa}.
\end{align*}
Consequently,
\begin{align*}
    A_2+A_3
    \lesssim&
    \sum_{|i-j|\leq 1}\sum_{|i'-j'|\leq 1}
    \textbf{1}_{l\lesssim i}\textbf{1}_{l\lesssim i'}
    \sum_{k,k'\in\m Z}
    \rho_l(k+k')^2
    \rho_i(k)\rho_{i'}(k)
    \rho_j(k')\rho_{j'}(k')\\
    &\times
    |k|^{1-4\theta-4\sigma\theta}|k'|^{1-4\theta-4\sigma\theta}|t-s|^{2\kappa}\\
    \lesssim&
    \sum_{|i-j|\leq 1}
    \textbf{1}_{l\lesssim i}
    \sum_{k\in\operatorname{supp}\rho_i,\,
    k'\in\operatorname{supp}\rho_j}
    \rho_l(k+k')^2
    |k|^{1-4\theta-4\sigma\theta}|k'|^{1-4\theta-4\sigma\theta}
    |t-s|^{2\kappa}.
\end{align*}
Since $|i-j|\leq 1$ and $l\lesssim i$, we have $|k|\sim |k'|\sim 2^i$ on the support of the summand. Moreover, for each fixed $k$, the condition $\rho_l(k+k')\neq 0$ restricts $k'$ to a set of cardinality of order $2^l$. Hence,
\begin{align*}
    A_2+A_3
    \lesssim&
    \sum_{i\gtrsim l}
    2^i2^l2^{2i(1-4\theta-4\sigma\theta)}
    |t-s|^{2\kappa}  \lesssim
    \sum_{i\gtrsim l}
    2^{l+(3-8\theta-8\sigma\theta)i}
    |t-s|^{2\kappa}\lesssim |t-s|^{2\kappa}2^{(4-8\theta-8\sigma\theta)l}
\end{align*}
for $\theta\in(\frac{3}{8(1+\sigma)},\frac{\sigma(2\nu-1-2\kappa)}{2\nu(1+\sigma)})$, which is possible provided that $\sigma>3/4$ and $\nu>4\sigma/(8\sigma-3)$. By the same argument, we also have
\begin{align*}
    \m E[\vert \sum_{\vert i-j\vert \leq 1}&\textbf{1}_{i\gtrsim l}\Delta_l (\Delta_i\vartheta_{s,t}\Delta_j(\p_x\vartheta(s,\cdot)))\vert^2]
    \lesssim |t-s|^{2\kappa}2^{(4-8\theta-8\sigma\theta)l}.
\end{align*}
Therefore, for sufficiently small $\kappa>0$ and arbitrarily large $p\geq 1$, by Gaussian hypercontractivity, we obtain
\begin{equation}
    \m E[\Vert \Delta_l (\vartheta\circ\p_x\vartheta)_{s,t}\Vert^{2p}_{L^{2p}(\m T)}]
    \lesssim |t-s|^{2\kappa p}2^{(4-8\theta-8\sigma\theta)lp}.
\end{equation}
Hence, for any $\gamma<4\theta+4\sigma\theta-2$,
\begin{align*}
    \m E[\Vert( \vartheta\circ\p_x\vartheta)_{s,t}\Vert_{B_{2p,2p}^{\gamma}}^{2p}]
    =&\sum_{l\geq -1}2^{2p\gamma l}
    \m E[\Vert \Delta_l(\vartheta\circ\p_x\vartheta)_{s,t}\Vert_{L^{2p}(\m T)}^{2p}]\\
    \lesssim& |t-s|^{2\kappa p}
    \sum_{l\geq -1}2^{(2\gamma+4-8\theta-8\sigma\theta)lp}\\
    \lesssim& |t-s|^{2\kappa p}.
\end{align*}
Choosing $p$ sufficiently large such that $2\kappa p>1$, Kolmogorov's continuity criterion yields a continuous modification of $\vartheta\circ\p_x\vartheta$ with values in $B_{2p,2p}^{\gamma}$ for any $\gamma<4\theta+4\sigma\theta-2$. Taking $p$ sufficiently large and using the Besov embedding, we obtain
\[
    \vartheta\circ\p_x\vartheta\in C_T\ml C^{\gamma-\frac{1}{2p}}
\]
almost surely. Since $\theta$ can be chosen arbitrarily close to $\frac{\sigma(2\nu-1)}{2\nu(1+\sigma)}$, for every $\alpha<2\sigma-\sigma/\nu-1/2$, we may choose $\theta$, $\gamma$, and $p$ such that
\[
    2\alpha-1<\gamma-\frac{1}{2p}<4\theta+4\sigma\theta-2.
\]
Consequently,
\begin{equation}
    \vartheta\circ\p_x\vartheta\in C_T\ml C^{2\alpha-1}
\end{equation}
almost surely.
      
\end{proof}
\end{lemma}

\begin{remark}
As a consequence of Lemma \ref{le;area;1}, the family
$\vartheta^\varepsilon\circ\partial_x\vartheta^\varepsilon$ is uniformly bounded in probability as $\varepsilon\to0$. More precisely, for every $\delta>0$, there exists $\varepsilon_\delta>0$ such that, for all $0<\varepsilon<\varepsilon_\delta$,
\begin{align*}
    &\mathbb{P}\left(
        \left\|
            \vartheta^\varepsilon\circ\partial_x\vartheta^\varepsilon
        \right\|_{C_T\mathcal{C}^{2\alpha-1}}
        \geq
        1+
        \left\|
            \vartheta\circ\partial_x\vartheta
        \right\|_{C_T\mathcal{C}^{2\alpha-1}}
    \right)
    \\
    &\qquad\leq
    \mathbb{P}\left(
        \left\|
            \vartheta^\varepsilon\circ\partial_x\vartheta^\varepsilon
            -
            \vartheta\circ\partial_x\vartheta
        \right\|_{C_T\mathcal{C}^{2\alpha-1}}
        \geq 1
    \right)
    \\
    &\qquad\leq
    \mathbb{E}\left[
        \left\|
            \vartheta^\varepsilon\circ\partial_x\vartheta^\varepsilon
            -
            \vartheta\circ\partial_x\vartheta
        \right\|_{C_T\mathcal{C}^{2\alpha-1}}^p
    \right]
    <\delta,
\end{align*}
where the last inequality follows from Lemma \ref{le;area;1} by taking $\varepsilon_\delta>0$ sufficiently small. Consequently,
\begin{align}
    \mathbb{P}\left(
        \left\|
            \vartheta^\varepsilon\circ\partial_x\vartheta^\varepsilon
        \right\|_{C_T\mathcal{C}^{2\alpha-1}}
        \leq
        1+
        \left\|
            \vartheta\circ\partial_x\vartheta
        \right\|_{C_T\mathcal{C}^{2\alpha-1}}
    \right)
    \geq 1-\delta
    \label{eq:area-high-probability}
\end{align}
for every $0<\varepsilon<\varepsilon_\delta$.
In particular, the above estimate should be understood in the probabilistic sense: with arbitrarily high probability and uniformly for all sufficiently small $\varepsilon$,
\begin{align*}
    \left\|
        \vartheta^\varepsilon\circ\partial_x\vartheta^\varepsilon
    \right\|_{C_T\mathcal{C}^{2\alpha-1}}
    \leq
    1+
    \left\|
        \vartheta\circ\partial_x\vartheta
    \right\|_{C_T\mathcal{C}^{2\alpha-1}}.
\end{align*}
Equivalently,
\begin{align*}
    \lim_{\varepsilon_0\downarrow0}
    \sup_{0<\varepsilon<\varepsilon_0}
    \mathbb{P}\left(
        \left\|
            \vartheta^\varepsilon\circ\partial_x\vartheta^\varepsilon
        \right\|_{C_T\mathcal{C}^{2\alpha-1}}
        >
        1+
        \left\|
            \vartheta\circ\partial_x\vartheta
        \right\|_{C_T\mathcal{C}^{2\alpha-1}}
    \right)
    =0.
\end{align*}
\end{remark}
\subsection{Local well-posedness}
We now establish the well-posedness of \eqref{eq;Burgers type} by approximation. Let $L=\p_t^\nu-\Delta$, $N\in C_b^3$ and $\xi^\varepsilon$ be a smooth approximation of $\xi$. For each $\varepsilon>0$, the regularized equation
\begin{equation}
    L u^\varepsilon
    =
    \ml N(u^\varepsilon)\p_xu^\varepsilon+\xi^\varepsilon
\end{equation}
is a classical equation and admits a unique solution $u^\varepsilon$. To pass to the limit as $\varepsilon\to0$, it remains to derive estimates for $u^\varepsilon$ which are uniform in $\varepsilon$ and depend only on the enhanced data associated with $\vartheta^\varepsilon$. Following the strategy of \cite{gubinelli2015paracontrolled}, let $\vartheta^\varepsilon$ solve the corresponding linear equation and write
\begin{equation}
    w^\varepsilon=u^\varepsilon-\vartheta^\varepsilon.
\end{equation}
Then $w^\varepsilon$ satisfies
\begin{equation}
    Lw^\varepsilon
    =
    \ml N(\vartheta^\varepsilon+w^\varepsilon)
    \p_x(\vartheta^\varepsilon+w^\varepsilon).
\end{equation}
Let $\sigma\in(5/6,1]$ and $\nu\in(6\sigma/(12\sigma-5),1]$, $1/3<\alpha<2\sigma-\sigma/\nu-1/2$, by Proposition \ref{prop;regularity} and Lemma \ref{le;area;1}, $(\vartheta^\varepsilon,\vartheta^\varepsilon\circ\p_x\vartheta^\varepsilon)$ is uniformly bounded in $C_T\ml C^\alpha\times C_T\ml C^{2\alpha-1}$. We firstly show that the term
\begin{equation}
    \Phi^\varepsilon
    =
    \ml N(\vartheta^\varepsilon+w^\varepsilon)\p_x\vartheta^\varepsilon
\end{equation}
can be controlled by $w^\varepsilon$. Using the paraproduct decomposition,
\begin{align*}
    \Phi^\varepsilon
    =&
    \ml N(\vartheta^\varepsilon+w^\varepsilon)\prec\p_x\vartheta^\varepsilon
    +\ml N(\vartheta^\varepsilon+w^\varepsilon)\succ\p_x\vartheta^\varepsilon\\
    &+
    \ml N(\vartheta^\varepsilon+w^\varepsilon)\circ\p_x\vartheta^\varepsilon.
\end{align*}
The first two terms is uniformly bounded in $C_T\ml C^{\alpha-1}$ by the standard paraproduct estimates. The only singular contribution is the resonant term. By the paralinearization lemma, we have
\begin{align*}
    \ml N(\vartheta^\varepsilon+w^\varepsilon)
    =
    \ml N'(\vartheta^\varepsilon+w^\varepsilon)
    \prec\vartheta^\varepsilon
    +
    R_\ml N(\vartheta^\varepsilon+w^\varepsilon).
\end{align*}
Hence
\begin{align*}
\ml N(\vartheta^\varepsilon+w^\varepsilon)\circ\p_x\vartheta^\varepsilon=&\ml N'(\vartheta^\varepsilon+w^\varepsilon)\prec \vartheta^\varepsilon \circ\p_x\vartheta^\varepsilon+R_\ml N(\vartheta^\varepsilon+w^\varepsilon)\circ\p_x\vartheta^\varepsilon\\
    =&C(\ml N'(\vartheta^\varepsilon+w^\varepsilon),\vartheta^\varepsilon,\p_x\vartheta^\varepsilon)+\ml N'(\vartheta^\varepsilon+w^\varepsilon)\vartheta^\varepsilon\circ\p_x\vartheta^\varepsilon+R_\ml N(\vartheta^\varepsilon+w^\varepsilon)\circ\p_x\vartheta^\varepsilon.
\end{align*}
Since $3\alpha-1>0$, using Lemma \ref{le;est;paraproduct}, \ref{le;para-controlled}, and \ref{le;paralinearization}, we obtain
\begin{equation}\label{eq;Phi}
    \Vert\Phi^\varepsilon\Vert_{\alpha-1}
    \lesssim
    C_{\ml N}C_{\vartheta^\varepsilon}
    \left(
    1+\Vert w^\varepsilon\Vert_{\alpha+1}
    \right),
\end{equation}
where we set $C_{\ml N}=\Vert \ml N\Vert_{C_b^3}$
\begin{align*}
   C_{\vartheta^\varepsilon}=1+\Vert\vartheta^\varepsilon\Vert_{C_T\ml C^\alpha}^3+\Vert \vartheta^\varepsilon\Vert_{C_T\ml C^\alpha}\Vert \vartheta^\varepsilon\circ\p_x\vartheta^\varepsilon\Vert_{C_T\ml C^{2\alpha-1}}
\end{align*}
We next derive an a priori estimate for $w^\varepsilon$. Writing
\begin{equation}
    Lw^\varepsilon
    =
    \Phi^\varepsilon
    +
    \ml N(\vartheta^\varepsilon+w^\varepsilon)\p_xw^\varepsilon,
\end{equation}
the corresponding integral equation is
\begin{equation}
    w^\varepsilon(t)
    =
    U(t)u_0
    +
    \int_0^tV(t-s)\Phi^\varepsilon(s)\mathrm ds
    +
    \int_0^tV(t-s)
    \left(
    \ml N(\vartheta^\varepsilon(s)+w^\varepsilon(s))
    \p_xw^\varepsilon(s)
    \right)\mathrm ds.
\end{equation}
Applying the Schauder-type estimates for the time-fractional heat flow in Lemma \ref{le;est;fractional heat flow;linear}, we obtain
 \begin{align*}
     \Vert w^\varepsilon(t)\Vert_{\alpha-1+2\sigma}
     \lesssim& t^{-\nu(2\sigma-1)/(2\sigma)}(\Vert u_0\Vert_{\alpha}+\sup_{s\in[0,t]}s^{\nu(2\sigma-1)/(2\sigma)}\Vert \Phi^\varepsilon\Vert_{\alpha-1})\\
     &+\int_0^t \frac{\Vert \ml N(\vartheta^\varepsilon+w^\varepsilon)\p_xw^\varepsilon\Vert_{L^\infty}}{(t-s)^{1-\nu+\nu(\alpha-1+2\sigma)/(2\sigma)}}\mathrm ds
 \end{align*}
 Since $\alpha-2+2\sigma>0$,
\begin{align*}
    \Vert
    \ml N(\vartheta^\varepsilon+w^\varepsilon)
    \p_xw^\varepsilon
    \Vert_{L^\infty}
    \lesssim
    C_{\ml N}\Vert w^\varepsilon\Vert_{\alpha-1+2\sigma}.
\end{align*}
Therefore,
\begin{align*}
    t^{\nu(2\sigma-1)/(2\sigma)}
    &\int_0^t
    \frac{
    \Vert
    \ml N(\vartheta^\varepsilon+w^\varepsilon)
    \p_xw^\varepsilon
    \Vert_{L^\infty}}
    {(t-s)^{1-\nu+\nu(\alpha-1+2\sigma)/(2\sigma)}}
    \mathrm ds\\
    \lesssim&
    C_{\ml N}
    \left(
    \sup_{s\in[0,t]}
    s^{\nu(2\sigma-1)/(2\sigma)}
    \Vert w^\varepsilon(s)\Vert_{\alpha-1+2\sigma}
    \right)
    t^{\nu-\nu(\alpha-1+2\sigma)/(2\sigma)}.
\end{align*}
Combining the above estimates, we conclude that
\begin{align*}
    \sup_{s\in[0,t]}
    s^{\nu(2\sigma-1)/(2\sigma)}
    \Vert w^\varepsilon(s)\Vert_{\alpha-1+2\sigma}
    \lesssim&
    \Vert u_0\Vert_\alpha
    +
    C_{\ml N}C_{\vartheta^\varepsilon}
    \left(
    1+
    \sup_{s\in[0,t]}
    s^{\nu(2\sigma-1)/(2\sigma)}
    \Vert w^\varepsilon(s)\Vert_{\alpha-1+2\sigma}
    \right)
\end{align*}
Using a similar argument, we also obtain, uniformly in $t\in[0,T]$,
\begin{align}
    \|w^\varepsilon(t)\|_\alpha
    \lesssim
    \|u_0\|_\alpha
    +
    C_{\ml N}C_{\vartheta^\varepsilon}
    \left(
        1+
        \sup_{s\in(0,t]}
         s^{\nu(2\sigma-1)/(2\sigma)}\|w^\varepsilon(s)\|_{\alpha-1+2\sigma}
    \right).
\end{align}
Define $C_\vartheta=$
We now use a scaling argument to close the above estimates. Since the following argument is carried out for each fixed $\varepsilon>0$, we suppress the superscript $\varepsilon$ throughout and write
\[
    (u,w,\vartheta,\Phi)
    =
    (u^\varepsilon,w^\varepsilon,\vartheta^\varepsilon,\Phi^\varepsilon)
\]
for simplicity. The dependence on $\varepsilon$ will be restored at the end of the argument.

For $\lambda\in(0,1]$, define
\begin{align*}
    \Lambda_\lambda u(t,x)
    :=
    u(\lambda^{2\sigma/\nu}t,\lambda x).
\end{align*}
Then
\begin{align*}
    L\Lambda_\lambda w
    =
    \lambda^{2\sigma}\Lambda_\lambda Lw.
\end{align*}
Set
\begin{align*}
    w^\lambda=\Lambda_\lambda w,
    \qquad
    \vartheta^\lambda=\Lambda_\lambda\vartheta.
\end{align*}
Since
\begin{align*}
    Lw
    =
    \Phi
    +
    \ml N(\vartheta+w)\p_xw,
\end{align*}
we obtain
\begin{align*}
    Lw^\lambda
    &=
    \lambda^{2\sigma}\Lambda_\lambda\Phi
    +
    \lambda^{2\sigma-1}
    \ml N(\vartheta^\lambda+w^\lambda)
    \p_xw^\lambda.
\end{align*}

By \cite[Lemmas A.4 and B.1]{gubinelli2015paracontrolled}, the scaling properties of Besov norms and resonant products yield
\begin{align*}
    C_{\vartheta^\lambda}
    \lesssim
    \lambda^{-1}C_\vartheta,
    \qquad
    \lambda\in(0,1],
\end{align*}
where $C_{\vartheta^\lambda}$ denotes the control quantity associated with the rescaled noise $\vartheta^\lambda$. Moreover, the same argument as in the derivation of \eqref{eq;Phi} gives
\begin{align*}
    \|\Lambda_\lambda\Phi(t)\|_{\alpha-1}
    \lesssim
    C_{\ml N}C_{\vartheta^\lambda}
    \left(
        1+\|w^\lambda(t)\|_{\alpha-1+2\sigma}
    \right).
\end{align*}
Consequently,
\begin{align*}
    \lambda^{2\sigma}
    \|\Lambda_\lambda\Phi(t)\|_{\alpha-1}
    \lesssim
    \lambda^{2\sigma-1} C_{\ml N}C_\vartheta
    \left(
        1+\|w^\lambda(t)\|_{\alpha-1+2\sigma}
    \right).
\end{align*}
Combining this estimate with the contribution of
$\lambda^{2\sigma-1}\ml N(\vartheta^\lambda+w^\lambda)\p_xw^\lambda$, we obtain
\begin{align*}
    \sup_{s\in(0,t]}
    s^{\nu(2\sigma-1)/(2\sigma)}\|w^\lambda(s)\|_{\alpha-1+2\sigma}
    \lesssim&
    \|\Lambda_\lambda u_0\|_\alpha
    \\
    &+
    \lambda^{2\sigma-1} C_{\ml N}(1+C_\vartheta)
    \left(
        1+
        \sup_{s\in(0,t]}
        s^{\nu(2\sigma-1)/(2\sigma)}\|w^\lambda(s)\|_{\alpha-1+2\sigma}
    \right).
\end{align*}
Since $\lambda\in(0,1]$ and $\alpha>0$, we have
\begin{align*}
    \|\Lambda_\lambda u_0\|_\alpha
    \sim
    \lambda^\alpha\|u_0\|_\alpha.
\end{align*}
Choosing $\lambda>0$ sufficiently small, depending only on
$C_{\ml N}$ and $C_\vartheta$, the last term can be absorbed into the left-hand side. Hence,
\begin{equation}\label{eq;est;scaling}
    \sup_{s\in(0,T]}
    s^{\nu(2\sigma-1)/(2\sigma)}\|w^\lambda(s)\|_{\alpha-1+2\sigma}
    \lesssim
    \|u_0^\lambda\|_\alpha+1.
\end{equation}
Together with the corresponding estimate in $C_T\ml C^\alpha$, this gives
\begin{equation}\label{eq;w;uniform}
    \|w^\lambda\|_{C_T\ml C^\alpha}
    +
    \sup_{s\in(0,T]}
    s^{\nu(2\sigma-1)/(2\sigma)}\|w^\lambda(s)\|_{\alpha-1+2\sigma}
    \lesssim
    \|u_0^\lambda\|_\alpha+1.
\end{equation}
Scaling back gives that
\begin{align*}
    \|w\|_{C_{\lambda^{2\sigma/\nu }T}\ml C^\alpha}
    +
    \sup_{s\in(0,\lambda^{2\sigma/\nu }T]}
    s^{\nu(2\sigma-1)/(2\sigma)}\|w(s)\|_{\alpha-1+2\sigma}
    \lesssim
    \|u_0\|_\alpha+C_\lambda.
\end{align*}
since $u=w+\vartheta$, we get 
\begin{align*}
    \sup_{t\in[0,\lambda ^{2\sigma/\nu}T]}\Vert u\Vert_\alpha\lesssim\Vert u_0\Vert_\alpha+C_{\vartheta,\ml N}
\end{align*}
We now restore the superscript $\varepsilon$. By the above estimate, the family $(u^\varepsilon)_{\varepsilon>0}$ admits a uniform bound in $C_T\mathcal C^\alpha$ with arbitrarily high probability for all sufficiently small $\varepsilon$. Since, for $\mathcal N\in C_b^3$, the solution map
\[
    (u_0,\vartheta,\vartheta\circ\partial_x\vartheta)
    \longmapsto u
\]
is locally Lipschitz continuous on bounded sets, and the enhanced noises
\[
    \bigl(
        \vartheta^\varepsilon,
        \vartheta^\varepsilon\circ\partial_x\vartheta^\varepsilon
    \bigr)
\]
form a Cauchy sequence in $C_T\ml C^{\alpha}\times C_T\ml C^{2\alpha-1}$, we conclude that $(u^\varepsilon)_{\varepsilon>0}$ is a Cauchy sequence in $C_T\mathcal C^\alpha$ in probability.
\begin{theorem}\label{thm;main;local;Burgers}
Let $\sigma\in(5/6,1]$, $\nu\in(6\sigma/(12\sigma-5),1]$ and $1/3<\alpha<2\sigma-\sigma/\nu-1/2$. Let $G\in C_b^3$ and $u_0\in\mathcal C^\alpha$. Let $(\xi^\varepsilon)_{\varepsilon>0}$ be a family of smooth approximations of $\xi$. For each $\varepsilon>0$, let $u^\varepsilon$ be the classical solution to the regularized equation
\begin{equation}
    Lu^\varepsilon
    =
    \mathcal N(u^\varepsilon)\partial_xu^\varepsilon+\xi^\varepsilon,
    \qquad
    u^\varepsilon(0)=u_0.
\end{equation}
Then there exists $T>0$ such that the equation
\begin{equation}
    Lu
    =
    \mathcal N(u)\partial_xu+\xi,
    \qquad
    u(0)=u_0,
\end{equation}
admits a unique weak solution $u$ on $[0,T]$, in the following sense:
\begin{equation}
    u^\varepsilon\longrightarrow u
    \qquad
    \text{in probability in } C_T\mathcal C^\alpha.
\end{equation}
\end{theorem}
\subsection{Global well-posedness}
In \cite{gubinelli2015paracontrolled}, Gubinelli et al. establish global well-posedness by successively extending the solution from a short time interval to the whole interval $[0,T]$. The key point is that the scaling parameter $\lambda$ can be chosen independently of the initial condition $u_0$, so that the same local existence time can be used after each restart. Moreover, the semigroup property of
\[
    P(t)=e^{-t(-\Delta)^\sigma}
\]
allows the equation to be restarted at any positive time. Indeed, suppose that $w$ satisfies
\[
    w(t)
    =
    P(t)u_0
    +
    \int_0^tP(t-s)\Psi(w(s))\,\mathrm ds.
\]
Fix $t>0$ and let $\widetilde w$ be the solution starting from $w(t)$,
\[
    \widetilde w(t_0)
    =
    P(t_0)w(t)
    +
    \int_0^{t_0}P(t_0-s)\Psi(\widetilde w(s))\,\mathrm ds.
\]
By the semigroup property, the time-shifted solution $w(t+\cdot)$ satisfies the same equation with initial value $w(t)$. Hence, by uniqueness,
\[
    \widetilde w(t_0)=w(t+t_0).
\]
Since the choice of $\lambda$, and hence the length of the local existence interval, is independent of the initial condition, this argument can be iterated to extend the solution to $[0,T]$.

In the present setting, however, $U$ and $V$ do not satisfy the semigroup property, and hence the solution cannot be restarted exactly. We instead to control the increment $w_\lambda(t+T)-w_\lambda(T)$. 
Subtracting the mild formulations at times $T+t$ and $T$,
we obtain
\begin{align*}
    w_\lambda(T+t)-w_\lambda(T)
    ={}&
    [U(T+t)-U(T)]u_0^\lambda\\
    &+
    \int_0^T
    [V(T+t-s)-V(T-s)]
    \Bigl(
        \Phi_\lambda(s)
        +
        \ml N(\vartheta_\lambda(s)+w_\lambda(s))
        \p_x w_\lambda(s)
    \Bigr)\,\mathrm ds\\
    &+
    \int_0^t V(t-s)
    \Bigl(
        \Phi_\lambda(T+s)
        +
        \ml N(\vartheta_\lambda(T+s)+w_\lambda(T+s))
        \p_x w_\lambda(T+s)
    \Bigr)\,\mathrm ds\\
    =:{}&
    R_T^1(t)+R_T^2(t)+R_T^3(t).
\end{align*}
The increment estimates for $U$ give
\[
    \sup_{t\in(0,T]}t^\gamma
    \|R_T^1(t)\|_{\alpha-1+2\sigma}
    +
    \|R_T^1\|_{C_T\ml C^\alpha}
    \lesssim
    1+\|u_0^\lambda\|_\alpha.
\]
By Lemma~\ref{le;est;truncated-Schauder} and
\eqref{eq;w;uniform}, the history term satisfies
\[
    \sup_{t\in(0,T]}t^\gamma
    \|R_T^2(t)\|_{\alpha-1+2\sigma}
    +
    \|R_T^2\|_{C_T\ml C^\alpha}
    \lesssim
    \lambda^{2\sigma-1}C_{\ml N}(1+C_\vartheta)
    (1+\|u_0^\lambda\|_\alpha).
\]
For the integral over the new interval, the local
Schauder estimates yield
\[
    \begin{aligned}
        &\sup_{t\in(0,T]}t^\gamma
        \|R_T^3(t)\|_{\alpha-1+2\sigma}
        +
        \|R_T^3\|_{C_T\ml C^\alpha}\\
        &\qquad\lesssim
        \lambda^{2\sigma-1}C_{\ml N}(1+C_\vartheta)
        \left(
            1+
            \sup_{s\in(0,T]}s^\gamma
            \|w_\lambda(T+s)\|_{\alpha-1+2\sigma}
        \right).
    \end{aligned}
\]
Moreover, the bound on the preceding interval implies
\begin{align*}
    &\sup_{s\in(0,T]}s^\gamma
    \|w_\lambda(T+s)\|_{\alpha-1+2\sigma}\\
    &\qquad\leq
    \sup_{s\in(0,T]}s^\gamma
    \|w_\lambda(T+s)-w_\lambda(T)\|_{\alpha-1+2\sigma}
    +
    T^\gamma\|w_\lambda(T)\|_{\alpha-1+2\sigma}\\
    &\qquad\lesssim
    \sup_{s\in(0,T]}s^\gamma
    \|w_\lambda(T+s)-w_\lambda(T)\|_{\alpha-1+2\sigma}
    +
    1+\|u_0^\lambda\|_\alpha.
\end{align*}
Combining the preceding estimates, we obtain
\begin{align*}
    &\sup_{t\in(0,T]}t^\gamma
    \|w_\lambda(T+t)-w_\lambda(T)\|_{\alpha-1+2\sigma}
    +
    \|w_\lambda(T+\cdot)-w_\lambda(T)\|_{C_T\ml C^\alpha}\\
    &\qquad\lesssim
    1+\|u_0^\lambda\|_\alpha\\
    &\qquad\quad+
    \lambda^{2\sigma-1}C_{\ml N}(1+C_\vartheta)
    \left(
        1+\|u_0^\lambda\|_\alpha
        +
        \sup_{s\in(0,T]}s^\gamma
        \|w_\lambda(T+s)-w_\lambda(T)\|_{\alpha-1+2\sigma}
    \right).
\end{align*}
By our choice of $\lambda$, the last term involving
the increment can be absorbed into the left-hand side.
Consequently,
\[
    \begin{aligned}
        &\sup_{t\in(0,T]}t^\gamma
        \|w_\lambda(T+t)-w_\lambda(T)\|_{\alpha-1+2\sigma}\\
        &\qquad+
        \|w_\lambda(T+\cdot)-w_\lambda(T)\|_{C_T\ml C^\alpha}
        \lesssim
        1+\|u_0^\lambda\|_\alpha.
    \end{aligned}
\]
Together with the bounds on $[0,T]$, this yields
\[
    \sup_{t\in(0,T]}t^\gamma
    \|w_\lambda(T+t)\|_{\alpha-1+2\sigma}
    +
    \|w_\lambda(T+\cdot)\|_{C_T\ml C^\alpha}
    \lesssim
    1+\|u_0^\lambda\|_\alpha.
\]
Hence 
\begin{align*}
    \Vert w_\lambda(t)\Vert_{C_{2T}\ml C^\alpha}\lesssim 1+\Vert u_0^\lambda\Vert_\alpha
\end{align*}
Thus, the estimates extend from $[0,T]$ to $[0,2T]$.
Iterating this argument over successive intervals yields
the uniform bound
\begin{equation}
    \|w_\lambda\|_{C_{\lambda^{-2\sigma/\nu}T}\ml C^\alpha}
    \lesssim 1+\|u_0^\lambda\|_\alpha.
\end{equation}
Scaling back, we obtain
\begin{equation}
    \|w\|_{C_T\ml C^\alpha}
    \lesssim 1+\|u_0\|_\alpha.
\end{equation}
Together with the decomposition $u=\vartheta+w$ and the
bounds for $\vartheta$, this provides the a priori
control required to extend the local solution globally.
We summarize the result in the following theorem.

\begin{theorem}[Global existence]\label{thm;main;global;Burgers}
    The solution constructed in
    Theorem~\ref{thm;main;local;Burgers}
    extends uniquely to a global-in-time solution.
\end{theorem}
\section{Time-fractional heat equation with multiplicative noise}\label{sec;time-fractional}
In this section, we consider the Time-fractional heat equation with multiplicative noise on $\m T^2$
\begin{equation}\label{eq;Anderson model}
    \p_t^\nu u-\Delta u=F(u)\xi.
\end{equation}
where $\xi$ is a spatial white noise. Similarly, we firstly construct the area 
\subsection{Construction of
\texorpdfstring{$\vartheta\diamond\xi$}{theta diamond xi}}
Let $\alpha\in(2/3,1)$ and let $\delta>0$ be such that
$\alpha+\delta<1$. In dimension $d=2$, we have
\begin{equation}
    \xi\in\mathcal C^{\alpha+\delta-2}.
\end{equation}
Let $\Pi$ denote the projection onto the nonzero Fourier modes, so that
\begin{equation}
    \xi=\Pi\xi+\mathcal F(\xi)(0).
\end{equation}
Consider the linear convolution
\begin{equation}
    \vartheta
    :=
    \int_0^\infty V(t)\Pi\xi\,\mathrm dt.
\end{equation}
Then, for every $t>0$,
\begin{align}\label{eq;difference}
    \vartheta-I\xi(t)
    &=
    \int_t^\infty V(\tau)\Pi\xi\,\mathrm d\tau
    -
    \int_0^t V(\tau)\mathcal F(\xi)(0)\,\mathrm d\tau.
\end{align}

By Lemma~\ref{le;est;fractional heat flow;linear},
\begin{equation}
    \|V(\tau)\Pi\xi\|_{2\alpha}
    \lesssim
    \tau^{-1-\nu\alpha/2}
    \|\Pi\xi\|_{\alpha-2}.
\end{equation}
Consequently, for every $t>0$,
\begin{align}
    t^{\nu\alpha/2}
    \left\|
        \int_t^\infty V(\tau)\Pi\xi\,\mathrm d\tau
    \right\|_{2\alpha}
    &\le
    t^{\nu\alpha/2}
    \int_t^\infty
    \|V(\tau)\Pi\xi\|_{2\alpha}\,\mathrm d\tau
    \notag\\
    &\lesssim
    t^{\nu\alpha/2}
    \|\Pi\xi\|_{\alpha-2}
    \int_t^\infty
    \tau^{-1-\nu\alpha/2}\,\mathrm d\tau
    \notag\\
    &\lesssim
    \|\Pi\xi\|_{\alpha-2}.
\end{align}
Since
\begin{equation}
    \mathcal C^{\alpha+\delta-2}
    \hookrightarrow
    \mathcal C^{\alpha-2},
\end{equation}
we further have
\begin{equation}
    \|\Pi\xi\|_{\alpha-2}
    \lesssim
    \|\xi\|_{\alpha+\delta-2}.
\end{equation}

On the other hand, since
\begin{equation}
    V(\tau)\mathcal F(\xi)(0)
    =
    \tau^{\nu-1}\phi(0)\mathcal F(\xi)(0),
\end{equation}
we obtain
\begin{equation}
    \int_0^t
    V(\tau)\mathcal F(\xi)(0)\,\mathrm d\tau
    =
    \frac{t^\nu}{\nu}
    \phi(0)\mathcal F(\xi)(0).
\end{equation}
The right-hand side is spatially constant and hence belongs to
$\mathcal C^\beta$ for every $\beta\in\mathbb R$. Therefore, for
$t\in(0,T]$,
\begin{align}
    t^{\nu\alpha/2}
    \|\vartheta-I\xi(t)\|_{2\alpha}
    &\lesssim
    \|\xi\|_{\alpha+\delta-2}
    +
    t^{\nu(1+\alpha/2)}
    |\mathcal F(\xi)(0)|.
\end{align}
In particular, for fixed $T>0$,
\begin{equation}\label{eq:weighted-vartheta}
    \sup_{t\in(0,T]}
    t^{\nu\alpha/2}
    \|\vartheta-I\xi(t)\|_{2\alpha}
    \lesssim_T
    \|\xi\|_{\alpha+\delta-2}
    +
    |\mathcal F(\xi)(0)|.
\end{equation}

Moreover, using \eqref{eq;difference} together with
Lemma~\ref{le;est;fractional heat flow;linear}, we obtain
\begin{equation}\label{eq:vartheta-C-alpha}
    \|\vartheta-I\xi\|_{C_T\mathcal C^\alpha}
    \lesssim_T
    1+\|\xi\|_{\alpha+\delta-2},
\end{equation}
where the contribution of the zero Fourier mode is harmless.

Finally, since
\begin{equation}
    \frac{\nu\alpha}{2}
    <
    \nu+\frac{\nu(\alpha+\delta-2)}{2}
    =
    \frac{\nu(\alpha+\delta)}{2},
\end{equation}
Lemma~\ref{le;est;Schauder-type} yields
\begin{equation}
    \|I\xi\|_{C_T^{\nu\alpha/2}L^\infty}
    \lesssim_T
    \|\xi\|_{\alpha+\delta-2}.
\end{equation}
Since $\vartheta$ is independent of time, it follows that
\begin{equation}
    \|\vartheta-I\xi\|_{C_T^{\nu\alpha/2}L^\infty}
    =
    \|I\xi\|_{C_T^{\nu\alpha/2}L^\infty}
    \lesssim_T
    \|\xi\|_{\alpha+\delta-2}.
\end{equation}

Combining the above estimates, we conclude that
\begin{equation}\label{eq;est;theta-Ixi}
    \sup_{t\in(0,T]}
    t^{\nu\alpha/2}
    \|\vartheta-I\xi(t)\|_{2\alpha}
    +
    \|\vartheta-I\xi\|_{\mathcal L_T^{\alpha,\nu}}
    \lesssim_T
    1+\|\xi\|_{\alpha+\delta-2}.
\end{equation}

\begin{lemma}\label{le;renormalization constant}
    Let
    \begin{equation}
        g_t
        =
        \m E\left[
        V(t)\Pi\xi\circ\xi(x)
        \right].
    \end{equation}
    Then $g_t$ is independent of $x$ and
    \begin{equation}
        g_t
        =
        (2\pi)^{-2}
        \sum_{k\in\m Z_0^2}
        t^{\nu-1}
        \phi(t^{\nu/2}|k|).
    \end{equation}

    Let $\psi\in\ml S$ satisfy
    \begin{equation}
        \int_{\m R^2}\psi(x)\mathrm dx=1,
    \end{equation}
    and define the mollified noise $\xi^\varepsilon$ and the corresponding
    renormalization constant by
    \begin{align*}
        \xi^\varepsilon
        &=
        \psi^\varepsilon*\xi,\\
        c_\varepsilon
        &=
        \int_0^\infty
        \m E\left[
        V(t)\Pi\xi^\varepsilon\circ\xi^\varepsilon(x)
        \right]
        \mathrm dt.
    \end{align*}
    Then
    \begin{equation}
        c_\varepsilon
        =
        (2\pi)^{-2}
        \sum_{k\in\m Z_0^2}
        |k|^{-2}
        |\widehat\psi(\varepsilon k)|^2.
    \end{equation}
    In particular, $c_\varepsilon$ diverges logarithmically as
    $\varepsilon\to0$.

    \begin{proof}
        Choosing $\sigma=1$, recalling that
        \begin{equation}
            V(t)
            =
            t^{\nu-1}\phi(t^{\nu/2}|D|),
        \end{equation}
        the Fourier representation and the covariance of $\xi$ yield
        \begin{align*}
            g_t
            &=
            \m E\left[
            \Delta_{-1}
            (V(t)\Pi\xi\circ\xi)(x)
            \right]\\
            &=
            (2\pi)^{-2}
            \sum_{k\in\m Z_0^2}
            t^{\nu-1}
            \phi(t^{\nu/2}|k|).
        \end{align*}
        In particular, $g_t$ does not depend on $x$.

        Similarly, for the mollified noise,
        \begin{align*}
            \m E\left[
            V(t)\Pi\xi^\varepsilon
            \circ\xi^\varepsilon(x)
            \right]
            =
            (2\pi)^{-2}
            \sum_{k\in\m Z_0^2}
            t^{\nu-1}
            \phi(t^{\nu/2}|k|)
            |\widehat\psi(\varepsilon k)|^2.
        \end{align*}
        Therefore,
        \begin{align*}
            c_\varepsilon
            &=
            (2\pi)^{-2}
            \sum_{k\in\m Z_0^2}
            \left(
            \int_0^\infty
            t^{\nu-1}
            \phi(t^{\nu/2}|k|)
            \mathrm dt
            \right)
            |\widehat\psi(\varepsilon k)|^2\\
            &=
            (2\pi)^{-2}
            \sum_{k\in\m Z_0^2}
            |k|^{-2}
            |\widehat\psi(\varepsilon k)|^2,
        \end{align*}
        where we used
        \begin{equation}
            \int_0^\infty
            t^{\nu-1}
            \phi(t^{\nu/2}|k|)
            \mathrm dt
            =
            |k|^{-2}.
        \end{equation}
        Since
        \begin{equation}
            \sum_{0<|k|\lesssim\varepsilon^{-1}}
            |k|^{-2}
            \sim
            \log\frac1\varepsilon
        \end{equation}
        in dimension two, the constants $c_\varepsilon$ diverge
        logarithmically as $\varepsilon\to0$.
    \end{proof}
\end{lemma}

By the relation $U'(t)=\Delta V(t)$, for every $k\neq0$ we have
\begin{align*}
    |k|^2t^{\nu-1}\phi(t^{\nu/2}|k|)
    =
    -\p_t\widehat U(t,k).
\end{align*}
Therefore,
\begin{align*}
    \int_0^\infty
    t^{\nu-1}\phi(t^{\nu/2}|k|)
    \mathrm dt
    &=
    -|k|^{-2}
    \int_0^\infty
    \p_t\widehat U(t,k)
    \mathrm dt\\
    &=
    |k|^{-2},
\end{align*}
where we used $\widehat U(0,k)=1$ and
$\widehat U(t,k)\to0$ as $t\to\infty$.
Although the time-fractional heat kernel has only finite smoothing for each fixed $t>0$, the time-integrated kernel still gains two spatial derivatives. Indeed, for every $k\neq0$,
\begin{align*}
    \int_0^\infty
    t^{\nu-1}\phi(t^{\nu/2}|k|)
    \mathrm dt
    =
    |k|^{-2}.
\end{align*}
Consequently,
\begin{align*}
    \widehat{\vartheta}(k)
    =
    |k|^{-2}\widehat{\xi}(k),
    \qquad k\neq0,
\end{align*}
and hence
\begin{equation}
    -\Delta\vartheta=\Pi\xi.
\end{equation}
In particular, if $\xi\in\ml C^{\alpha-2}$, then
\begin{equation}
    \vartheta\in\ml C^\alpha.
\end{equation}

\begin{lemma}\label{le;area;2}
    Set
    \begin{align*}
        (\vartheta\diamond \xi)
        =
        \int_0^\infty
        \left(
        V(t)\Pi\xi\circ\xi-g_t
        \right)
        \mathrm dt.
    \end{align*}
    Then, for every $\alpha\in(1/2,1)$, $\delta>0$ satisfying $\alpha+\delta<1$ and $p>1$,
    \begin{equation}
        \m E\left[
        \Vert\vartheta\diamond\xi\Vert_{2{\alpha+2\delta}-2}^{p}
        \right]
        <\infty.
    \end{equation}
Moreover, let $\xi^\varepsilon=\varepsilon^{-2}\psi(\varepsilon)\ast\xi$ for $\varepsilon>0$, and $\vartheta^\varepsilon=\int_0^\infty V(t)\Pi\xi^\varepsilon\mathrm dt$. Then for every $\alpha\in(1/2,1)$, $\delta>0$ satisfying $\alpha+\delta<1$ and $p>1$, we have 
\begin{equation}
\lim_{\varepsilon\to0}\Vert \vartheta^\varepsilon\circ\xi^\varepsilon-c_\varepsilon-\vartheta\diamond\xi\Vert_{2\alpha+2\delta-2}=0
\end{equation}
almost surely and 
    \begin{proof}
        Let $t\in(0,1]$, and define
        \begin{equation}
            \Xi_t
            =
            V(t)\Pi\xi\circ\xi-g_t.
        \end{equation}
        Since $\Xi_t$ belongs to the second homogeneous Wiener chaos, Gaussian hypercontractivity yields, for every $p\geq1$ and $l\geq-1$,
        \begin{equation}\label{eq;Gaussian hypercontractivity}
            \m E\left[
            \Vert\Delta_l\Xi_t\Vert_{L^{2p}}^{2p}
            \right]
            \lesssim_p
            \left\Vert
            \m E\left[
            |\Delta_l\Xi_t(x)|^2
            \right]
            \right\Vert_{L^p}^p.
        \end{equation}
        Since $g_t$ is deterministic and spatially constant, for $l\geq0$ we have
        \begin{equation}
            \m E\left[
            |\Delta_l\Xi_t(x)|^2
            \right]
            =
            \mathrm{Var}
            \left(
            \Delta_l(V(t)\Pi\xi\circ\xi)(x)
            \right).
        \end{equation}
        By Wick's formula and the support properties of the Littlewood--Paley blocks,
        \begin{align*}
            \mathrm{Var}
            \left(
            \Delta_l(V(t)\Pi\xi\circ\xi)(x)
            \right)
            \lesssim
            \sum_{i\gtrsim l}
            2^{2i}2^{2l}
            t^{2\nu-2}
            \phi^2(t^{\nu/2}2^i).
        \end{align*}
        Using Lemma \ref{le;decay}, for every $\theta\in[0,1]$,
        \begin{align*}
            \mathrm{Var}
            \left(
            \Delta_l(V(t)\Pi\xi\circ\xi)(x)
            \right)
            \lesssim
            2^{2(1-\theta)l}
            t^{\nu(1-\theta)-2}.
        \end{align*}
        Therefore, by \eqref{eq;Gaussian hypercontractivity},
        \begin{align*}
            \left(
            \m E\left[
            \Vert\Delta_l\Xi_t\Vert_{L^{2p}}^{2p}
            \right]
            \right)^{1/(2p)}
            \lesssim_p
            2^{(1-\theta)l}
            t^{-1+\frac{\nu}{2}(1-\theta)}.
        \end{align*}

        It follows that
        \begin{align*}
            \left(
            \m E\left[
            \Vert\Xi_t\Vert_{B_{2p,2p}^{2\alpha+2\delta-2}}^{2p}
            \right]
            \right)^{1/(2p)}
            &\lesssim_p
            t^{-1+\frac{\nu}{2}(1-\theta)}
            \left(
            \sum_{l\geq-1}
            2^{2p(2\alpha-1-\theta)l}
            \right)^{1/(2p)}.
        \end{align*}
        Since $\alpha+\delta<1$, we may choose
        \begin{equation}
            2\alpha+2\delta-1<\theta<1.
        \end{equation}
        Hence the dyadic sum is finite and
        \begin{equation}
            \left(
            \m E\left[
            \Vert\Xi_t\Vert_{B_{2p,2p}^{2\alpha+2\delta-2}}^{2p}
            \right]
            \right)^{1/(2p)}
            \lesssim_p
            t^{-1+\frac{\nu}{2}(1-\theta)}.
        \end{equation}
        Since $\theta<1$,
        \begin{equation}
            -1+\frac{\nu}{2}(1-\theta)>-1,
        \end{equation}
        and consequently
        \begin{align*}
            \int_0^1
            \left(
            \m E\left[
            \Vert\Xi_t\Vert_{B_{2p,2p}^{2\alpha+2\delta-2}}^{2p}
            \right]
            \right)^{1/(2p)}
            \mathrm dt
            <\infty.
        \end{align*}
        By Minkowski's inequality,
        \begin{align*}
            \left(
            \m E\left[
            \left\Vert
            \int_0^1\Xi_t\mathrm dt
            \right\Vert_{B_{2p,2p}^{2\alpha+2\delta-2}}^{2p}
            \right]
            \right)^{1/(2p)}
            &\leq
            \int_0^1
            \left(
            \m E\left[
            \Vert\Xi_t\Vert_{B_{2p,2p}^{2\alpha+2\delta-2}}^{2p}
            \right]
            \right)^{1/(2p)}
            \mathrm dt\\
            &<\infty.
        \end{align*}
        For $t\geq1$, using again Lemma \ref{le;decay}, the multiplier
        \begin{equation}
            t^{\nu-1}\phi(t^{\nu/2}|k|)
        \end{equation}
        has sufficient decay in $t$ and $k$, and the same argument gives
        \begin{equation}
            \int_1^\infty
            \left(
            \m E\left[
            \Vert\Xi_t\Vert_{B_{2p,2p}^{2\alpha+2\delta-2}}^{2p}
            \right]
            \right)^{1/(2p)}
            \mathrm dt
            <\infty.
        \end{equation}
        We thus obtain
        \begin{equation}
            \m E\left[
            \left\Vert
            \int_0^\infty\Xi_t\mathrm dt
            \right\Vert_{B_{2p,2p}^{2\alpha+2\delta-2}}^{2p}
            \right]
            <\infty.
        \end{equation}
        Finally, choosing $p$ sufficiently large and using the Besov embedding
        \begin{equation}
            B_{2p,2p}^{2\alpha+2\delta-2}
            \subset
            \ml C^{2\alpha+2\delta-2-\frac1p},
        \end{equation}
        we obtain, after replacing $\alpha$ by a slightly larger exponent if necessary,
        \begin{equation}
            \m E\left[
            \Vert\vartheta\diamond\xi\Vert_{2\alpha+2\delta-2}^{p}
            \right]
            <\infty.
        \end{equation}
        The convergence is easily follows from the dominated convergence theorem and Besov embedding.
    \end{proof}
\end{lemma}

\subsection{Local well-posedness}
We now establish the well-posedness of \eqref{eq;Anderson model} by approximation. Let $\xi^\epsilon$ be a smooth approximation of $\xi$, and let $\vartheta^\epsilon=\int_0^\infty V(t)\Pi\xi^\epsilon\mathrm dt$. For each $\epsilon>0$, the regularized equation with renormalization 
\begin{equation}
    Lu^\varepsilon=F(u^\varepsilon)\xi^\varepsilon-c_{\varepsilon}F'(u^\varepsilon)F(u^\varepsilon)
\end{equation}
where $c_{\varepsilon}$ is the renormalized constant in Lemma \ref{le;renormalization constant}. Let $u^{\sharp,\varepsilon}$ satisfy
\begin{equation}\label{eq;paracontrolled ansatz}
    u^\varepsilon=F(u^\varepsilon)\prec\vartheta^\varepsilon+u^{\sharp,\varepsilon}
\end{equation}
and define
\begin{equation}
    \Phi^{\sharp,\varepsilon}
    =
    F(u^\varepsilon)\xi^\varepsilon
    -
    F(u^\varepsilon)\prec\xi^\varepsilon
    -
    c_\varepsilon F'(u^\varepsilon)F(u^\varepsilon).
\end{equation}
then $u^{\sharp,\varepsilon}$ satisfies the equation

\begin{equation}
    u^{\sharp,\varepsilon}
    =U(t)u_0+
    I(\Phi^{\sharp,\varepsilon})
    +
    I(F(u^\varepsilon)\prec\xi^\varepsilon)
    -
    F(u^\varepsilon)\prec\vartheta^\varepsilon.
\end{equation}
Let $\nu\in(0,1]$, $\alpha\in(2/3,1)$, $\delta>0$ satisfy $\alpha+\delta<1$, then by Lemma \ref{le;area;2}, $(\xi^\varepsilon,\vartheta^\varepsilon,\vartheta^\varepsilon\circ\xi^\varepsilon-c_\varepsilon)$ is uniformly bounded in 
$$\ml C^{\alpha+\delta-2}\times \ml C^{\alpha+\delta}\times \ml C^{2\alpha+2\delta-2}.$$
 We first estimate $\Phi^{\sharp,\varepsilon}$. By the paraproduct decomposition,
\begin{align*}
    \Phi^{\sharp,\varepsilon}
    =&
    F(u^\varepsilon)\succ\xi^\varepsilon
    +
    F(u^\varepsilon)\circ\xi^\varepsilon
    -
    c_\varepsilon F'(u^\varepsilon)F(u^\varepsilon).
\end{align*}
By the paralinearization of $F(u^\varepsilon)$ and \eqref{eq;paracontrolled ansatz}, we have
\begin{equation}
    F(u^\varepsilon)
   =
   F'(u^\varepsilon)\prec(F(u^\varepsilon)\prec \vartheta^\varepsilon)
   +
   F'(u^\varepsilon)\prec u^{\sharp,\varepsilon}
   +
   R_F(u^\varepsilon).
\end{equation}
Consequently,
\begin{align*}
    F(u^\varepsilon)\circ\xi^\varepsilon
    =&
    \left(
    F'(u^\varepsilon)\prec(F(u^\varepsilon)\prec \vartheta^\varepsilon)
    \right)\circ\xi^\varepsilon\\
    &+
    \left(
    F'(u^\varepsilon)\prec u^{\sharp,\varepsilon}
    \right)\circ\xi^\varepsilon
    +
    R_F(u^\varepsilon)\circ\xi^\varepsilon.
\end{align*}
Using the commutator identity only for the singular paraproduct term, we obtain
\begin{align*}
    F(u^\varepsilon)\circ\xi^\varepsilon
    =&
    C(F'(u^\varepsilon),F(u^\varepsilon)\prec\vartheta^\varepsilon,\xi^\varepsilon)
    +
    F'(u^\varepsilon)
    \left(
    (F(u^\varepsilon)\prec\vartheta^\varepsilon)\circ\xi^\varepsilon
    \right)\\
    &+
    \left(
    F'(u^\varepsilon)\prec u^{\sharp,\varepsilon}
    \right)\circ\xi^\varepsilon
    +
    R_F(u^\varepsilon)\circ\xi^\varepsilon.
\end{align*}
Therefore,
\begin{align*}
    \Phi^{\sharp,\varepsilon}
    =&
    F(u^\varepsilon)\succ\xi^\varepsilon
    +
    C(F'(u^\varepsilon),F(u^\varepsilon)\prec\vartheta^\varepsilon,\xi^\varepsilon)\\
    &+
    F'(u^\varepsilon)
    C(F(u^\varepsilon),\vartheta^\varepsilon,\xi^\varepsilon)
    +
    \left(
    F'(u^\varepsilon)\prec u^{\sharp,\varepsilon}
    \right)\circ\xi^\varepsilon\\
    &+
    R_F(u^\varepsilon)\circ\xi^\varepsilon
    +
    F'(u^\varepsilon)F(u^\varepsilon)
    \left(
    \vartheta^\varepsilon\circ\xi^\varepsilon-c_\varepsilon
    \right).
\end{align*}

We next derive estimates for the above terms that are uniform in $\varepsilon$.  
Lemma \ref{le;est;paraproduct} gives
\begin{align*}
    \Vert F(u^\varepsilon)\succ\xi^\varepsilon\Vert_{2\alpha-2}
    \lesssim&
    \Vert F(u^\varepsilon)\Vert_{\alpha}
    \Vert\xi^\varepsilon\Vert_{\alpha-2}\\
    \lesssim&
    \Vert F\Vert_{C_b^1}
    \left(
    1+\Vert u^\varepsilon\Vert_\alpha
    \right)
    \Vert\xi^\varepsilon\Vert_{\alpha-2}.
\end{align*}
and
\begin{align*}
    \Vert \left(
    F'(u^\varepsilon)\prec u^{\sharp,\varepsilon}
    \right)\circ\xi^\varepsilon\Vert_{3\alpha-2}\lesssim \Vert F'\Vert_{L^\infty}\Vert u^{\sharp,\varepsilon}\Vert_{2\alpha}\Vert\xi^\varepsilon\Vert_{\alpha-2}
\end{align*}
By Lemma \ref{le;para-controlled},
\begin{align*}
    \Vert
    C(F'(u^\varepsilon),F(u^\varepsilon)\prec\vartheta^\varepsilon,\xi^\varepsilon)
    \Vert_{3\alpha-2}
   \lesssim
   \Vert F'\Vert_{C_b^1}\Vert F\Vert_{L^\infty}\Vert u^\varepsilon\Vert_\alpha\Vert \vartheta^\varepsilon\Vert_\alpha\Vert \xi^\varepsilon\Vert_{\alpha-2}.
\end{align*}
Lemma \ref{le;paralinearization} yields
\begin{align*}
    \Vert R_F(u^\varepsilon)\Vert_{2\alpha}
    \lesssim
    \Vert F\Vert_{C_b^2}
    \left(
    1+\Vert u^\varepsilon\Vert_\alpha^2
    \right).
\end{align*}
Since $\alpha>2/3$, we have $3\alpha-2>0$, and hence the resonant product is well-controlled with
\begin{align*}
    \Vert R_F(u^\varepsilon)\circ\xi^\varepsilon\Vert_{3\alpha-2}
    \lesssim&
    \Vert R_F(u^\varepsilon)\Vert_{2\alpha}
    \Vert\xi^\varepsilon\Vert_{\alpha-2}\\
    \lesssim&
    \Vert F\Vert_{C_b^2}
    \left(
    1+\Vert u^\varepsilon\Vert_\alpha^2
    \right)
    \Vert\xi^\varepsilon\Vert_{\alpha-2}.
\end{align*}
while the renormalized resonant term satisfies
\begin{align*}
    &
    \left\Vert
    F'(u^\varepsilon)F(u^\varepsilon)
    \left(
    \vartheta^\varepsilon\circ\xi^\varepsilon-c_\varepsilon
    \right)
    \right\Vert_{2\alpha-2}\\
    &\qquad\lesssim
    C_F
    \left(
    1+\Vert u^\varepsilon\Vert_\alpha
    \right)
    \left\Vert
    \vartheta^\varepsilon\circ\xi^\varepsilon-c_\varepsilon
    \right\Vert_{2\alpha-2}.
\end{align*}
Combining the above estimates and using the uniform bounds on the enhanced noise, we obtain
\begin{equation}\label{eq;Phi;sharp}
    \Vert\Phi^{\sharp,\varepsilon}\Vert_{2\alpha-2}
    \lesssim
    C_F C_{\xi}
    \left(
    1
    +
    \Vert u^\varepsilon\Vert_\alpha^2
    +
    \Vert u^{\sharp,\varepsilon}\Vert_{2\alpha}
    \right),
\end{equation}
where $C_F=\Vert F\Vert_{C_b^2}(1+\Vert F\Vert_{C_b^2}^2)$ and 
\begin{equation}
    C_{\xi^\varepsilon}=(1+\Vert \xi^\varepsilon\Vert_{\alpha+\delta-2})^3+\Vert\vartheta^\varepsilon\circ\xi^\varepsilon-c_\varepsilon\Vert_{2\alpha+2\delta-2}
\end{equation}
\par Next, we derive the uniform bounded for $u^{\sharp,\varepsilon}$, noting that 
\begin{align*}
    I(F(u^\varepsilon)\prec \xi^\varepsilon)-F(u^\varepsilon)\prec \vartheta^\varepsilon=[I(F(u^\varepsilon)\prec \xi^\varepsilon)-F(u^\varepsilon)\prec I\xi^\varepsilon]+F(u^\varepsilon)\prec (I\xi^\varepsilon-\vartheta^\varepsilon)
\end{align*}
For the first term, by Lemma \ref{le;commutator;heat flow} and Lemma \ref{le;est;Schauder-type}, we have 
\begin{equation}
    \Vert I(F(u^\varepsilon)\prec \xi^\varepsilon)-F(u^\varepsilon)\prec I\xi^\varepsilon\Vert_{2\alpha}\lesssim \Vert F(u^\varepsilon)\Vert_{\ml L_T^{\alpha,\nu}}\Vert \xi^\varepsilon\Vert_{\alpha+\delta-2}.
\end{equation}
and 
\begin{equation}
    \Vert I(F(u^\varepsilon)\prec \xi^\varepsilon)-F(u^\varepsilon)\prec I\xi^\varepsilon\Vert_{\ml L_T^{\alpha,\nu}}\lesssim \Vert F(u^\varepsilon)\Vert_{\ml L_T^{\alpha,\nu}}\Vert \xi^\varepsilon\Vert_{\alpha+\delta-2}
\end{equation}
Combining with \eqref{eq;est;theta-Ixi}, by $ \Vert F(u^\varepsilon)\Vert_{\ml L_T^{\alpha,\nu}}
    \lesssim
    \Vert F\Vert_{C_b^1}\Vert u^\varepsilon\Vert_{\ml L_T^{\alpha,\nu}}$, it then follows that for $t\in[0,T]$,
\begin{align*}
    t^{\nu\alpha/2}\Vert I(F(u^\varepsilon)\prec \xi^\varepsilon)-F(u^\varepsilon)\prec I\xi^\varepsilon\Vert_{2\alpha}
    +&
    \Vert I(F(u^\varepsilon)\prec \xi^\varepsilon)-F(u^\varepsilon)\prec I\xi^\varepsilon\Vert_{\ml L_T^{\alpha,\nu}}\\
    \lesssim &
    \Vert F\Vert_{C^1_b}\Vert u^\varepsilon\Vert_{\ml L_T^{\alpha,\nu}}\Vert \xi^\varepsilon\Vert_{\alpha+\delta-2}.
\end{align*}

Consequently, by Lemma \ref{le;est;Schauder-type}, we have
\begin{align*}
    t^{\nu\alpha/2}\Vert u^{\sharp,\varepsilon}\Vert_{2\alpha}\lesssim&
    \Vert u_0\Vert_{\alpha}
    +
    t^{\nu\alpha/2} \Vert I(F(u^\varepsilon)\prec \xi^\varepsilon)
    -
    F(u^\varepsilon)\prec \vartheta^\varepsilon\Vert_{2\alpha}
    +
    t^{\nu\alpha/2}\Vert I(\Phi^{\sharp,\varepsilon})\Vert_{2\alpha}\\
    \lesssim&
    \Vert u_0\Vert_\alpha
    +t^{\nu\alpha/2}(\Vert F\Vert_{C_b^1}\Vert u^\varepsilon\Vert_{\ml L_T^{\alpha,\nu}}\Vert \xi^\varepsilon\Vert_{\alpha+\delta-2} )+\sup_{s\in[0,t]}s^{\nu\alpha/2}\Vert \Phi^{\sharp,\varepsilon}\Vert_{2\alpha-2}\\
    \lesssim&
    \Vert u_0\Vert_\alpha +C_FC_{\xi^\varepsilon} \Vert u^\varepsilon\Vert_{\ml L_T^{\alpha,\nu}}+C_FC_{\xi^\varepsilon}(1+\Vert u^\varepsilon\Vert_{C_T\ml C^\alpha}^2 +\sup_{s\in[0,t]}s^{\nu\alpha/2} \Vert u^{\sharp,\varepsilon}\Vert_{2\alpha} )
\end{align*}
and since 
\begin{align*}
    \Vert I\Phi^{\sharp,\varepsilon}(t)\Vert_{C_t^{\nu\alpha/2}L^\infty}\leq \sup_{s\in[0,t]}s^{\nu\alpha/2} \Vert \Phi^{\sharp,\varepsilon}\Vert_{2\alpha-2}
\end{align*}
then 
\begin{align*}
   \Vert u^{\sharp,\varepsilon}\Vert_{\ml L_T^{\alpha,\nu}}
   \lesssim& 
   \Vert u_0\Vert_\alpha+\sup_{t\in[0,T]}\Vert \Phi^{\sharp,\varepsilon}(t)\Vert_{2\alpha-2}
   +
   \Vert I(F(u^\varepsilon)\prec\xi^\varepsilon)-F(u^\varepsilon)\prec\vartheta^\varepsilon\Vert_{\ml L_T^{\alpha,\nu}}\\
   \lesssim&
   \Vert u_0\Vert_\alpha+C_FC_{\xi^\varepsilon}(1+\Vert u^\varepsilon\Vert_\alpha^2
   +
   C_FC_{\xi^\varepsilon} \Vert u^\varepsilon\Vert_{\ml L_T^{\alpha,\nu}}
   +
   \sup_{t\in[0,T]}t^{\nu\alpha/2}\Vert u^{\sharp,\epsilon}(t)\Vert_{2\alpha})
\end{align*}

uniformly in $\varepsilon$. Combining above estimates, we have
\begin{align*}
    \sup_{t\in[0,T]}t^{\nu\alpha/2}\Vert u^{\sharp,\varepsilon}\Vert_{2\alpha}
    +
    \Vert u^{\sharp,\varepsilon}\Vert_{\ml L_T^{\alpha+\nu}}
    \lesssim 
    \Vert u_0\Vert_\alpha
    +
     C_FC_{\xi^\varepsilon}(1+\Vert u^\varepsilon\Vert_\alpha^2
   +
 \Vert u^\varepsilon\Vert_{\ml L_T^{\alpha,\nu}}
   +
   \sup_{t\in[0,T]}t^{\nu\alpha/2}\Vert u^{\sharp,\epsilon}(t)\Vert_{2\alpha})
\end{align*}

For $\Vert u^\varepsilon\Vert_{C_T^{\nu\alpha/2}L^\infty}$, we have 
\begin{equation}
\begin{aligned}
      \Vert u^\varepsilon\Vert_{\ml L_T^{\alpha,\nu}}
      \leq&
      \Vert F(u^\varepsilon)\prec\vartheta^\varepsilon\Vert_{\ml L_T^{\alpha,\nu}}
      +
      \Vert u^{\sharp,\varepsilon}\Vert_{\ml L_T^{\alpha,\nu}}\\
      \lesssim& 
      C_FC_{\xi^\varepsilon}\Vert u^\varepsilon\Vert_{\ml L_T^{\alpha,\nu}}
      +
        \Vert u^{\sharp,\varepsilon}\Vert_{\ml L_T^{\alpha,\nu}}
\end{aligned}
\end{equation}
To closed the estimates, we need the scaling argument, using the similar argument, setting $\Lambda_\lambda f(t,x)=f(\lambda^{2/\nu}t,\lambda x)$, so that $L\Lambda_\lambda=\lambda^2\Lambda_\lambda L$.Now let $u^\lambda=\Lambda_\lambda u$, $u_0^\lambda=\Lambda_\lambda u_0$,
$\xi^\lambda=\lambda^{2-\alpha}\Lambda_\lambda\xi$, and
$F_\lambda=\lambda^\alpha F$. Then
\begin{equation}
    Lu_\lambda^\varepsilon
    =
    F_\lambda(u_\lambda^\varepsilon)\xi_\lambda^\varepsilon
    -
    c_\varepsilon^\lambda
    F_\lambda'(u_\lambda^\varepsilon)F_\lambda(u_\lambda^\varepsilon),
\end{equation}
where
\begin{align*}
    c_\varepsilon^\lambda
    =
    \lambda^{2-2\alpha}c_\varepsilon.
\end{align*}

and using Lemma ,we have
Let
\begin{align*}
    u_\lambda^\varepsilon=\Lambda_\lambda u^\varepsilon,
    \qquad
    \xi_\lambda^\varepsilon
    =
    \lambda^{2-\alpha}\Lambda_\lambda\xi^\varepsilon,
    \qquad
    F_\lambda=\lambda^\alpha F.
\end{align*}
Then $u_\lambda^\varepsilon$ satisfies
\begin{align*}
    Lu_\lambda^\varepsilon
    =
    F_\lambda(u_\lambda^\varepsilon)\xi_\lambda^\varepsilon
    -
    c_\varepsilon^\lambda
    F_\lambda'(u_\lambda^\varepsilon)
    F_\lambda(u_\lambda^\varepsilon),
\end{align*}
where
\begin{align*}
    c_\varepsilon^\lambda
    =
    \lambda^{2-2\alpha}c_\varepsilon.
\end{align*}
Repeating the previous estimates for the rescaled
equation, and using
\begin{align*}
    \Vert F_\lambda\Vert_{C_b^k}
    \lesssim
    \lambda^\alpha C_F,
\end{align*}
we obtain
\begin{equation}\label{eq;bounded;scaling}
    \begin{aligned}
    &\sup_{t\in(0,T]}t^{\nu\alpha/2}
    \Vert u_\lambda^{\sharp,\varepsilon}(t)\Vert_{2\alpha}
    +
    \Vert u_\lambda^{\sharp,\varepsilon}\Vert_{\ml L_T^{\alpha,\nu}}
    \\
    &\qquad\lesssim
    \Vert u_0^\lambda\Vert_\alpha
    +
    \lambda^\alpha C_FC_{\xi^{\varepsilon}_\lambda}
    \left(
        1+
        \Vert u_\lambda^\varepsilon\Vert_\alpha^2
        +
        \Vert u_\lambda^\varepsilon\Vert_{\ml L_T^{\alpha,\nu}}
        +
        \sup_{t\in(0,T]}t^{\nu\alpha/2}
        \Vert u_\lambda^{\sharp,\varepsilon}(t)\Vert_{2\alpha}
    \right).
\end{aligned}
\end{equation}

On the other hand, the paracontrolled decomposition
\begin{align*}
    u_\lambda^\varepsilon
    =
    F_\lambda(u_\lambda^\varepsilon)
    \prec\vartheta_\lambda^\varepsilon
    +
    u_\lambda^{\sharp,\varepsilon}
\end{align*}
and the usual paraproduct and composition estimates yield
\begin{align*}
    \Vert u_\lambda^\varepsilon\Vert_{\ml L_T^{\alpha,\nu}}
    &\lesssim
    \Vert
    F_\lambda(u_\lambda^\varepsilon)
    \prec\vartheta_\lambda^\varepsilon
    \Vert_{\ml L_T^{\alpha,\nu}}
    +
    \Vert u_\lambda^{\sharp,\varepsilon}
    \Vert_{\ml L_T^{\alpha,\nu}}
    \\
    &\lesssim
    \lambda^\alpha C_FC_{\xi^{\varepsilon}_\lambda}
    \left(
        1+
        \Vert u_\lambda^\varepsilon
        \Vert_{\ml L_T^{\alpha,\nu}}
    \right)
    +
    \Vert u_\lambda^{\sharp,\varepsilon}
    \Vert_{\ml L_T^{\alpha,\nu}}.
\end{align*}
Since $\Vert u_0^\lambda\Vert_\alpha\sim \lambda^\alpha\Vert u_0\Vert_\alpha $ and $\Vert \xi^\varepsilon_\lambda\Vert_\alpha\sim \Vert \xi^\varepsilon\Vert_\alpha$ and according to Lemma B.1 in \cite{gubinelli2015paracontrolled}, also $\Vert \vartheta_\lambda^\varepsilon\circ\xi_\lambda^\varepsilon-c^\lambda_\varepsilon\Vert_{2\alpha+2\delta-2}\lesssim \Vert \vartheta^\varepsilon\circ\xi^\varepsilon-c_\varepsilon\Vert_{2\alpha+2\delta-2}+\Vert \xi^\varepsilon\Vert_{\alpha+\delta-2}$, we have 
\[
C_{\xi^\varepsilon_\lambda}\lesssim C_{\xi^\varepsilon}
\]
Choosing $\lambda>0$ sufficiently small such that
\begin{align*}
    \lambda^\alpha C_FC_{\xi}
    \ll1,
\end{align*}
the term containing
$\Vert u_\lambda^\varepsilon\Vert_{\ml L_T^{\alpha,\nu}}$
can be absorbed into the left-hand side. Hence
\begin{equation}\label{eq;bounded u;2}
    \Vert u_\lambda^\varepsilon\Vert_{\ml L_T^{\alpha,\nu}}
    \lesssim
    1+
    \Vert u_\lambda^{\sharp,\varepsilon}
    \Vert_{\ml L_T^{\alpha,\nu}}.
\end{equation}
with arbitrary high probability for all sufficient small $\varepsilon$.
Plugging it into \eqref{eq;bounded;scaling}, we obtain that
\begin{equation}\label{eq;bounded;scaling;closed}
    \begin{aligned}
    &\sup_{t\in(0,T]}t^{\nu\alpha/2}
    \Vert u_\lambda^{\sharp,\varepsilon}(t)\Vert_{2\alpha}
    +
    \Vert u_\lambda^{\sharp,\varepsilon}\Vert_{\ml L_T^{\alpha,\nu}}
    \\
    &\qquad\lesssim
    \Vert u_0^\lambda\Vert_\alpha
    +
    \lambda^\alpha C_FC_{\xi}
    \left(
        1+
        (1+ \Vert u_\lambda^{\sharp,\varepsilon}\Vert_{\ml L_T^{\alpha,\nu}})^2
        +
        \Vert u_\lambda^{\sharp,\varepsilon}\Vert_{\ml L_T^{\alpha,\nu}}
        +
        \sup_{t\in(0,T]}t^{\nu\alpha/2}
        \Vert u_\lambda^{\sharp,\varepsilon}(t)\Vert_{2\alpha}
    \right)
\end{aligned}
\end{equation}
in the same probability sense. Let
\begin{align*}
    x
    :=
    \sup_{t\in(0,T]}t^{\nu\alpha/2}
    \Vert u_\lambda^{\sharp,\varepsilon}(t)\Vert_{2\alpha}
    +
    \Vert u_\lambda^{\sharp,\varepsilon}\Vert_{\ml L_T^{\alpha,\nu}}.
\end{align*}
Then \eqref{eq;bounded;scaling;closed} implies
\begin{align*}
    x
    \leq
    \Vert u_0^\lambda\Vert_\alpha
    +
    C\lambda^\alpha(1+x)^2.
\end{align*}
Choose $\lambda>0$(Only dependent on $C_\xi$, $C_F$ and $u_0$) sufficiently small such that
\begin{align*}
    C\lambda^\alpha
    \bigl(1+\Vert u_0^\lambda\Vert_\alpha\bigr)
    \ll 1.
\end{align*}
A standard continuity argument then yields
\begin{align*}
    \sup_{t\in(0,T]}t^{\nu\alpha/2}
    \Vert u_\lambda^{\sharp,\varepsilon}(t)\Vert_{2\alpha}
    +
    \Vert u_\lambda^{\sharp,\varepsilon}\Vert_{\ml L_T^{\alpha,\nu}}=x
    \lesssim
    1+\Vert u_0\Vert_\alpha.
\end{align*}
Combining with \eqref{eq;bounded u;2}, we have
\begin{align*}
    \Vert u^\varepsilon\Vert_{\ml L_T^{\alpha,\nu}}\lesssim C(C_F,C_{\xi},u_0)+\Vert u_0\Vert_\alpha.
\end{align*}
Then,  the family $(u^\varepsilon)_{\varepsilon>0}$ admits a uniform bounded in $\ml L_{T}^{\alpha,\nu}$ with arbitrary high probability for all sufficiently small $\varepsilon$. Since for $F\in C_b^3$, the map $(u_0,\xi,\vartheta,\xi\circ\vartheta)\mapsto u\in C_T\ml C^\alpha$ is locally Lipschitz continuous, then there is a unique paracontrolled solution on a small time interval in weak sense, which we summarize as following.
\begin{theorem}\label{thm;main;PAM}
Let $\nu\in(0,1]$, $\alpha\in(2/3,1)$, and $F\in C_b^3$, and let
$u_0\in\mathcal C^\alpha$. Let $(\xi^\varepsilon)_{\varepsilon>0}$ be a family of smooth approximations of $\xi$. For each $\varepsilon>0$, let $u^\varepsilon$ be the classical solution to the renormalized equation
\begin{equation}
    Lu^\varepsilon
    =
    F(u^\varepsilon)\xi^\varepsilon
    -
    c_\varepsilon F'(u^\varepsilon)F(u^\varepsilon),
    \qquad
    u^\varepsilon(0)=u_0.
\end{equation}

Then there exists $T>0$ such that the equation
\begin{equation}
    Lu
    =
    F(u)\diamond\xi,
    \qquad
    u(0)=u_0,
\end{equation}
admits a unique weak solution $u$ on $[0,T]$, in the following sense
\begin{equation}
    u^\varepsilon\longrightarrow u
    \qquad
    \text{in probability in } C_T\mathcal C^\alpha.
\end{equation}
\end{theorem}
\subsection{Blow up criterion}

In this section, we establish a blow-up criterion in terms
of the maximal existence time. Following the argument in
\cite{gubinelli2015paracontrolled}, for $F\in C_b^3$,
we refine the estimate for $\Phi^\sharp$ as follows:
\begin{equation}\label{eq;Phi sharp;enhanced}
    \begin{aligned}
        \|\Phi^\sharp\|_{2\alpha-2}
        \lesssim{}&
        \|F\|_{C_b^3}(1+C_FC_\xi)
        (1+\|u\|_{C_TL^\infty}^2)\\
        &\times
        \left(
            1+\|u\|_{\ml L_T^{\alpha,\nu}}
            +\|u^\sharp\|_{2\alpha}
        \right).
    \end{aligned}
\end{equation}
For $C>0$, define
\[
    \tau_C
    :=
    \inf\left\{
        t\geq0:\|u(t)\|_{L^\infty}\geq C
    \right\}.
\]
On every interval contained in $[0,\tau_C)$, the
$L^\infty$ norm of $u$ is bounded by $C$. Consequently,
the right-hand side of
\eqref{eq;Phi sharp;enhanced} depends at most linearly
on the paracontrolled norms. We therefore propagate the
local estimates from one time interval to the next,
as in the classical parabolic setting. The memory
effect, however, prevents a direct application of the
usual iteration argument. We address this difficulty
by estimating the increments of the solution while
retaining the history terms explicitly.

Fix $C>0$ and $T<\tau_C$. We first establish an a priori
bound for $u$ in $C_T\ml C^\alpha$. Since
$\|u\|_{C_TL^\infty}\leq C$, the scaled solution satisfies
\[
    \|u_\lambda\|_{C_{\lambda^{-2/\nu}T}L^\infty}
    \leq C.
\]
Repeating the local existence argument with
\eqref{eq;Phi sharp;enhanced}, we choose $\lambda>0$,
independently of $u_0$, sufficiently small that
\begin{equation}\label{eq;scaled;smallness}
    \lambda^\alpha
    \|F\|_{C_b^3}(1+C^2)(1+C_FC_\xi)
    \ll1.
\end{equation}
The local estimates then yield
\begin{equation}\label{eq;scaled;local;bound}
    \begin{aligned}
        \|u_\lambda\|_{\ml L_T^{\alpha,\nu}}
        &\lesssim 1+\|u_0^\lambda\|_\alpha,\\
        \sup_{t\in(0,T]}t^\gamma
        \|u_\lambda^\sharp(t)\|_{2\alpha}
        +
        \|u_\lambda^\sharp\|_{\ml L_T^{\alpha,\nu}}
        &\lesssim 1+\|u_0^\lambda\|_\alpha,
    \end{aligned}
\end{equation}
where $\gamma:=\nu\alpha/2$.

We next estimate the increments on the following interval.
Set
\[
    w_T(t):=u_\lambda(T+t)-u_\lambda(T),
    \qquad
    z_T(t):=
    u_\lambda^\sharp(T+t)-u_\lambda^\sharp(T),
    \qquad 0\leq t\leq T.
\]
By paracontrolled ansatz, we have
\begin{align*}
    \Vert w_T\Vert_{\ml L_T^{\alpha,\nu}}\lesssim C_{F^\lambda}C_{\xi^\lambda}\Vert w_T\Vert_{\ml L_T^{\alpha,\nu}}+\Vert z_T\Vert_{\ml L_T^{\alpha,\nu}}
\end{align*}
combining with \eqref{eq;scaled;smallness}, we have
\begin{equation}\label{eq;controll 2}
     \Vert w_T\Vert_{\ml L_T^{\alpha,\nu}}\lesssim \Vert z_T\Vert_{\ml L_T^{\alpha,\nu}}
\end{equation}
Subtracting the mild formulations at times $T+t$ and $T$,
we obtain
\begin{equation}\label{eq;decomposition;R}
        z_T(t)=R^1_T(t)+R^2_T(t)+R^3_T(t),
\end{equation}
where
\begin{align*}
    R^1_T(t)
    :={}&
    [U(T+t)-U(T)]u_0^\lambda\\
    &+
    \int_0^T
    [V(T+t-s)-V(T-s)]
    \Phi_\lambda^\sharp(u_\lambda(s))\,\mathrm ds,\\
    R^2_T(t)
    :={}&
    \int_0^t
    V(t-s)
    \Phi_\lambda^\sharp(u_\lambda(T+s))\,\mathrm ds,
\end{align*}
and
\begin{align*}
    R^3_T(t)
    :={}&
    I\bigl(F_\lambda(u_\lambda)\prec\xi_\lambda\bigr)(T+t)
    -
    I\bigl(F_\lambda(u_\lambda)\prec\xi_\lambda\bigr)(T)\\
    &-
    \Bigl[
        F_\lambda(u_\lambda(T+t))
            \prec\vartheta_\lambda(T+t)
        -
        F_\lambda(u_\lambda(T))
            \prec\vartheta_\lambda(T)
    \Bigr].
\end{align*}

Applying the difference estimates together with
\eqref{eq;Phi sharp;enhanced} and
\eqref{eq;scaled;local;bound}, we obtain
\begin{equation}\label{eq;est;R1}
    \sup_{t\in(0,T]}t^\gamma
    \|R^1_T(t)\|_{2\alpha}
    +
    \|R^1_T\|_{\ml L_T^{\alpha,\nu}}
    \lesssim
    1+\|u_0^\lambda\|_\alpha.
\end{equation}
Similarly, the local Schauder estimates give
\begin{equation}\label{eq;est;R2}
    \begin{aligned}
        \sup_{t\in(0,T]}t^\gamma
        \|R^2_T(t)\|_{2\alpha}
        +
        \|R^2_T\|_{\ml L_T^{\alpha,\nu}}\lesssim&
        \lambda^\alpha
        \|F\|_{C_b^3}(1+C^2)(1+C_FC_\xi)\\
        &\times
        \left(
            1+
            \|u_\lambda(T+\cdot)\|_{\ml L_T^{\alpha,\nu}}
            +
            \sup_{s\in(0,T]}s^\gamma
            \|u_\lambda^\sharp(T+s)\|_{2\alpha}
        \right).
    \end{aligned}
\end{equation}
noting that
\[
    \sup_{s\in(0,T]}s^\gamma
    \|u_\lambda^\sharp(T+s)\|_{2\alpha}
    \lesssim
    1+\|u_0^\lambda\|_\alpha
    +
    \sup_{s\in(0,T]}s^\gamma\|z_T(s)\|_{2\alpha},
\]
then we have
\begin{equation}
    \begin{aligned}
         \sup_{t\in(0,T]}t^\gamma
        \|R^2_T(t)\|_{2\alpha}
        +
        \|R^2_T\|_{\ml L_T^{\alpha,\nu}}\lesssim& \lambda^\alpha
        \|F\|_{C_b^3}(1+C^2)(1+C_FC_\xi)\\
        &\times(1+\Vert u_0^\lambda\Vert_\alpha+\Vert w_T\Vert_{\ml L_T^{\alpha,\nu}}+\Vert z_T\Vert_{\ml L_T^{\alpha,\nu}})
    \end{aligned}
\end{equation}
To estimate $R^3_T$, define
\[
    \widetilde F_\lambda(t)
    :=
    F_\lambda(u_\lambda(T+t))
    -
    F_\lambda(u_\lambda(T)).
\]
Changing variables gives
\[
    \begin{aligned}
        &\int_T^{T+t}
        V(T+t-s)
        \bigl(F_\lambda(u_\lambda(s))
            \prec\xi_\lambda\bigr)\,\mathrm ds\\
        &\qquad=
        \int_0^t
        V(t-s)
        \bigl(F_\lambda(u_\lambda(T+s))
            \prec\xi_\lambda\bigr)\,\mathrm ds.
    \end{aligned}
\]
Consequently,
\begin{align*}
    R^3_T(t)
    ={}&
    \int_0^T
    [V(T+t-s)-V(T-s)]
    \bigl(F_\lambda(u_\lambda(s))
        \prec\xi_\lambda\bigr)\,\mathrm ds\\
    &+
    I(\widetilde F_\lambda\prec\xi_\lambda)(t)
    +
    \int_0^t
    V(t-s)
    \bigl(F_\lambda(u_\lambda(T))
        \prec\xi_\lambda\bigr)\,\mathrm ds\\
    &-
    \Bigl[
        F_\lambda(u_\lambda(T+t))
            \prec\vartheta_\lambda(T+t)
        -
        F_\lambda(u_\lambda(T))
            \prec\vartheta_\lambda(T)
    \Bigr].
\end{align*}
The history terms are estimated using the difference
estimates, while the remaining terms are treated
together using the commutator estimates. In particular,
the shifted history integral satisfies
\[
    \begin{aligned}
        &t^\gamma
        \left\|
            \int_0^T
            V(T+t-s)
            \bigl(F_\lambda(u_\lambda(s))
                \prec\xi_\lambda\bigr)\,\mathrm ds
        \right\|_{2\alpha}\\
        &\qquad\lesssim
        t^\gamma
        \int_0^T
        (T+t-s)^{-1-\gamma}\,\mathrm ds\,
        C_{F_\lambda}C_{\xi_\lambda}
        \left(
            1+\|u_\lambda\|_{\ml L_T^{\alpha,\nu}}
        \right)\\
        &\qquad\lesssim
        C_{F_\lambda}C_{\xi_\lambda}
        \left(
            1+\|u_\lambda\|_{\ml L_T^{\alpha,\nu}}
        \right).
    \end{aligned}
\]
Combining the history and commutator estimates with
\eqref{eq;scaled;local;bound}, we obtain
\begin{equation}\label{eq;est;R3}
    \begin{aligned}
        &\sup_{t\in(0,T]}t^\gamma
        \|R^3_T(t)\|_{2\alpha}
        +
        \|R^3_T\|_{\ml L_T^{\alpha,\nu}}\\
        &\qquad\lesssim
        C_{F_\lambda}C_{\xi_\lambda}
        \left(
            1+\|u_0^\lambda\|_\alpha
            +\|w_T\|_{\ml L_T^{\alpha,\nu}}
        \right).
    \end{aligned}
\end{equation}
Combining \eqref{eq;est;R1}, \eqref{eq;est;R2},
\eqref{eq;est;R3}, and \eqref{eq;controll 2}, and 
we obtain
\begin{align*}
    &\sup_{t\in(0,T]}t^\gamma\|z_T(t)\|_{2\alpha}
    +\|z_T\|_{\ml L_T^{\alpha,\nu}}\\
    &\qquad\lesssim
    1+\|u_0^\lambda\|_\alpha\\
    &\qquad\quad+
    \Bigl[
        \lambda^\alpha\|F\|_{C_b^3}
        (1+C^2)(1+C_FC_\xi)
        +C_{F_\lambda}C_{\xi_\lambda}
    \Bigr]
    \left(
        \sup_{t\in(0,T]}t^\gamma\|z_T(t)\|_{2\alpha}
        +\|z_T\|_{\ml L_T^{\alpha,\nu}}
    \right).
\end{align*}
By our choice of $\lambda$, the last term can be
absorbed into the left-hand side. Consequently,
\[
    \sup_{t\in(0,T]}t^\gamma\|z_T(t)\|_{2\alpha}
    +\|z_T\|_{\ml L_T^{\alpha,\nu}}
    \lesssim 1+\|u_0^\lambda\|_\alpha.
\]
Using \eqref{eq;controll 2}, we conclude that
\[
    \|u_\lambda(T+\cdot)\|_{\ml L_T^{\alpha,\nu}}
    \lesssim
    \|u_\lambda(T)\|_\alpha
    +\|w_T\|_{\ml L_T^{\alpha,\nu}}
    \lesssim 1+\|u_0^\lambda\|_\alpha.
\]
Thus, the estimates extend to $[0,2T]$. Iterating
this argument over successive intervals yields
\[
    u_\lambda\in C_{\lambda^{-2/\nu}T}\ml C^\alpha.
\]
Scaling back, we obtain 
\begin{equation}\label{eq;u;uniform}
    \Vert u\Vert_{C_T\ml C^\alpha}\lesssim 1+\Vert u_0\Vert_\alpha
\end{equation}
$u\in C_T\ml C^\alpha$. We have shown that $u\in C_T\ml C^\alpha$. We now
consider its continuation beyond time $T$. After rescaling
$u$ to $u_{\lambda_0}$, we apply the preceding argument
to the increments
\[
    w_b(t)
    :=
    u_{\lambda_0}(b+t)-u_{\lambda_0}(b),
    \qquad
    z_b(t)
    :=
    u_{\lambda_0}^\sharp(b+t)
    -u_{\lambda_0}^\sharp(b),
    \qquad
    b:=\lambda_0^{-2/\nu}T,
\]
for $t\in[0,T]$.

The decomposition \eqref{eq;decomposition;R} remains valid:
\[
    z_b=R_b^1+R_b^2+R_b^3.
\]
The estimates for $R_b^1$ and $R_b^3$ follow as before.
The only difference concerns $R_b^2$: on the new
interval, we no longer have an a priori $L^\infty$
bound and must therefore use the quadratic estimate
\eqref{eq;Phi;sharp}. The Schauder estimates yield
\begin{equation}\label{eq;continuation;R2}
    \begin{aligned}
        &\sup_{t\in(0,T]}t^\gamma
        \|R_b^2(t)\|_{2\alpha}
        +
        \|R_b^2\|_{\ml L_T^{\alpha,\nu}}\\
        &\qquad\lesssim
        \lambda_0^\alpha C_FC_\xi
        \left(
            1+
            \|u_{\lambda_0}(b+\cdot)\|_{C_T\ml C^\alpha}^2
            +
            \sup_{s\in(0,T]}s^\gamma
            \|u_{\lambda_0}^\sharp(b+s)\|_{2\alpha}
        \right).
    \end{aligned}
\end{equation}
Moreover,
\[
    \|u_{\lambda_0}(b+\cdot)\|_{C_T\ml C^\alpha}^2
    \lesssim
    \|w_b\|_{C_T\ml C^\alpha}^2
    +
    \|u_{\lambda_0}(b)\|_\alpha^2,
\]
and
\begin{align*}
    \sup_{s\in(0,T]}s^\gamma
    \|u_{\lambda_0}^\sharp(b+s)\|_{2\alpha}
    &\leq
    \sup_{s\in(0,T]}s^\gamma
    \|z_b(s)\|_{2\alpha}\\
    &\quad+
    T^\gamma\|u_{\lambda_0}^\sharp(b)\|_{2\alpha}.
\end{align*}

To control the contribution at the left endpoint, we
apply the preceding argument on the interval before
$b$. By the uniform bound \eqref{eq;u;uniform}, the
estimate for $\Phi^\sharp$ on that interval reduces
to the linear estimate
\eqref{eq;Phi sharp;enhanced}. Choosing $\lambda_0$
sufficiently small, depending only on $u_0$, $C_F$,
$C_\xi$, and $T$, the resulting estimates give
\[
    \|u_{\lambda_0}(b)\|_\alpha
    +
    T^\gamma\|u_{\lambda_0}^\sharp(b)\|_{2\alpha}
    \lesssim 1+\|u_0^{\lambda_0}\|_\alpha.
\]
Consequently, \eqref{eq;continuation;R2} implies
\begin{equation}\label{eq;continuation;R2;increment}
    \begin{aligned}
        &\sup_{t\in(0,T]}t^\gamma
        \|R_b^2(t)\|_{2\alpha}
        +
        \|R_b^2\|_{\ml L_T^{\alpha,\nu}}\\
        &\qquad\lesssim
        \lambda_0^\alpha C_FC_\xi
        \left(
            (1+\|u_0^{\lambda_0}\|_\alpha)^2
            +
            \|w_b\|_{\ml L_T^{\alpha,\nu}}^2
            +
            \sup_{s\in(0,T]}s^\gamma
            \|z_b(s)\|_{2\alpha}
        \right).
    \end{aligned}
\end{equation}

For brevity, write
\[
    \|z_b\|_{\mathcal X_T}
    :=
    \sup_{t\in(0,T]}t^\gamma
    \|z_b(t)\|_{2\alpha}
    +
    \|z_b\|_{\ml L_T^{\alpha,\nu}}.
\]
As before, the paracontrolled ansatz and the smallness
of $\lambda_0$ imply
\[
    \|w_b\|_{\ml L_T^{\alpha,\nu}}
    \lesssim
    \|z_b\|_{\ml L_T^{\alpha,\nu}}
    \leq \|z_b\|_{\mathcal X_T}.
\]
Combining \eqref{eq;continuation;R2;increment} with
the estimates for $R_b^1$ and $R_b^3$, and absorbing
the linear terms, we obtain
\[
    \|z_b\|_{\mathcal X_T}
    \lesssim
    1+\|u_0^{\lambda_0}\|_\alpha
    +
    \lambda_0^\alpha C_FC_\xi
    \left(
        (1+\|u_0^{\lambda_0}\|_\alpha)^2
        +\|z_b\|_{\mathcal X_T}^2
    \right).
\]
Taking $\lambda_0$ sufficiently small that
\[
    \lambda_0^\alpha C_FC_\xi
    (1+\|u_0^{\lambda_0}\|_\alpha)\ll1,
\]
a standard continuity argument on $[0,a]$, $0<a\leq T$,
yields
\[
    \|z_b\|_{\mathcal X_T}
    \lesssim 1+\|u_0^{\lambda_0}\|_\alpha.
\]
Hence,
\[
    \|u_{\lambda_0}(b+\cdot)\|_{\ml L_T^{\alpha,\nu}}
    \lesssim
    \|u_{\lambda_0}(b)\|_\alpha
    +
    \|w_b\|_{\ml L_T^{\alpha,\nu}}
    \lesssim 1+\|u_0^{\lambda_0}\|_\alpha.
\]
Scaling back yields an a priori bound on
$[T,T+\lambda_0^{2/\nu}T]$, closing the estimates
on the new interval. Summarize our discussion, we obtain the blow up criterion
\begin{theorem}
      The solution constructed in Theorem \ref{thm;main;PAM} can extends uniquely to $[0,T^*)$, where $T^*=\tau_\infty$ is defined as 
      \begin{equation}
          \tau_\infty
    :=
    \inf\left\{
        t\geq0:\|u(t)\|_{L^\infty}=\infty
    \right\}.
      \end{equation}
\end{theorem}

\appendix 
\section{Integral inequality}\label{sec;integral}
\begin{lemma}\label{le;fractional calculus}
Let $0\leq t'\leq t\leq T$ and $\delta\in]1,2[$. Then, for any
$\sigma\in(0,2-\delta]$, we have
\begin{align*}
    \int_0^{t'}
    \int_{t'-s}^{t-s}
    \tau^{-\delta}\mathrm d\tau\mathrm ds
    \lesssim
    T^{2-\delta-\sigma}(t-t')^\sigma.
\end{align*}
If $\delta\in]1,\frac{3}{2}[$, then, for any
$\sigma\in(0,3-2\delta]$, we have
\begin{align*}
    \int_0^{t'}
    \left(
    \int_{t'-s}^{t-s}
    \tau^{-\delta}\mathrm d\tau
    \right)^2
    \mathrm ds
    \lesssim
    T^{3-2\delta-\sigma}(t-t')^\sigma.
\end{align*}
Moreover, if $\delta\in]1,2[$ and $\rho\in[0,1)$, then, for any
$\sigma>0$ satisfying $\sigma+\rho\leq2-\delta$, we have
\begin{align*}
    \int_0^{t'}
    \int_{t'-s}^{t-s}
    \tau^{-\delta}s^{-\rho}
    \mathrm d\tau\mathrm ds
    \lesssim
    T^{2-\delta-\rho-\sigma}
    (t-t')^\sigma.
\end{align*}

\begin{proof}
    Let $\gamma=t-t'$ and set
    \[
        D_\gamma(\eta)
        :=\eta^{1-\delta}-(\eta+\gamma)^{1-\delta}.
    \]
    By the mean value theorem,
    \begin{equation}\label{eq:kernel-difference-bound}
        0\le D_\gamma(\eta)
        \lesssim
        \min\{\eta^{1-\delta},\gamma\eta^{-\delta}\}.
    \end{equation}
    We may assume $\gamma>0$, since the case $\gamma=0$
    is trivial.

    For $p=1,2$, changing variables $\eta=t'-s$
    and using \eqref{eq:kernel-difference-bound}, we obtain
    \begin{align*}
        &\int_0^{t'}
        \left(
            \int_{t'-s}^{t-s}\tau^{-\delta}
            \,\mathrm d\tau
        \right)^p\mathrm ds\\
        &\qquad\lesssim
        \int_0^\gamma \eta^{p(1-\delta)}
        \,\mathrm d\eta
        +\gamma^p\int_\gamma^\infty
        \eta^{-p\delta}\,\mathrm d\eta\\
        &\qquad\lesssim
        \gamma^{1+p(1-\delta)},
    \end{align*}
    provided $p(\delta-1)<1$.
    Taking $p=1$ and $p=2$ gives the first two
    estimates, respectively.

    For the third estimate, suppose
    $0\le\rho<2-\delta$. By the same change of
    variables, it suffices to estimate
    \[
        \int_0^{t'}
        D_\gamma(\eta)(t'-\eta)^{-\rho}
        \,\mathrm d\eta.
    \]
    If $t'\le 2\gamma$, then
    \begin{align*}
        \int_0^{t'}
        D_\gamma(\eta)(t'-\eta)^{-\rho}
        \,\mathrm d\eta
        &\lesssim
        \int_0^{t'}
        \eta^{1-\delta}(t'-\eta)^{-\rho}
        \,\mathrm d\eta\\
        &\lesssim
        (t')^{2-\delta-\rho}
        \lesssim
        \gamma^{2-\delta-\rho}.
    \end{align*}
    If $t'>2\gamma$, splitting the integral
    at $t'/2$ and using
    \eqref{eq:kernel-difference-bound}, we obtain
    \begin{align*}
        &\int_0^{t'}
        D_\gamma(\eta)(t'-\eta)^{-\rho}
        \,\mathrm d\eta\\
        &\qquad\lesssim
        (t')^{-\rho}
        \int_0^{t'/2}D_\gamma(\eta)\,\mathrm d\eta
        +\gamma(t')^{-\delta}
        \int_{t'/2}^{t'}
        (t'-\eta)^{-\rho}\,\mathrm d\eta\\
        &\qquad\lesssim
        (t')^{-\rho}\gamma^{2-\delta}
        +\gamma(t')^{1-\delta-\rho}\\
        &\qquad\lesssim
        \gamma^{2-\delta-\rho}.
    \end{align*}
    Finally, since $\gamma\le T$, the desired
    estimates follow from
    \[
        \gamma^a\le T^{a-\sigma}\gamma^\sigma,
        \qquad 0<\sigma\le a,
    \]
    with $a=2-\delta$, $3-2\delta$, and
    $2-\delta-\rho$, respectively.
\end{proof}

\end{lemma}

\begin{lemma}\label{le;est;I,II}
    Let $\alpha\in(0,1/2)$, $\beta\in \m R$ functional $I_{t,t'}$ and $II_{t,t'}$ satisfy \eqref{eq;I} and \eqref{eq;II},  $f_k(t)$ satisfying $f_0(t)=1$ and  
        \begin{align}
            \label{eq;condition 1}\sup_{t\geq0,k\in \m Z\backslash \{0\}}t^{\alpha+1}\vert k\vert^{\beta }|f'_{k}(t)|\leq C\\
             \label{eq;condition 2}\sup_{t\geq0,k\in \m Z}t^{\alpha}|k|^{\beta} |f_k(t)|\leq   C      \end{align}
    and $f_0(t)$, then for $\sigma\in(0,1-2\alpha]$, we have 
     \begin{align*}
         I_{t,t'}(f_k)+II_{t,t'}(f_k)+|II_{t,0}(f_k)-II_{t',0}(f_k)|\lesssim |k|^{-2\beta} T^{1-2\alpha-\sigma}|t-t'|^{\sigma}
     \end{align*}
    \begin{proof}
    Without loss of generality, for $k\geq 1$ using the first assumption, using Lemma \ref{le;fractional calculus}, for  $\alpha\in(0,1/2)$ and $\sigma\in(0,1-2\alpha]$, we have
        \begin{align*}
            I_{t,t'}(f_k)= &\int_0^{t\wedge t'}(f_k(t-s)-f_k(t'-s))^2\mathrm{d}s\\
            =&\int_0^{t\wedge t'}(\int_{t'-s}^{t-s}f_k'(\tau)\mathrm d\tau)^2\mathrm{d}s\\
            \leq&|k|^{-2\beta}\int_0^{t\wedge t'}(\int_{t'-s}^{t-s}\tau^{-1-\alpha}\mathrm d\tau)^2)\\
            \leq &|k|^{4(\theta-1)} T^{1-2\alpha-\sigma}|t-t'|^{\sigma}.
        \end{align*}
         Using the second assumption, we have
         \begin{align*}
             II_{t,t'}(f_k)=&\int_{t\wedge t'}^{t\vee t'}f_k^2(t\vee t'-s)\mathrm{d}s \\
             \leq& |k|^{-2\beta}\int_{t\wedge t'}^{t\vee t'}(t\vee t'-s)^{-2\alpha}\mathrm d s\\
             \leq &|k|^{-2\beta}T^{1-2\alpha-\sigma}|t-t'|^{\sigma}.
         \end{align*}
         Finally, we consider $II_{t,0}(f_k)-II_{t',0}(f_k)$, noting that 
         \begin{align*}
             \int_0^{t\wedge t'}f^2_k(t-s)-f_k^2(t'-s)\mathrm ds=&\int_0^{t\wedge t'}(f_k(t-s)-f_k(t'-s))(f_k(t-s)+f_k(t'-s))\mathrm ds\\
             =&\int_0^{t\wedge t'}\int_{t'-s}^{t-s}f'_k(\tau)\mathrm d \tau(f_k(t-s)+f_k(t'-s))\mathrm ds
             =I_1+I_2
         \end{align*}
         Without loss of generality, let $t\geq t'$ For $I_1$, by \eqref{eq;condition 1} and  \eqref{eq;condition 2}, for any $\sigma>0$, we have 
         \begin{align*}
             I_1\lesssim& \int_0^{ t'}\int_{t'-s}^{t-s}f'_k(\tau)\mathrm d \tau f_k(t-s)\mathrm ds\\
             \lesssim& |k|^{-\beta}|t-t'|^{\alpha}\int_0^{t'}\frac{f_k(t-s)}{(t-s)^{\alpha}(t'-s)^{\alpha}}\mathrm ds\\
             \lesssim& |k|^{-2\beta}|t-t'|^{\alpha}\int_0^{t'}\frac{1}{(h+t'-s)^{2\alpha}(t'-s)^{\alpha}}\mathrm ds\\
             \lesssim& |k|^{-2\beta}|t-t'|^{\alpha-2\alpha\theta}\int_0^{t'}\frac{1}{(t'-t)^{\alpha+2\alpha(1-\theta)}}\mathrm ds\\
             \lesssim&|k|^{-2\beta}T^\sigma|t-t'|^{1-2\alpha-\sigma}
         \end{align*}
         It's similar for $I_2$ which we omit the details. By definition \eqref{eq;II}, for any $\sigma>0$ we have 
         \begin{align*}
             II_{t,0}(f_k)-II_{t',0}(f_k)=&\int_0^tf^2_k(t-s)\mathrm ds-\int_0^{t'}f_k^2(t'-s)\mathrm ds\\
             =&\int_0^{t\wedge t'}f^2_k(t-s)-f_k^2(t'-s)\mathrm ds+II_{t,t'}(f_k)\\
             \leq&  T^{\sigma}|k|^{-2\beta}|t-t'|^{1-2\alpha-\sigma}
         \end{align*}
         \end{proof}
\end{lemma}

\section{Commutator estimates}\label{sec;commutator}
In this section, we consider the commutator estimates. In the integer-order case, the Leibniz rule asserts that $\p_t^n(uv)=\sum_{i=1}^k C(k,i)\p_t^iu\p_t^{k-i}v$, This formula is fundamental to many powerful tools in pde analysis. In contrast, generally, for any $\nu\in(0,1)$, $\p_t^\nu (uv)\neq v\p_t^\nu u+u\p_t^\nu v$.  For $\nu=1$, the commutator estimate take the form of $L(u\pprec v)-u\pprec Lv$ in \cite{gubinelli2015paracontrolled,gubinelli2017kpz}. However in the setting of the time fractional case, it's not easy since the $\p_t^\nu (u\pprec v)-u\pprec \p_t^\nu v $ can't easily gain the smoothing 
\begin{lemma}\label{le;commutator;decay}
    Let $\alpha\in(0,1)$, $\beta\in \m R$, and let $\varphi$ and $\phi$ satisfy the assumptions of Lemma \ref{le;decay;varphi;phi}. Then, for $\delta\in[-1,2]$, we have
    \begin{align*}
        \Vert \varphi(\epsilon D)(u\prec v)-u\prec \varphi(\epsilon D)v\Vert_{\alpha+\beta+\delta}
        \lesssim \epsilon^{-\delta}\Vert u\Vert_{\ml C^\alpha}\Vert v\Vert_{\ml C^\beta}.
    \end{align*}
    Moreover, for $\delta\in[-1,4]$, we have
    \begin{align*}
        \Vert \phi(\epsilon D)(u\prec v)-u\prec \phi(\epsilon D)v\Vert_{\alpha+\beta+\delta}
        \lesssim \epsilon^{-\delta}\Vert u\Vert_{\ml C^\alpha}\Vert v\Vert_{\ml C^\beta}.
    \end{align*}
    \begin{proof}
       Let $\psi\in C_0^\infty(\m R^d)$ be supported in an annulus and satisfy $\psi\rho=\rho$. For $j\geq0$ and $|\eta|=1$, by scaling and integration by parts, we have
       \begin{align*}
           \Vert x^\eta \ml F^{-1}(\psi(2^{-j}\cdot)\varphi(\epsilon\cdot))\Vert_{L^1}
           \lesssim&
           2^{-j}\Vert (1-\Delta)^N((\p^\eta \psi)\varphi(\epsilon2^j\cdot))\Vert_{L^1}\\
           &+\epsilon\Vert (1-\Delta)^N(\psi\p^\eta\varphi(\epsilon2^j\cdot))\Vert_{L^1},
       \end{align*}
       where $N>d/2$.

  By the local elliptic estimate and \eqref{eq;decay;varphi}, for $|\eta|=1$, we have
\begin{align*}
    |\Delta^k\p^\eta\varphi(x)|
    \lesssim (1+|x|^2)^{-k-\frac32}.
\end{align*}
Hence, arguing as in the proof of Lemma \ref{le;decay}, when $2^j\epsilon\gtrsim1$, we obtain
\begin{align*}
    \Vert (1-\Delta)^N((\p^\eta\psi)\varphi(2^j\epsilon\cdot))\Vert_{L^1}
    &\lesssim (2^j\epsilon)^{-2},\\
    \Vert (1-\Delta)^N(\psi\p^\eta\varphi(2^j\epsilon\cdot))\Vert_{L^1}
    &\lesssim (2^j\epsilon)^{-3}.
\end{align*}
Therefore, for $\delta\in[-1,2]$ and $2^j\epsilon\gtrsim1$,
\begin{align*}
    \sum_{|\eta|=1}
    \Vert x^\eta\ml F^{-1}
    (\psi(2^{-j}\cdot)\varphi(\epsilon\cdot))\Vert_{L^1}
    \lesssim
    2^{-j(\delta+1)}\epsilon^{-\delta}.
\end{align*}
       Hence, arguing as above,
       \begin{align*}
           \Vert x^\eta \ml F^{-1}(\psi(2^{-j}\cdot)\varphi(\epsilon\cdot))\Vert_{L^1}
           \lesssim 2^{-j}+\epsilon(2^j\epsilon)
           \lesssim 2^{-j}.
       \end{align*}
       Since $2^j\epsilon\lesssim1$ and $\delta\geq-1$, this yields
       \begin{align*}
           \Vert x^\eta \ml F^{-1}(\psi(2^{-j}\cdot)\varphi(\epsilon\cdot))\Vert_{L^1}
           \lesssim
           2^{-j}(2^j\epsilon)^{-\delta}.
       \end{align*}
       Consequently, for all $j\geq0$,
       \begin{align}\label{eq;commutator;kernel;varphi}
           \sum_{\eta\in\m N^d:|\eta|=1}
           \Vert x^\eta \ml F^{-1}(\psi(2^{-j}\cdot)\varphi(\epsilon\cdot))\Vert_{L^1}
           \lesssim
           2^{-j(\delta+1)}\epsilon^{-\delta}.
       \end{align}

       Using the definition of the paraproduct, we write
       \begin{align*}
           \varphi(\epsilon D)(u\prec v)-u\prec\varphi(\epsilon D)v
           =
           \sum_{j\geq0}
           [(\psi(2^{-j}\cdot)\varphi(\epsilon\cdot))(D),S_{j-1}u]\Delta_jv,
       \end{align*}
       up to the low-frequency term, which can be treated in the same way. By \eqref{eq;commutator;kernel;varphi}, we have
       \begin{align*}
           &\Vert [(\psi(2^{-j}\cdot)\varphi(\epsilon\cdot))(D),S_{j-1}u]\Delta_jv\Vert_{L^\infty}\\
           &\quad\lesssim
           \sum_{\eta\in\m N^d:|\eta|=1}
           \Vert x^\eta \ml F^{-1}(\psi(2^{-j}\cdot)\varphi(\epsilon\cdot))\Vert_{L^1}
           \Vert\p^\eta S_{j-1}u\Vert_{L^\infty}
           \Vert\Delta_jv\Vert_{L^\infty}\\
           &\quad\lesssim
           2^{-j(1+\delta)}\epsilon^{-\delta}
           2^{j(1-\alpha)}2^{-j\beta}
           \Vert u\Vert_\alpha\Vert v\Vert_\beta\\
           &\quad\lesssim
           2^{-j(\alpha+\beta+\delta)}\epsilon^{-\delta}
           \Vert u\Vert_\alpha\Vert v\Vert_\beta.
       \end{align*}
       By the standard Littlewood--Paley summation, we obtain
       \begin{align*}
           \Vert \varphi(\epsilon D)(u\prec v)-u\prec\varphi(\epsilon D)v\Vert_{\alpha+\beta+\delta}
           \lesssim
           \epsilon^{-\delta}\Vert u\Vert_{\ml C^\alpha}\Vert v\Vert_{\ml C^\beta}.
       \end{align*}

       The estimate for $\phi$ follows from the same argument. Indeed, by \eqref{eq;decay;phi}, the corresponding high-frequency estimate has decay of order four, while the same low-frequency argument applies. Hence, for $\delta\in[-1,4]$,
       \begin{align*}
           \sum_{\eta\in\m N^d:|\eta|=1}
           \Vert x^\eta \ml F^{-1}(\psi(2^{-j}\cdot)\phi(\epsilon\cdot))\Vert_{L^1}
           \lesssim
           2^{-j(\delta+1)}\epsilon^{-\delta}.
       \end{align*}
       Repeating the above argument gives
       \begin{align*}
           \Vert \phi(\epsilon D)(u\prec v)-u\prec\phi(\epsilon D)v\Vert_{\alpha+\beta+\delta}
           \lesssim
           \epsilon^{-\delta}\Vert u\Vert_{\ml C^\alpha}\Vert v\Vert_{\ml C^\beta},
       \end{align*}
       which completes the proof.
    \end{proof}
\end{lemma}
\begin{lemma}\label{le;commutator;heat flow}
Let $\nu\in(0,1)$, $0<T\leq 1$, $\alpha\in(0,1)$, $\beta\in\mathbb R$, 
and let
\[
u\in \mathcal L_T^{\alpha,\nu},\qquad v\in C_T\mathcal C^\beta .
\]
Then, for any $\delta\in[-1,2)$, we have
\begin{equation}\label{eq;commutator;heat flow;1}
\|I(u\prec v)-u\prec Iv\|_{C_T\mathcal C^{\alpha+\beta+\delta}}
\lesssim
\|u\|_{\mathcal L_T^{\alpha,\nu}}
\|v\|_{C_T\mathcal C^\beta}.
\end{equation}
\end{lemma}

\begin{proof}
We decompose
\[
I(u\prec v)-u\prec Iv=I_1(t)+I_2(t),
\]
where
\begin{align*}
I_1(t)
&=\int_0^t
V(t-\tau)(u\prec v)(\tau)
-u(\tau)\prec V(t-\tau)v(\tau)\,\mathrm d\tau,\\
I_2(t)
&=\int_0^t
(u(\tau)-u(t))\prec V(t-\tau)v(\tau)\,\mathrm d\tau .
\end{align*}

For $I_2$, using Lemma~\ref{le;est;fractional heat flow;linear}, we obtain, for
$\delta<2$,
\begin{align*}
\|I_2(t)\|_{\mathcal C^{\alpha+\beta+\delta}}
&\lesssim
\int_0^t
\|u(t)-u(s)\|_{L^\infty}
\|V(t-s)v(s)\|_{\mathcal C^{\alpha+\beta+\delta}}
\,\mathrm ds\\
&\lesssim
\|u\|_{C_T^{\nu\alpha/2}L^\infty}
\|v\|_{C_T\mathcal C^\beta}
\int_0^t
(t-s)^{\frac{\nu\alpha}{2}}
(t-s)^{\nu-1-\frac{\nu}{2}(\alpha+\delta)}
\,\mathrm ds\\
&\lesssim
T^{\nu(1-\delta/2)}
\|u\|_{C_T^{\nu\alpha/2}L^\infty}
\|v\|_{C_T\mathcal C^\beta}.
\end{align*}

For $I_1$, by Lemma~\ref{le;commutator;decay} and the above representation
of $V(t)$, we have
\begin{align*}
\|I_1(t)\|_{\mathcal C^{\alpha+\beta+\delta}}
&\lesssim
\int_0^t
(t-s)^{\nu-1-\frac{\nu\delta}{2}}
\|u(s)\|_{\mathcal C^\alpha}
\|v(s)\|_{\mathcal C^\beta}
\,\mathrm ds\\
&\lesssim
T^{\nu(1-\delta/2)}
\|u\|_{C_T\mathcal C^\alpha}
\|v\|_{C_T\mathcal C^\beta}.
\end{align*}

Combining the estimates for $I_1$ and $I_2$, and using
$\mathcal L_T^{\alpha,\nu}
=C_T\mathcal C^\alpha\cap C_T^{\nu\alpha/2}L^\infty$, we conclude that
\[
\|I(u\prec v)-u\prec Iv\|_{C_T\mathcal C^{\alpha+\beta+\delta}}
\lesssim
\|u\|_{\mathcal L_T^{\alpha,\nu}}
\|v\|_{C_T\mathcal C^\beta}.
\]
This completes the proof.
\end{proof}
We also discuss the H\"older time regularity of the commutator. The above argument cannot directly yield the desired time regularity, since the time regularity of $u$ also needs to be taken into account.

\begin{lemma}
    Moreover, if $\alpha+\beta\in(-2,0)$, then for any
    \[
        \sigma\leq\min\left(\nu+\frac{\nu(\alpha+\beta)}{2},\frac{\nu\alpha}{2}\right),
    \]
    we have the time regularity estimate
    \begin{equation}\label{eq;commutator;heat flow;2}
        \Vert I(u\prec v)-u\prec Iv\Vert_{C_T^{\sigma} L^\infty}
        \lesssim \Vert u\Vert_{\ml L_T^{\alpha,\nu}}\Vert v\Vert_{C_T\ml C^\beta}.
    \end{equation}
    \begin{proof}
        For \eqref{eq;commutator;heat flow;2}, we decompose
        \begin{align*}
            I(u\prec v)(t)-u(t)\prec I(v)(t)=I_1(t)+I_2(t),
        \end{align*}
        where
        \begin{align*}
            I_1(t)&=\int_0^t\big[V(t-\tau)(u\prec v)(\tau)
            -u(\tau)\prec V(t-\tau)v(\tau)\big]\mathrm d\tau,\\
            I_2(t)&=\int_0^t(u(\tau)-u(t))\prec V(t-\tau)v(\tau)\mathrm d\tau.
        \end{align*}
        For $\Vert I_1(t)-I_1(s)\Vert_{L^\infty}$, we have
        \begin{align*}
            \Vert I_1(t)-I_1(s)\Vert_{L^\infty}
            \lesssim& \Vert \int_0^s[V(t-\tau)-V(s-\tau)](u\prec v)(\tau)
            -u(\tau)\prec[V(t-\tau)-V(s-\tau)]v(\tau)\mathrm d\tau\Vert_{L^\infty}\\
            &+\Vert \int_s^t V(t-\tau)(u\prec v)(\tau)
            -u(\tau)\prec V(t-\tau)v(\tau)\mathrm d\tau\Vert_{L^\infty}\\
            \lesssim&\int_0^s\int_{s-\tau}^{t-\tau}
            \Vert V'(\eta)(u\prec v)(\tau)
            -u(\tau)\prec V'(\eta)v(\tau)\Vert_{L^\infty}\mathrm d\eta\mathrm d\tau\\
            &+\int_s^t\Vert V(t-\tau)(u\prec v)(\tau)
            -u(\tau)\prec V(t-\tau)v(\tau)\Vert_{L^\infty}\mathrm d\tau.
        \end{align*}
        For the first integral, by Lemma \ref{le;commutator;decay} and Lemma \ref{le;fractional calculus}, since $\alpha+\beta\in(-2,0)$, for any
        $\sigma\in(0,\nu+\frac{\nu(\alpha+\beta)}2)$, we can choose $\epsilon>0$ sufficiently small such that
        \begin{align*}
            &\int_0^s\int_{s-\tau}^{t-\tau}
            \Vert V'(\eta)(u\prec v)(\tau)
            -u(\tau)\prec V'(\eta)v(\tau)\Vert_{L^\infty}\mathrm d\eta\mathrm d\tau\\
            \lesssim& \int_0^s\int_{s-\tau}^{t-\tau}
            \eta^{\nu-2}
            \Vert \phi_{\theta,1,1}(\eta^{\nu/2}D)(u\prec v)(\tau)
            -u(\tau)\prec\phi_{\theta,1,1}(\eta^{\nu/2}D)v(\tau)
            \Vert_{L^\infty}\mathrm d\eta\mathrm d\tau\\
            \lesssim&
            \Vert u\Vert_{C_T\ml C^\alpha}\Vert v\Vert_{C_T\ml C^\beta}
            \int_0^s\int_{s-\tau}^{t-\tau}
            \eta^{\nu-2+\frac{\nu(\alpha+\beta)}2}
            \mathrm d\eta\mathrm d\tau.
        \end{align*}
        Using Lemma \ref{le;fractional calculus}, since $\sigma\leq\nu+\frac{\nu(\alpha+\beta)}{2}$ annd $\alpha+\beta\in(-2,0)$,
        we obtain
        \begin{align*}
            &\int_0^s\int_{s-\tau}^{t-\tau}
            \Vert V'(\eta)(u\prec v)(\tau)
            -u(\tau)\prec V'(\eta)v(\tau)\Vert_{L^\infty}\mathrm d\eta\mathrm d\tau\\
            &\qquad\lesssim
            T^{\nu+\frac{\nu(\alpha+\beta)}2-\sigma}
            |t-s|^\sigma
            \Vert u\Vert_{C_T\ml C^\alpha}\Vert v\Vert_{C_T\ml C^\beta}.
        \end{align*}
        The second integral can be estimated similarly by Lemma \ref{le;commutator;decay}, which yields
        \begin{align*}
            &\int_s^t\Vert V(t-\tau)(u\prec v)(\tau)
            -u(\tau)\prec V(t-\tau)v(\tau)\Vert_{L^\infty}\mathrm d\tau\\
            &\qquad\lesssim
            |t-s|^{\nu+\frac{\nu(\alpha+\beta-\epsilon)}2}
            \Vert u\Vert_{C_T\ml C^\alpha}\Vert v\Vert_{C_T\ml C^\beta}.
        \end{align*}
        Therefore,
        \begin{equation}
            \Vert I_1\Vert_{C_T^\sigma L^\infty}
            \lesssim \Vert u\Vert_{C_T\ml C^\alpha}
            \Vert v\Vert_{C_T\ml C^\beta}.
        \end{equation}

        For $I_2(t)$, we have
        \begin{align*}
            I_2(t)-I_2(s)
            =&\int_s^t (u(\tau)-u(t))\prec V(t-\tau)v(\tau)\mathrm d\tau\\
            &+\int_0^s(u(s)-u(t))\prec V(t-\tau)v(\tau)\mathrm d\tau\\
            &+\int_0^s(u(\tau)-u(s))\prec
            (V(t-\tau)-V(s-\tau))v(\tau)\mathrm d\tau.
        \end{align*}
        For the first integral, we have
        \begin{align*}
            &\Vert\int_s^t (u(\tau)-u(t))\prec
            V(t-\tau)v(\tau)\mathrm d\tau\Vert_{L^\infty}\\
            &\quad\lesssim
            \Vert u\Vert_{C_T^{\nu\alpha/2}L^\infty}
            \Vert v\Vert_{C_T\ml C^\beta}
            \int_s^t(t-\tau)^{\nu\alpha/2}
            (t-\tau)^{\nu-1+\nu\beta/2}\mathrm d\tau\\
            &\quad\lesssim
            |t-s|^{\nu+\frac{\nu(\alpha+\beta)}2}
            \Vert u\Vert_{C_T^{\nu\alpha/2}L^\infty}
            \Vert v\Vert_{C_T\ml C^\beta}\\
            &\quad\lesssim
            |t-s|^\sigma
            \Vert u\Vert_{\ml L_T^{\alpha,\nu}}
            \Vert v\Vert_{C_T\ml C^\beta}.
        \end{align*}
        For the second integral, we have
        \begin{align*}
            &\Vert \int_0^s (u(s)-u(t))\prec
            V(t-\tau)v(\tau)\mathrm d\tau\Vert_{L^\infty}\\
            &\quad\lesssim
            \Vert u\Vert_{C_T^{\nu\alpha/2}L^\infty}
            |t-s|^{\nu\alpha/2}
            \int_0^s\Vert V(t-\tau)v(\tau)\Vert_{\ml C^\epsilon}\mathrm d\tau\\
            &\quad\lesssim
            \Vert u\Vert_{C_T^{\nu\alpha/2}L^\infty}
            \Vert v\Vert_{C_T\ml C^\beta}
            |t-s|^{\nu\alpha/2}
            \int_0^s(t-\tau)^{\nu-1+\nu\beta/2}\mathrm d\tau\\
            &\quad\lesssim
            T^{\nu+\frac{\nu\beta}2}
            |t-s|^{\nu\alpha/2}
            \Vert u\Vert_{C_T^{\nu\alpha/2}L^\infty}
            \Vert v\Vert_{C_T\ml C^\beta}\\
            &\quad\lesssim
            |t-s|^\sigma
            \Vert u\Vert_{\ml L_T^{\alpha,\nu}}
            \Vert v\Vert_{C_T\ml C^\beta},
        \end{align*}
        where we used $\sigma\leq\nu\alpha/2$.

        For the third integral, we have
        \begin{align*}
            &\Vert \int_0^s(u(\tau)-u(s))\prec
            (V(t-\tau)-V(s-\tau))v(\tau)\mathrm d\tau\Vert_{L^\infty}\\
            &\quad\lesssim
            \int_0^s\int_{s-\tau}^{t-\tau}
            \Vert (u(\tau)-u(s))\prec V'(\eta)v(\tau)\Vert_{L^\infty}
            \mathrm d\eta\mathrm d\tau\\
            &\quad\lesssim
            \Vert u\Vert_{C_T^{\nu\alpha/2}L^\infty}
            \Vert v\Vert_{C_T\ml C^\beta}
            \int_0^s(s-\tau)^{\nu\alpha/2}
            \int_{s-\tau}^{t-\tau}
            \eta^{\nu-2+\nu(\beta)/2}\mathrm d\eta\mathrm d\tau\\
            &\quad\lesssim
            T^{\nu+\frac{\nu(\alpha+\beta)}2-\sigma}
            |t-s|^\sigma
            \Vert u\Vert_{C_T^{\nu\alpha/2}L^\infty}
            \Vert v\Vert_{C_T\ml C^\beta}.
        \end{align*}
        Here we used the same argument as in the estimate of $I_1(t)-I_1(s)$. Therefore,
        \begin{align*}
            \Vert I_2\Vert_{C_T^\sigma L^\infty}
            \lesssim
            \Vert u\Vert_{\ml L_T^{\alpha,\nu}}
            \Vert v\Vert_{C_T\ml C^\beta}.
        \end{align*}
        Combining the estimates for $I_1$ and $I_2$, we obtain
        \begin{align*}
            \Vert I(u\prec v)-u\prec Iv\Vert_{C_T^\sigma L^\infty}
            \lesssim
            \Vert u\Vert_{\ml L_T^{\alpha,\nu}}
            \Vert v\Vert_{C_T\ml C^\beta},
        \end{align*}
        which completes the proof.
    \end{proof}
\end{lemma}
 	\bibliography{ref}

\newcommand{\etalchar}[1]{$^{#1}$}
\begin{thebibliography}{MBSB02}

\bibitem[BCD11]{bahouri2011fourier}
Hajer Bahouri, Jean-Yves Chemin, and Rapha{\"e}l Danchin.
\newblock {\em Fourier analysis and nonlinear partial differential equations},
  volume 343.
\newblock Springer, 2011.

\bibitem[Cap66]{caputo1966linear}
Michele Caputo.
\newblock Linear models of dissipation whose q is almost frequency independent.
\newblock {\em Annals of Geophysics}, 19(4):383--393, 1966.

\bibitem[Cap67]{caputo1967linear}
Michele Caputo.
\newblock Linear models of dissipation whose q is almost frequency
  independent—ii.
\newblock {\em Geophysical journal international}, 13(5):529--539, 1967.

\bibitem[CC18]{catellier2018paracontrolled}
R{\'e}mi Catellier and Khalil Chouk.
\newblock Paracontrolled distributions and the 3-dimensional stochastic
  quantization equation.
\newblock {\em The Annals of Probability}, 46(5):2621--2679, 2018.

\bibitem[GIP15]{gubinelli2015paracontrolled}
Massimiliano Gubinelli, Peter Imkeller, and Nicolas Perkowski.
\newblock Paracontrolled distributions and singular {PDEs}.
\newblock In {\em Forum of Mathematics, Pi}, volume~3, page~e6. Cambridge
  University Press, 2015.

\bibitem[GKO23]{gubinelli2023paracontrolled}
Massimiliano Gubinelli, Herbert Koch, and Tadahiro Oh.
\newblock Paracontrolled approach to the three-dimensional stochastic nonlinear
  wave equation with quadratic nonlinearity.
\newblock {\em Journal of the European Mathematical Society}, 26(3):817--874,
  2023.

\bibitem[GP17]{gubinelli2017kpz}
Massimiliano Gubinelli and Nicolas Perkowski.
\newblock Kpz reloaded.
\newblock {\em Communications in Mathematical Physics}, 349:165--269, 2017.

\bibitem[Hai13]{hairer2013solving}
Martin Hairer.
\newblock Solving the kpz equation.
\newblock {\em Annals of mathematics}, pages 559--664, 2013.

\bibitem[Hai14]{hairerTheoryRegularityStructures2014}
Martin Hairer.
\newblock A theory of regularity structures.
\newblock {\em Inventiones mathematicae}, 198(2):269--504, 2014.

\bibitem[J{\etalchar{+}}21]{jin2021fractional}
Bangti Jin et~al.
\newblock {\em Fractional differential equations}, volume 206.
\newblock Springer, 2021.

\bibitem[LL22]{li2022finite}
Changpin Li and Zhiqiang Li.
\newblock The finite-time blow-up for semilinear fractional diffusion equations
  with time $\psi$-caputo derivative.
\newblock {\em Journal of Nonlinear Science}, 32(6):82, 2022.

\bibitem[LRdS18]{liu2018quasi}
Wei Liu, Michael R{\"o}ckner, and Jos{\'e}~Lu{\'\i}s da~Silva.
\newblock Quasi-linear (stochastic) partial differential equations with
  time-fractional derivatives.
\newblock {\em SIAM Journal on Mathematical Analysis}, 50(3):2588--2607, 2018.

\bibitem[LZR16]{li2016linear}
Dongfang Li, Chengjian Zhang, and Maohua Ran.
\newblock A linear finite difference scheme for generalized time fractional
  burgers equation.
\newblock {\em Applied Mathematical Modelling}, 40(11-12):6069--6081, 2016.

\bibitem[Mai22]{mainardi2022fractional}
Francesco Mainardi.
\newblock {\em Fractional calculus and waves in linear viscoelasticity}.
\newblock World Scientific Publishing Company, 2022.

\bibitem[MBSB02]{meerschaert2002stochastic}
Mark~M Meerschaert, David~A Benson, Hans-Peter Scheffler, and Boris Baeumer.
\newblock Stochastic solution of space-time fractional diffusion equations.
\newblock {\em Physical Review E}, 65(4):041103, 2002.

\bibitem[MK00]{metzler2000random}
Ralf Metzler and Joseph Klafter.
\newblock The random walk's guide to anomalous diffusion: a fractional dynamics
  approach.
\newblock {\em Physics reports}, 339(1):1--77, 2000.

\bibitem[MLP07]{mainardi2007fundamental}
Francesco Mainardi, Yuri Luchko, and Gianni Pagnini.
\newblock The fundamental solution of the space-time fractional diffusion
  equation.
\newblock {\em arXiv preprint cond-mat/0702419}, 2007.

\bibitem[Mom06]{momani2006non}
Shaher Momani.
\newblock Non-perturbative analytical solutions of the space-and
  time-fractional burgers equations.
\newblock {\em Chaos, Solitons \& Fractals}, 28(4):930--937, 2006.

\bibitem[MS13]{meerschaert2013inverse}
Mark~M Meerschaert and Peter Straka.
\newblock Inverse stable subordinators.
\newblock {\em Mathematical modelling of natural phenomena}, 8(2):1--16, 2013.

\bibitem[Oks13]{oksendal2013stochastic}
Bernt Oksendal.
\newblock {\em Stochastic differential equations: an introduction with
  applications}.
\newblock Springer Science \& Business Media, 2013.

\bibitem[Pr{\"u}13]{pruss2013evolutionary}
Jan Pr{\"u}ss.
\newblock {\em Evolutionary integral equations and applications}.
\newblock Birkh{\"a}user, 2013.

\bibitem[Tua20]{tuan2020asymptotic}
HT~Tuan.
\newblock On the asymptotic behavior of solutions to time-fractional elliptic
  equations driven a multiplicative white noise.
\newblock {\em arXiv preprint arXiv:2002.06054}, 2020.

\bibitem[YK17]{yokucs2017numerical}
As{\i}f Yoku{\c{s}} and Do{\u{g}}an Kaya.
\newblock Numerical and exact solutions for time fractional burgers' equation.
\newblock {\em Journal of Nonlinear Sciences \& Applications (JNSA)},
  10(7):3419, 2017.

\bibitem[Zac09]{zacher2009weak}
Rico Zacher.
\newblock Weak solutions of abstract evolutionary integro-differential
  equations in hilbert spaces.
\newblock {\em Funkcialaj Ekvacioj}, 52(1):1--18, 2009.

\end{thebibliography}
    \bibliographystyle{alpha}
	\end{document}